\documentclass[12pt]{amsart}
\usepackage{euscript,epsf,amssymb,xr}
\usepackage{color}
\usepackage{graphicx}

\let\oldmarginpar\marginpar
\renewcommand\marginpar[1]
{\oldmarginpar{\tiny\bf \begin{flushleft} #1 \end{flushleft}}}

\input xy
\xyoption{all}

\newtheorem{theorem}{\bf Theorem}[section]
\newtheorem{lemma}[theorem]{\bf Lemma}

\newtheorem{remark}[theorem]{\bf Remark}
\newtheorem{prop}[theorem]{\bf Proposition}
\newtheorem{corollary}[theorem]{\bf Corollary}
\newtheorem{definition}[theorem]{\bf Definition}
\newtheorem{example}[theorem]{\bf Example}
\newtheorem{question}[theorem]{\bf Question}

\newcommand{\BB}{{\mathbb B}}
\newcommand{\CC}{{\mathbb C}}

\newcommand{\DD}{{\mathbb D}}
\newcommand{\EE}{{\mathbb E}}
\newcommand{\GG}{{\mathbb G}}

\newcommand{\NN}{{\mathbb N}}
\newcommand{\PP}{{\mathbb P}}
\newcommand{\QQ}{{\mathbb Q}}
\newcommand{\RR}{{\mathbb R}}
\newcommand{\TT}{{\mathbb T}}

\newcommand{\ZZ}{{\mathbb Z}}

\newcommand{\blie}{{\mathfrak b}}
\newcommand{\glie}{{\mathfrak g}}
\newcommand{\hlie}{{\mathfrak h}}

\newcommand{\plie}{{\mathfrak p}}

\newcommand{\tlie}{{\mathfrak t}}

\newcommand{\Ch}{\operatorname{Ch}}

\newcommand{\Coeff}{\operatorname{Coeff}}

\newcommand{\Diff}{\operatorname{Diff}}

\newcommand{\GL}{\operatorname{GL}}

\newcommand{\Ham}{\operatorname{Ham}}
\newcommand{\Hom}{\operatorname{Hom}}
\newcommand{\Homeo}{\operatorname{Homeo}}

\newcommand{\PGL}{\operatorname{PGL}}
\newcommand{\PU}{\operatorname{PU}}

\newcommand{\SO}{\operatorname{SO}}
\newcommand{\Sp}{\operatorname{Sp}}

\newcommand{\SU}{\operatorname{SU}}
\newcommand{\Sym}{\operatorname{Sym}}

\newcommand{\U}{\operatorname{U}}

\renewcommand{\exp}{\operatorname{exp}}

\newcommand{\curly}{\mathcal}

\newcommand{\LLL}{{\curly L}}

\newcommand{\OOO}{{\curly O}}
\newcommand{\PPP}{{\curly P}}

\newcommand{\ggreat}{>\kern-.7ex>}
\newcommand{\ssmall}{<\kern-.7ex<}
\newcommand{\qu}{/\kern-.7ex/}
\newcommand{\exh}{\to\kern-1.8ex\to}

\newcommand{\VP}{{\curly V}\kern-0.9ex\PPP}

\newcommand{\imag}{{\mathbf i}}

\newcommand{\w}{\omega}

\newcommand{\ZII}{Z_{I\kern-.3ex I}}
\newcommand{\ZIII}{Z_{I\kern-.3ex I\kern-.3ex I}}

\date{\today}

\title[Nontrivial bundles of coadjoint orbits over $S^2$]
{Nontrivial bundles of coadjoint orbits over $S^2$}

\author[D. Mart\'{\i}nez Torres]{David Mart\'{\i}nez Torres}
\address{Departamento de Matem\'atica, PUC-Rio, R. Mq. S. Vicente 225, Rio de Janeiro 22451-900, Brazil}
\email{dfmtorres@gmail.com}

\author[I. Mundet i Riera]{Ignasi Mundet i Riera}
\address{Facultat de Matem\`atiques i Inform\`atica, Universitat de Barcelona, Gran Via de les 
Corts Catalanes 585, 08007 Barcelona, Spain}
\email{ignasi.mundet@gmail.com}

\begin{document}

\begin{abstract}
Let $G$ be a compact connected semisimple Lie group with Lie algebra $\glie$.
Let $\OOO\subset\glie^*$ be a coadjoint orbit. The  action of $G$
on $\OOO$ induces a morphism $\rho:G\to \Homeo(\OOO)$. We prove that
the induced map $\pi_1(\rho):\pi_1(G)\to\pi_1(\Homeo(\OOO))$ is injective.
This strengthens a theorem of McDuff and Tolman (conjectured by Weinstein in 1989)
according to which
the analogous map $G\to\Ham(\OOO)$ is injective on fundamental groups, where
$\Ham(\OOO)$ is the group of Hamiltonian diffeomorphisms of the standard
symplectic structure on $\OOO$. To prove our theorem we associate to every
nontrivial element of $\pi_1(G)$ a bundle over $S^2$ with fiber $\OOO$, using
the standard patching construction. We then prove that the resulting bundle
is topologically nontrivial by studying its cohomology. For this, we prove that
it suffices to consider the case in which $G$ is simple and $\OOO$ is a minimal orbit, and then we prove 
our result for simple $G$ and minimal orbit $\OOO$ in a case by case analysis; the proof for the exceptional groups 
$E_6$ and $E_7$ relies on computer calculations, while all the other ones are
addressed by hand. A basic tool in some of our computations is a generalization 
of Chevalley's formula to bundles of coadjoint orbits over $S^2$ that we prove in
this paper.
\end{abstract}

\maketitle

\tableofcontents

\section{Introduction}

The action of a Lie group $G$ on itself by conjugation induces a canonical linear action on the dual of
its algebra,
the coadjoint representation. A coadjoint orbit $\OOO\subset \glie^*$ has a canonical symplectic structure
and the coadjoint action of $G$ on $(\OOO,\w)$ is Hamiltonian. This implies that
 the morphism  $G\rightarrow \Diff(\OOO)$  associated to the coadjoint action factors through the group
of Hamiltonian diffeomorphisms $\Ham(\OOO,\w)$:

\begin{equation}\label{eq:diag-manifolds}
\xymatrix{G\ar[rr] \ar[rd] & & \Diff(\OOO)  \\
& \Ham (\OOO,\w) \ar[ru] }
\end{equation}

Coadjoint orbits of compact, connected Lie groups are among the most basic Hamiltonian spaces, and for this reason the study of the properties of
the morphism
$G\rightarrow \Ham(\OOO,\w)$ is
a fundamental problem in symplectic geometry. In this respect,  Weinstein \cite[Section 4]{W} asked the following:

\begin{question}\label{q:ham} When $G$ is a compact, connected, semisimple Lie group, is the homomorphism of fundamental groups
$\rho_{H}\colon \pi_1(G)\rightarrow \pi_1(\Ham (\OOO,\w))$ induced by the coadjoint action
 injective?
\end{question}
Observe that such a Lie group $G$ has finite fundamental group --and therefore made of torsion elements--
whereas $\Ham (\OOO,\w)$ is an  infinite dimensional CW complex (we use the compact open topology). Thus
the injectivity of $\rho_{H}$ would be a remarkable result, since one might expect to have `enough space' in  $\Ham (\OOO,\w)$
to cap off the torsion loops coming from $G$.

In \cite[Theorem 6]{V}  Vi\~{n}a gave a partial answer to Weinstein's question by proving the injectivity of $\rho_{H}$ for some quantizable
coadjoint orbits of $\PU(n)$. Question \ref{q:ham} was settled in full generality by McDuff and  Tolman:

\begin{theorem}[\cite{MT}]\label{thm:injh} If $(\OOO,\w)$ is a coadjoint orbit of a compact, connected, semisimple Lie group $G$, then
the homomorphism $\rho_{H}\colon \pi_1(G)\rightarrow \pi_1(\Ham (\OOO,\w))$ is injective.
\end{theorem}

 A different proof of Theorem \ref{thm:injh} based on the injectivity of the Seidel morphism was given in \cite[Theorem 2]{CMP}.

It is natural to ask whether Theorem \ref{thm:injh} reflects a `symplectic property', or a property of different nature. More specifically,
the commutative diagram (\ref{eq:diag-manifolds}) induces a commutative diagram of morphisms of fundamental groups:
\begin{equation}\label{eq:diag-fundgroups}
\xymatrix{\pi_1(G)\ar_{\rho}[rr] \ar_{\rho_H}[rd] & & \pi_1(\Diff (\OOO)) \\
& \pi_1(\Ham (\OOO,\w)) \ar[ru] }
\end{equation}
and one asks the following:

\begin{question}\label{q:diff} When $G$ is a compact, connected, semisimple Lie group, is the morphism
$\rho\colon \pi_1(G)\rightarrow \pi_1(\Diff (\OOO))$
 injective?
\end{question}

As remarked  in \cite[Section 4]{W}, for $G=\SO(3)$ and $\OOO\cong S^2$, a well-known result of Smale asserts that
$\SO(3)\rightarrow \Diff(S^2)$ is
a homotopy equivalence, so in particular $\rho$ is injective on fundamental groups.

In this paper we address Question \ref{q:diff} and prove that the injectivity of $\rho$ for $G=\SO(3)$ and is not an isolated phenomenon.

\begin{theorem}\label{thm:injdiff} If $\OOO$ is a coadjoint orbit of a compact, connected, semisimple Lie group $G$, then the
homomorphism  $\rho\colon \pi_1(G)\rightarrow \pi_1(\Diff (\OOO))$ is injective.
\end{theorem}

Of course, since the diagram (\ref{eq:diag-fundgroups}) is commutative, Theorem \ref{thm:injh} becomes
a straightforward consequence of Theorem \ref{thm:injdiff},  so the former result is by
no means a manifestation of a `symplectic property'.

Theorem \ref{thm:injdiff} provides relevant information on the topology
of the identity component of the Frechet Lie group $\Diff (\OOO)$.
Recall that a homotopy equivalence between infinite dimensional
Frechet manifolds is homotopic to a homeomorphism \cite[Theorem 0.1]{BH}.
The morphism $G\to \Diff(\OOO)$ is continuous in the Frechet topology, which is finer than the compact open one; hence
the injectivity of $\rho$ holds as well for the Frechet topology.
So we can look at Theorem \ref{thm:injdiff} as the first steps in understanding the
homotopy type (or, equivalently by \cite{BH}, the homeomorphism type)
of the identity component of $\Diff(\OOO)$. From this perspective,
Theorem \ref{thm:injdiff} is complemented (in the case $G=\SU(n)$) by the
 work of K\c{e}dra and McDuff \cite[Proposition 4.8]{KM}, where the injectivity of
\[\rho\otimes_\ZZ \mathrm{Id}\colon \pi_k(\SU(n))\otimes_\ZZ \QQ\rightarrow \pi_k(\Diff(\OOO))\otimes_\ZZ \QQ\]
is proved for  all $k\geq 1$.

 There are classical results describing torsion elements in the groups $\pi_{i}(\Diff(S^{n}))$, for many pairs $(i,n)$ \cite{ABK}. Our methods
are radically different from those in  \cite{ABK}, since they rely heavily on coadjoint orbits having an integral cohomology ring with very rich structure.

To be more precise about our methods, to prove our main result we will look at the fundamental group of $\Diff(\OOO)$ from a geometric perspective via
 the classical clutching construction. To any loop $\gamma:S^1\to \Diff(\OOO)$
based at the identity, the clutching construction assigns
\begin{itemize}
 \item a bundle with fiber $\OOO$ \[\pi\colon Y_{\gamma,\OOO}\rightarrow S^2,\]
\item  a trivialization of the fiber over a point $p\in S^2$,
\[i\colon \OOO\overset {\cong}{\longrightarrow} \pi^{-1}(p),\]
\end{itemize}
 such that   $\gamma,\gamma'$ are homotopic if and only if there exists a bundle diffeomorphism
\begin{equation}\label{eq:fibrationisom}
\xymatrix{Y_{\gamma,\OOO}\ar[rr]^{\psi} \ar[rd]_{\pi} & & Y_{\gamma',\OOO}  \ar[ld]^{\pi'}\\
&  S^2}
\end{equation}
compatible with the fiber trivializations, i.e.  $\psi\circ i=i'$, where $i$ and $i'$ are the
trivializations over $p$ corresponding to $\gamma$ and $\gamma'$ respectively. We will refer to such $\psi$ as a compatible
diffeomorphism.

Recall that for any choice of base point we can canonically identify the homotopy classes of based loops in the identity component of $\Diff(\OOO)$ with the homotopy classes of free loops in the same space, because
$\Diff(\OOO)$ is a topological group. Hence 
we will often ignore the requirement that the loops in our discussion are based at the identity, and only assume that they take values in the identity component.

In this more geometric setting, Theorem \ref{thm:injdiff}  is equivalent to the following:

\begin{theorem}\label{thm:injdiffbun} Let $\OOO$ be a coadjoint orbit of a compact, connected, semisimple Lie group $G$,
let $\gamma,\gamma':S^1\to G$, and let $(Y_{\gamma,\OOO},i),(Y_{\gamma',\OOO},i')$ be the bundles obtained by the clutching
construction.

There exists a compatible diffeomorphism
\[
\xymatrix{Y_{\gamma,\OOO}\ar[rr]^{\psi} \ar[rd] & & Y_{\gamma',\OOO} \ar[ld] \\
& S^2  }
\]
if and only if
$\gamma$ and $\gamma'$ are homotopic loops in $G$.
\end{theorem}

Of course, the only nontrivial part of Theorem \ref{thm:injdiffbun} is the `only if'.

The proof of Theorem \ref{thm:injdiff} will rely on the analysis of how the integral cohomology rings of the bundles $Y_{\gamma,\OOO}$ behave
under compatible isomorphisms, which are the cohomological counterparts of compatible diffeomorphisms.
So in fact our strategy will not just prove Theorem \ref{thm:injdiff}, but also that
$G\rightarrow \Homeo(\OOO)$ induces a monomorphism on fundamental
groups.

In many situations it will be important to look at coadjoint orbits from an algebro-geometric
perspective. If $\mathbb{G}$ denotes the complexification of $G$, then a coadjoint orbit of $G$ is diffeomorphic
to the homogeneous space $\mathbb {G}/\PP$, where $\PP$ is a parabolic subgroup. The most relevant
 example is when $G=\PU(n)$ and $\mathbb{G}=\PGL(n,\CC)$,  where a classical result on algebraic groups
 identifies $\PGL(n,\CC)/\PP$ with a variety of (partial) flags of $\CC^n$.
 In this context, Theorem \ref{thm:injdiffbun} can be rewritten as follows:

\begin{theorem}\label{thm:injdiffflags}Let $E$ and $E'$ be rank $n$ holomorphic vector bundles over $\PP^1$ of degrees $d,d'$,
let $r=(r_1,\dots,r_k)$ be a tuple of integers with $0<r_1<\cdots<r_k<n$, and let $Y_{d,r},Y_{d',r}$ be the bundles of
flags of type $r$ associated to $E$ and $E'$ respectively.

There exists a compatible diffeomorphism
\[
\xymatrix{Y_{d,r}\ar[rr]^{\psi} \ar[rd] & & Y_{d',r}  \ar[ld] \\
& \PP^1 }
\]
if and only if $d-d'$ is divisible by $n$.
\end{theorem}

This paper is organized as follows.  Section \ref{sec:char} sets up the strategy to prove Theorem \ref{thm:injdiff}. We start by recalling basic results on coadjoint orbits, the clutching
construction and the cohomology of homogeneous bundles. This is used to prove our main result of the section which
asserts that for any loops $\gamma,\gamma'$ in a compact, connected semisimple Lie group and for any coadjoint orbit $\OOO$ of the group,
there exist a unique compatible isomorphism  between the rational cohomology rings of the bundles
$Y_{\gamma,\OOO}$ and $Y_{\gamma',\OOO}$. We refer to this isomorphism as the {\em characteristic isomorphism}.
This immediately implies that $\gamma$ and $\gamma'$ will represent different homotopy classes in the group of diffeomorphisms of $\OOO$ if
the characteristic isomorphism does not map the {\em integral} cohomology of $Y_{\gamma,\OOO}$ into the integral cohomology of $Y_{\gamma,\OOO'}$,
i.e., if the characteristic isomorphism is not integral. We finish the section by showing that to prove our main Theorem is it enough
to show that the characteristic isomorphism is not integral in some key cases, namely, for the minimal coadjoint orbits of adjoint simple groups.

Section \ref{sec:genSchubert} develops the tools that we will use to perform the cohomological/homological
computations needed to show the non-integrality of the characteristic isomorphism for minimal orbits
of adjoint simple groups. First, we use an algebraic torus action on the bundles $Y_{\gamma,\OOO}$
to produce a generalized Bruhat/Schubert decomposition. The core of the section describes two ways to
compute pairings between certain Schubert varieties of $Y_{\gamma,\OOO}$ and arbitrary rational equivariant cohomology classes on $Y_{\gamma,\OOO}$: The first one is via a generalized Chevalley formula,
which is valid for any Schubert variety. The second
one is via explicit localization formulas, which we give for smooth Schubert varieties which come from appropriate Dynkin subdiagrams.

Section \ref{sec:unitary} contains the computation that show that the characteristic isomorphism is not integral
for the minimal coadjoint orbits of $\PU(n)$ and $\Sp(n)/Z$. These computations only require looking at degree two and four integral cohomology classes,
and are performed using the generalized Chevalley formula.

In sections \ref{s:orthogonal-I} and \ref{s:orthogonal-II} we prove the non-integrality
of the characteristic isomorphism for the minimal coadjoint orbits of the odd and even orthogonal groups, respectively.
These computations present more difficulties than the ones for the projective unitary and symplectic groups. On the one hand,
we need to use integral cohomology classes given by polynomials which do not have integral coefficients of weight variables. On the
other hand, we have to use the full power of our localization formulas to compute the corresponding integrals.

Finally, in Section \ref{s-exotic}  we prove the non-integrality
of the characteristic isomorphism for the minimal coadjoint orbits of $\mathrm{E}_6/Z$ and $\mathrm{E}_7/Z$. The new feature of this section
is the use of computer assisted calculations for two of the minimal orbits of $\mathrm{E}_7$. In fact,
we use computers to solve two different kinds of problems. The first one is the implementation of the generalized Chevalley formula
for appropriate Schubert varieties/Weyl group elements. The second one is the calculation of polynomials representing integral
cohomology classes of these two minimal coadjoint orbits.

{\bf Acknowledgments.} The authors are grateful to C. Arias Abad, M. Crainic, R. L. Fernandes, G. Granja, E. Miranda,
Y. Mitsumatsu, J. Mour\~ao and
C. Woodward for valuable suggestions.

\section{The characteristic isomorphism}\label{sec:char}
In this section we set up a strategy to prove Theorem \ref{thm:injdiff}. Briefly, the clutching construction transfers a problem of algebraic
topology in an infinite CW complex --- deciding whether $\gamma,\gamma':S^1\to G$ are homotopic inside $\Diff(\OOO)$ --- into a problem of differential topology of (finite dimensional) manifolds: deciding
whether there exists a compatible diffeomorphism  $\psi:Y_{\gamma,\OOO}\to Y_{\gamma',\OOO}$ (Theorem \ref{thm:injdiffbun}).  Such a diffeomorphism should
induce the corresponding compatible isomorphism $\psi^*$ between the cohomology rings (for any coefficient ring).
The central point of our strategy is an algebraic study of abstract compatible isomorphisms for rational coefficients/cohomology (see Definition \ref{def:compatible-isomorphism} below). This will possible because of the following:
\begin{enumerate}
 \item[(i)] for any loop $\gamma:S^1\to G$,  $Y_{\gamma,\OOO}$
  is a homogeneous bundle (i.e., the quotient
 of a principal $G$-bundle by a closed subgroup of $G$) and hence its rational cohomology ring can be handled using Chern-Weil theory;
 \item[(ii)] the hard Lefschetz property for $H^\bullet(\OOO;\QQ)$ (which follows from the existence of Kaehler structures on $\OOO$) implies the uniqueness of compatible isomorphisms in rational cohomology.
\end{enumerate}

\subsection{Projections between coadjoint orbits and the order in the set of orbits}
\label{ss:order-orbits}
Here we review some well known facts from a convenient perspective.
One can define a partial order on the set of coadjoint orbits of $G$ by
the condition that for any pair of orbits $\OOO,\OOO'$ we have $\OOO\geq\OOO'$
if and only if there exists a $G$-equivariant map $\OOO\to\OOO'$ (which is necessarily
smooth and surjective). Using this order one can speak of minimal and maximal orbits.
Maximal orbits are also called regular orbits.

Let us fix for the rest of the section and the following one a maximal torus
$$T\subset G$$
and let $\tlie$ be the Lie algebra of $T$.

All regular orbits are $G$-equivariantly diffeomorphic to $G/T$. Indeed suppose that
$\OOO$ is a coadjoint orbit and let $x\in\OOO$ be any element. Let $P$ be the centralizer
of $x$, so that the map $G\ni g\mapsto gx\in\OOO$ induces a $G$-equivariant diffeomorphism $G/P\to\OOO$. The isomorphism induced by the Killing pairing
$f:\glie^*\to\glie$ is $G$-equivariant, so the stabilizer of $x\in\glie^*$ is equal to
the stabilizer of $x'=f(x)\in\glie$. Now, $x'$ is contained in the Lie algebra of
some maximal torus $T'\subset G$, and clearly $T'$ is contained in the stabilizer of $x'$,
hence in that of $x$: $T'\subset P$. Since all maximal
tori are conjugate, it follows that there is a $G$-equivariant projection
$G/T\to\OOO$.

To identify $G/T$ with a coadjoint orbit, let $y'\in\tlie$ be an element in the interior of any Weyl chamber. Then the centralizer of $y'$
is equal to $T$. Letting $y=f^{-1}(y')\in\glie^*$ it follows that the coadjoint orbit through $y$
can be identified with $G/T$.

These notions can also be understood in classical holomorphic terms
(see e.g. \cite[Page 6]{Bott}). Any coadjoint orbit
admits $G$-equivariant Kaehler structures, with respect to which the $G$ action extends to
a transitive holomorphic action of the complexification $\GG$ of $G$. In general, any orbit is $G$-equivariantly diffeomorphic to $\GG/\PP$ for some parabolic subgroup $\PP\subset\GG$,
the maximal orbits are the ones for which $\PP$ can be taken to be a Borel subgroup,
and the minimal orbits are the ones for which $\PP$ is
a maximal parabolic subgroup. This perspective will be adopted in this paper in
Section \ref{sec:genSchubert} and the next ones.

Recall that the Weyl group of $T$ is by definition $W:=N(T)/T$.
There is a natural right action of $W$ on $G/T$.
Suppose that $\OOO$ is any coadjoint orbit. Let us take a $G$-equivariant map
$G/T\to\OOO$, let $x$ be the image of $eT$, and let $P\subset G$ be the stabilizer of
$x$. Then $T\subset P$, and we define the residual Weyl group of $P$ as
$$W_P:=W\cap P.$$
Note that $W_P$ depends on the choice of equivariant map $q:G/T\to\OOO$, so it is not intrinsically
associated to $\OOO$.
There are finitely many such maps $q$. Once one such $q$ has been chosen, $W_P$ can be identified
with the subgroup of $W$ consisting of those elements whose action on $G/T$ preserves each fiber of
$q$. It is a classical fact (see \cite[Theorem 5.5] {BGG}) that
$q^*:H^{\bullet}(\OOO;\ZZ)\to H^{\bullet}(G/T;\ZZ)$
is injective, and that the image of $q^*$ can be identified with
$H^{\bullet}(G/T;\ZZ)^{W_P}$.

\subsection{The clutching construction}
Let $F$ be a manifold and $\gamma:S^1\to \Diff(F)$ a loop based at the identity. The clutching construction associates to $\gamma$ a bundle $\pi:Y_{\gamma,F}\to S^2$ with fiber $F$ and a trivialization
$i:\pi^{-1}(p)\longrightarrow F$, for some fixed $p\in S^2$.

Consider two copies $\DD^2_+\times F$, $\DD^2_-\times F$,
of the the trivial bundle with fiber $F$ over the unit closed disk centered at the origin in $\CC$,
and glue them along the boundary using the diffeomorphism induced by $\gamma$:
\begin{equation}\label{eq:clutching2}
Y_{\gamma,F}:= (\DD^2_+\times F)\sqcup (\DD^2_-\times F)/(\exp^{2\pi it},x)\sim (\exp^{-2\pi it},\gamma(\exp^{2\pi it})x).
\end{equation}

The projections $\DD^2_{\pm}\times F\to \DD^2_{\pm}$ induce the fibration $\pi:Y_{\gamma,F}\to S^2$.
To fix ideas, let us take $p\in S^2$ to be the point corresponding to $0\in\DD^2_+$. Then the trivialization
$i:\pi^{-1}(p)\longrightarrow F$ is induced by the natural
inclusion $\DD^2_+\times F\hookrightarrow Y_{\gamma,F}$.

Our first aim is describing the effect of a compatible diffeomorphism
$\psi:Y_{\gamma,F}\to Y_{\gamma',F}$  at the cohomological level.
Consider (here and everywhere else in the paper) the orientation of $S^2$ extending the natural orientation on $\DD_+^2$ given by its inclusion in $\CC$.
For any coefficient ring $A$ we denote by $b\in H^2(Y_{\gamma,F};A)$
the pullback via $\pi$ of the orientation class in $H^2(S^2;A)$.

\begin{definition}
\label{def:compatible-isomorphism}
Let  $\gamma,\,\gamma':S^1\to G$ be loops
and let $(Y_{\gamma,F},i),\,(Y_{\gamma',F},i')$ be the bundles over $S^2$ obtained
by the clutching construction (\ref{eq:clutching2}).
We say that \[\phi\colon H^\bullet(Y_{\gamma',F};A)\rightarrow H^\bullet(Y_{\gamma,F};A)\] is a compatible isomorphism
 if $\phi$ is a ring isomorphism such that:
 \begin{enumerate}
\item $\phi(b)=b$;
\item $ i^*\circ \phi= {i'}^*$.
\end{enumerate}
\end{definition}

Of course, any compatible diffeomorphism $\psi\colon (Y_{\gamma,F},i)\rightarrow (Y'_{\gamma',F},i')$
induces a compatible isomorphism $H^\bullet(Y_{\gamma',F};A)\rightarrow H^\bullet(Y_{\gamma,F};A)$ for any coefficient ring $A$, in particular for integer and rational coefficients.

The main result of this section is the following:

\begin{theorem}\label{thm:uniqueness} Let $\OOO$ be a coadjoint orbit of a compact, connected, semisimple Lie group $G$, and let
$\gamma,\,\gamma':S^1\to G$  be loops.

There exists a unique compatible isomorphism in rational cohomology:
\begin{equation}\label{eq:charisom}
 \kappa\colon H^\bullet(Y_{\gamma',\OOO};\QQ)\rightarrow H^\bullet(Y_{\gamma,\OOO};\QQ).
 \end{equation}
\end{theorem}

Theorem \ref{thm:uniqueness} is a purely algebraic result which reduces the proof of Theorem \ref{thm:injdiffbun}  to analyzing
in which cases $\kappa$ comes form a compatible diffeomorphism, i.e when $\kappa=\psi^*$. In turn, the core of this analysis will amount to
studying how the integral cohomology behaves under $\kappa$ or, more precisely, to determine whether
$\kappa$ restricts to an isomorphism between the lattices defined by the integral cohomology (see Subsection \ref{ss:strategy}). Remarkably, $\kappa$ will turn out to be compatible with integral cohomology precisely whenever
$\gamma$ and $\gamma'$ are homotopic inside $G$.

The proof of Theorem \ref{thm:uniqueness} is given in Subsection \ref{ss:proof-thm:uniqueness}.
Before that we introduce some notation and prove some preliminary results.

\subsection{The rational cohomology ring of homogeneous bundles}

Let $G$ be a Lie group and let $X\to M$ be a principal $G$-bundle. Each (closed) subgroup $H\subset G$ defines a homogeneous bundle:
\[Y:=X/H\to M.\]
This is a bundle with fiber the homogeneous space $G/H$ and structural group $G$ (acting by left multiplication).

These are the kind of bundles that appear in our discussion:
\begin{lemma}\label{lem:homog}
Let $G$ be a compact connected group and let $\gamma:S^1\to G$ be a loop. If $\OOO$ is a coadjoint orbit of $G$,
then the  bundle $Y_\gamma$ coming from the clutching construction (\ref{eq:clutching2}) is a  homogeneous bundle over $S^2$.
\end{lemma}
\begin{proof}
Choose an element of $\OOO$ and let $P\subset G$ be its stabilizer.
The action of $G$ on $\OOO$ gives a diffeomorphism $\phi:\OOO\cong G/P$.
Let $X=Y_{\gamma,G}$ be the principal $G$-bundle obtained by applying the clutching
construction, using the action of $G$ on itself on the left. Then we can
identify $Y_{\gamma,\OOO}$ with $X/P$ using $\phi$.
\end{proof}

The rational cohomology ring of homogeneous bundles for compact, connected groups is well understood.
Let us discuss for simplicity the case in which $H$ is a maximal torus (for the more
general case of stabilizers of arbitrary elements in $\glie^*$ see the proof of Theorem \ref{thm:char-isom}).

Let $X\to M$ be a principal $G$-bundle and let $\pi\colon Y=X/T\to M$. Observe
that $X$ is both a principal $G$-bundle (over $M$) and a principal $T$-bundle (over $Y$). First, the cohomology ring of $Y$ is a module over the cohomology
ring of $M$. Second, to get `fiber classes' one considers the Chern-Weil homomorphism for both the principal
$T$-action and the principal $G$-action on $X$:  for
$f\in \mathrm{Sym}(\tlie^*)$ and $F\in  \mathrm{Sym}(\glie^*)^G$ one gets cohomology classes $c(f), \pi^*c(F)\in H^\bullet(Y;\RR)$.
 There will be redundancy among these fiber classes exactly when we consider $f$ and $F$ in correspondence under
  the Chevalley restriction isomorphism:
\[\Ch:\mathrm{Sym}(\tlie^*)^W\rightarrow \mathrm{Sym}(\glie^*)^G.\]
We have (see \cite[Theorem 6.8.4]{GS}):
\begin{theorem}
\label{thm:gs}
Let $G$ be a compact, connected Lie group and
let $\pi\colon Y\rightarrow M$ be a homogeneous bundle with fiber $G/T$. The Chern-Weil homomorphism induces an isomorphism
 \begin{equation}\label{eq:homog}
 \kappa\colon H^\bullet(M;\QQ)[\tlie]/I^M_+\rightarrow H^\bullet(Y;\QQ),
 \end{equation}
 where the  $I^M_+\subset H^\bullet(M;\QQ)[\tlie]$ is the ideal generated by polynomials of the form $f-c(\Ch(f))$, where  $f \in {\mathrm{Sym}_\QQ(\tlie^*)}^W$ vanishes at zero and
 $\Ch(f)\in {\mathrm{Sym}_\QQ(\glie^*)}^G$ is the image of $f$ by the Chevalley restriction isomorphism.
\end{theorem}

\begin{remark}\label{rem:Borel}
When $M$ is a point Theorem \ref{thm:gs} recovers the well-known Borel isomorphism \cite[Chapter 6]{B}, identifying the rational cohomology
of $G/T$ with the ring of coinvariants
$\QQ[\tlie]/I_+$. This ring is generated by homogeneous polynomials with rational coefficients, where the notion of rationality uses the weight/root
lattice of $\tlie^*$.
\end{remark}

\subsection{The characteristic isomorphism}

We start by proving an important consequence of the structure of the rational cohomology ring of homogeneous bundles with fiber a coadjoint orbit:

\begin{theorem}\label{thm:char-isom} Let $\OOO$ be a coadjoint orbit of a compact, connected, semisimple Lie group $G$, let
$\gamma:S^1\to G$  be a loop, $e:S^1\to G$  be the constant loop based at the identity.
Applying the isomorphism (\ref{eq:homog}) to the homogeneous bundle $Y_{\gamma,\OOO}\to S^2$ we get
a compatible isomorphism
 \begin{equation}\label{eq:characteristicisom}
\kappa\colon H^\bullet(Y_{e,\OOO};\QQ)\rightarrow H^\bullet(Y_{\gamma,\OOO};\QQ),
  \end{equation}
which we refer to as the characteristic isomorphism.
\end{theorem}

The use of the same symbol $\kappa$ for (\ref{eq:homog}) and for (\ref{eq:characteristicisom})
will hopefully lead to no confusion because in this paper only (\ref{eq:characteristicisom})
will be used systematically.

\begin{proof}
 We first give the proof for a regular orbit $\OOO\cong G/T$.

 Because $G$ is semisimple, the homogeneous $W$-invariant polynomials of smallest degree which
 vanish at zero are  multiples of the Killing
 form, and thus have degree 2. Therefore for any $W$-invariant polynomial $f$ the characteristic class $c(\Ch(f))\in H^\bullet(S^2;\QQ)$
 has degree at least four, and hence it is trivial. Thus the isomorphism (\ref{eq:homog}) can be rewritten:
 \[\kappa\colon H^\bullet(S^2;\QQ)\otimes (\QQ[\tlie]/I_+)\rightarrow H^\bullet(Y_\gamma;\QQ).\]
 By  Kunneth's theorem the left hand side is canonically isomorphic to $H^\bullet(Y_e;\QQ)$, so $\kappa$ is an isomorphism
 $H^\bullet(Y_e;\QQ)\rightarrow H^\bullet(Y_\gamma;\QQ)$. It is also clear by construction that $\kappa^*b=b$ and that
 $ i_\gamma^*\circ \kappa= {i_e}^*$, and therefore $\kappa$  is a compatible isomorphism, as asserted.

Now we consider the case of a general orbit $\OOO$. Let $\OOO'\cong G/T$ be a regular orbit.
By the general considerations in Subsection \ref{ss:order-orbits}
there exists a $G$-equivariant diffeomorphism $\OOO\cong G/P$, with $T\subset P$,
and a $G$-equivariant map $q:\OOO'\to\OOO$. The choice of $P$ allows to define the subgroup $W_P\subset W$.
The map $q$
induces a projection
 $Y_{\gamma,\OOO'}\rightarrow Y_{\gamma,\OOO}$, which on its turn
 induces a monomorphism on rational cohomology (see for example \cite[Lemma 4]{Tu})
 \[H^\bullet(Y_{\gamma,\OOO};\QQ)\rightarrow H^\bullet(Y_{\gamma,\OOO'};\QQ)\] with image $H^\bullet(Y_{\gamma,\OOO'};\QQ)^{W_P}$.

 Since the Chern-Weil homomorphism  is equivariant  with respect to  the actions of the Weyl group on  $\QQ[\tlie]/I_+$ and on $H^\bullet(Y_{\gamma,\OOO'};\QQ)$,
 it preserves  $W_P$-invariant subspaces, and therefore it restricts to a compatible isomorphism
 \[H^\bullet(Y_{e,\OOO};\QQ)\rightarrow H^\bullet(Y_{\gamma,\OOO};\QQ).\]
\end{proof}

\begin{remark}
\label{rmk:existence-comp-isom}
Clearly, for any two loops $\gamma,\gamma':S^1\to G$ we can combine the characteristic isomorphisms given by the previous theorem to obtain a compatible isomorphism
$H^\bullet(Y_{\gamma',\OOO};\QQ)\rightarrow H^\bullet(Y_{\gamma,\OOO};\QQ)$.
\end{remark}

\subsection{Proof of Theorem \ref{thm:uniqueness}}
\label{ss:proof-thm:uniqueness}
The existence of a compatible isomorphism is justified by Remark \ref{rmk:existence-comp-isom}, so it suffices to prove uniqueness.
For that, it is enough to prove that any compatible isomorphism
$\phi:H^{\bullet}(S^2\times\OOO;\QQ)\to H^{\bullet}(S^2\times\OOO;\QQ)$ is equal to the identity
(recall that $S^2\times\OOO=Y_{e,\OOO}$).

It is well known that any coadjoint orbit $\OOO$ admits
Kaehler structures (see for example \cite[Page 6]{Bott}).
Let $\omega\in H^2(\OOO;\QQ)$ be the cohomology class of a (rational) Kaehler form on $\OOO$ and
let $\phi:H^{\bullet}(S^2\times\OOO;\QQ)\to H^{\bullet}(S^2\times\OOO;\QQ)$ be a compatible
isomorphism. We implicitly identify $H^{\bullet}(\OOO;\QQ)$ with a subspace of
$H^{\bullet}(S^2\times \OOO;\QQ)$ via the inclusion given by the projection.

Let $n$ be the complex dimension of $\OOO$. Since $\omega$ is a Kaehler class,
we have $\omega^n\neq 0$.

In general, we may have $\phi(\omega)=\omega+\lambda b$ for some $\lambda\in\QQ$. Since
$\omega^{n+1}=0$ and $\phi$ is a morphism of rings, we have $0=\phi(\omega^{n+1})=
\phi(\omega)^{n+1}=(\omega+\lambda b)^{n+1}=(n+1)\lambda \omega^nb$, which implies $\lambda=0$.
Hence $\phi(\omega)=\omega$.

Let $i:\OOO\to S^2\times \OOO$ and $\pi:S^2\times\OOO\to\OOO$ be the inclusion
of one fiber and the projection respectively. Define the map
$$\psi=i_!\circ\phi\circ\pi^*:H^{\bullet}(\OOO;\QQ)\to H^{\bullet-2}(\OOO;\QQ),$$
where $i_!:H^{\bullet}(S^2\times \OOO;\QQ)\to H^{\bullet-2}(\OOO;\QQ)$ is integration over $S^2$
(i.e., capping with $b$).
Since by assumption $\phi$ fixes $b$ and it
satisfies $i^*\circ\phi=i^*$, to prove that $\phi$ is the identity it suffices to prove that
$\psi=0$.

Let $\alpha\in H^{\bullet}(\OOO;\QQ)$ be any class. Since $\phi$ and $\pi^*$ are morphisms of
rings we have
\begin{align*}
\psi(\omega\alpha) &= i_!\circ\phi\circ\pi^*(\omega\alpha)
=i_!(\phi(\pi^*\omega)\,\phi(\pi^*\alpha))
=i_!(\pi^*\omega\, \phi(\pi^*\alpha)) \\
&=\omega \, i_!(\phi(\pi^*\alpha)) \qquad\text{by the product formula}\\
&=\omega\psi(\alpha).
\end{align*}
Now suppose that $\alpha\in H^{n-k}(\OOO;\QQ)$ is a primitive class (see e.g.
\cite[Chapter V, \S 6]{We}), so that
$\omega^{k+1}\alpha=0$. By the previous formula we have
$$0=\psi(\omega^{k+1}\alpha)=\omega^{k+1}\psi(\alpha).$$
Now, $\psi(\alpha)\in H^{n-k-2}(\OOO;\QQ)$, and multiplication by $\omega^{k+1}$
gives an injective map $$H^{n-k-2}(\OOO;\QQ)\to H^{n+k}(\OOO;\QQ).$$
It follows that $\psi(\alpha)=0$.
Finally, since any cohomology class in $H^{\bullet}(\OOO;\QQ)$ can be written as a polynomial
in $\omega$ with primitive classes as coefficients, it follows that $\psi$ is identically $0$.

\subsection{Strategy for proving Theorem \ref{thm:injdiff}: integrality of the characteristic isomorphism}
\label{ss:strategy}
Our aim is to show that for non-homotopic loops $\gamma,\gamma':S^1\to G$ there is no compatible diffeomorphism between $Y_{\gamma,\OOO}$ and $Y_{\gamma',\OOO}$.
The proof will rely on Theorem \ref{thm:uniqueness}. In concrete terms we will prove that if $\gamma$ and $\gamma'$ are not homotopic then the (unique) characteristic isomorphism
$H^{\bullet}(Y_{\gamma',\OOO};\QQ)\to H^{\bullet}(Y_{\gamma,\OOO};\QQ)$
is not compatible with the integral lattices inside the rational cohomology groups.

It is well known that the integral cohomology of any coadjoint orbit $\OOO$ is torsion free and concentrated in even degrees (see e.g. \cite[Proposition 5.2]{BGG}). Applying Wang's exact sequence (or Serre's spectral sequence), one proves that the same properties are enjoyed by $Y_{\gamma,\OOO}$. Therefore, we can (and will) identify $H^{\bullet}(Y_{\gamma,\OOO};\ZZ)$ with a lattice in $H^{\bullet}(Y_{\gamma,\OOO};\QQ)$
(this is what we call the integral lattice).

\begin{definition}  Let $\OOO$ be a coadjoint orbit of a compact, connected, semisimple Lie group $G$,
let $\gamma,\gamma':S^1\to G$  be  loops. We say that the characteristic isomorphism
 \[\kappa\colon H^\bullet(Y_{\gamma',\OOO};\QQ)\rightarrow H^\bullet(Y_{\gamma,\OOO};\QQ)\]
 is integral if and only if $\kappa(H^\bullet(Y_{\gamma',\OOO};\ZZ))= H^\bullet(Y_{\gamma,\OOO};\ZZ)$.
\end{definition}

\begin{remark}
\label{rmk:PD-integrality}
Poincar\'e duality over the integers implies that $\kappa$ is integral if and only if
 \[\kappa(H^\bullet(Y_{\gamma',\OOO};\ZZ))\not\subset H^\bullet(Y_{\gamma,\OOO};\ZZ).\]
\end{remark}

Here is the key result for our strategy to prove Theorem \ref{thm:injdiffbun}:

\begin{prop}\label{pro:nonint}  Let $\OOO$ be a coadjoint orbit of a compact, connected, semisimple Lie group $G$,
let $\gamma,\gamma':S^1\to G$  be  loops.
If the characteristic isomorphism $\kappa\colon H^\bullet(Y_{\gamma',\OOO};\QQ)\rightarrow H^\bullet(Y_{\gamma,\OOO};\QQ)$ is not integral,
then $\gamma$ and $\gamma'$ are not homotopic in $\Diff(\OOO)$.
\end{prop}

\begin{proof} If $\gamma$ and $\gamma'$ were homotopic in $\Diff(\OOO)$, then there would exists a compatible diffeomorphism
$\psi\colon Y_{\gamma,\OOO}\rightarrow Y_{\gamma',\OOO}$ inducing a compatible isomorphism
$\psi^*\colon H^\bullet(Y_{\gamma',\OOO};\ZZ)\rightarrow  H^\bullet(Y_{\gamma,\OOO};\ZZ)$.
The tensor product with the rationals would define
a compatible isomorphism $\psi^*_\QQ: H^\bullet(Y_{\gamma',\OOO};\QQ)\rightarrow  H^\bullet(Y_{\gamma,\OOO};\QQ)$. By the uniqueness in Theorem \ref{thm:uniqueness} $\psi^*_\QQ=\kappa$,
but this would contradict the non-integrality of the characteristic isomorphism.
\end{proof}

The analysis of how the characteristic isomorphism acts on integral cohomology will be necessarily linked to problems in Schubert calculus.
Depending on the
problem at hand, it is not always possible to find conceptual proofs valid for all semisimple groups and all coadjoint orbits; rather, one is often forced to use different strategies
for different simple groups. For that reason, it becomes important to analyze whether the proof of Proposition \ref{pro:nonint} can be reduced
to some key cases. To that end, we need to generalize another property of the integral cohomology of coadjoint orbits to the bundles $Y_\gamma$:

As we discussed at the end of the proof of Theorem \ref{thm:char-isom},
if $\OOO'\cong G/T$ denotes a regular orbit,
then there is a geometric action of the Weyl group on $Y_{\gamma,\OOO'}$ such that the map induced in rational cohomology by any projection
 $Y_{\gamma,\OOO'}\to  Y_{\gamma,\OOO}$ induced by a $G$-equivariant
 map $\OOO'\to\OOO$ corresponds to the inclusion:
\[H^\bullet(Y_{\gamma,\OOO};\QQ)\cong H^\bullet(Y_{\gamma,\OOO'};\QQ)^{W_{P}}\subset H^\bullet(Y_{\gamma,\OOO'};\QQ).\]

\begin{lemma}\label{lem:invcohom} The projection $Y_{\gamma,\OOO'}\to  Y_{\gamma,\OOO}$ induces a monomorphism in integral cohomology with
image the $W_P$-invariant subspace:
\[H^\bullet(Y_{\gamma,\OOO};\ZZ)\cong H^\bullet(Y_{\gamma,\OOO'};\ZZ)^{W_{P}}
=H^{\bullet}(Y_{\gamma,\OOO'};\QQ)^{W_P}\cap H^\bullet(Y_{\gamma,\OOO'};\ZZ)
\subset H^\bullet(Y_{\gamma,\OOO'};\ZZ).\]
\end{lemma}
\begin{proof}
The Weyl group acts by automorphisms of the Wang sequence for the integral cohomology of $Y_{\gamma,\OOO'}$
\begin{equation}\label{eq:wang}
\cdots \longrightarrow H^{k-2}(\OOO';\ZZ)\overset{b\cup}{\longrightarrow}H^k(Y_{\gamma,\OOO'};\ZZ)
\overset{i^*}{\longrightarrow}H^k(\OOO';\ZZ)\longrightarrow \cdots,
\end{equation}
and similarly for $Y_{\gamma,\OOO}$.
Therefore the $W_{P}$-invariant
part
\begin{equation}\label{eq:wanginv}
\cdots \longrightarrow H^{k-2}(\OOO';\ZZ)^{W_{P}}\overset{b\cup}{\longrightarrow}H^k(Y_{\gamma,\OOO'};\ZZ)^{W_{P}}
\overset{i^*}{\longrightarrow}H^k(\OOO';\ZZ)^{W_{P}}\longrightarrow \cdots
\end{equation}
is a chain complex. The projection $Y_{\gamma,\OOO'}\rightarrow Y_{\gamma,\OOO}$ induces a map between of chain complexes
from the Wang sequence for $Y_{\gamma,\OOO}$ to the invariant part (\ref{eq:wanginv}). Because the
cohomologies of $\OOO'$ and $\OOO$ are concentrated at even degrees, the Wang sequences
for $Y_{\gamma,\OOO'}$ and $Y_{\gamma,\OOO}$ split in short exact sequences. Combining this
with the fact that $H^\bullet(\OOO';\ZZ)\rightarrow H^\bullet(\OOO;\ZZ)^{W_{P'}}$ is
 an isomorphism, it follows that (\ref{eq:wanginv}) is
exact. Finally, the five lemma implies that
\[H^\bullet(Y_{\gamma,\OOO};\ZZ)\rightarrow H^\bullet(Y_{\gamma,\OOO'};\ZZ)^{W_{P}}\] is an isomorphism.
The remaining equality
$$H^\bullet(Y_{\gamma,\OOO'};\ZZ)^{W_{P}}
=H^{\bullet}(Y_{\gamma,\OOO'};\QQ)^{W_P}\cap H^\bullet(Y_{\gamma,\OOO'};\ZZ)$$
follows from this general fact: if $\Lambda$ is a free abelian group
on which a finite group $\Gamma$ acts by automorphisms, then $\Lambda^{\Gamma}=(\Lambda\otimes\QQ)^{\Gamma}\cap\Lambda$.
\end{proof}

Now we can identify the first key cases of Theorem \ref{thm:injdiff}:

\begin{prop}\label{pro:min-orbits}
Suppose that for any non-contractible loop $\gamma$ in a compact, connected, semisimple Lie group $G$ and for any minimal coadjoint orbit $\OOO''$ of $G$, the characteristic isomorphism
\[\kappa\colon H^\bullet(Y_{e,\OOO''};\QQ)\rightarrow H^\bullet(Y_{\gamma,\OOO''};\QQ)\]
is not integral.
 Then \[\rho:\pi_1(G)\rightarrow \pi_1(\Diff(\OOO))\] is injective for every coadjoint orbit $\OOO$ of $G$.
 \end{prop}
 \begin{proof}
Let $\OOO$ be any coadjoint orbit of $G$. Choose $G$-equivariant maps
$$\OOO'=G/T\to\OOO\to\OOO'',$$
where $\OOO''$ is minimal (and of course $\OOO'$ is maximal). Let
$x\in\OOO$ be the image of $eT\in\OOO'$ and let $x''\in\OOO''$ be the image of $x$.
Let $P$ and $P''$ be the stabilizers of $x$ and $x''$ respectively. We have inclusions
$T\subset P\subset P''$ and we denote as always the residual Weyl groups
$W_P=W\cap P$ and $W_{P''}=W\cap P''$.

The $G$-equivariant maps above give rise to a sequence of projections
\[Y_{\gamma,\OOO'}\to Y_{\gamma,\OOO}\to Y_{\gamma,\OOO''}.\]
 First, the maps induced in rational cohomology allow us to identify
the rational cohomologies of both $Y_{\gamma,\OOO''}$ and $Y_{\gamma,\OOO}$ with the following subspaces of
the rational cohomology of $Y_{\gamma,\OOO'}$:
\begin{equation}\label{eq:inc-rat}
 H^\bullet(Y_{\gamma,\OOO'};\QQ)^{W_{P''}}\subset H^\bullet(Y_{\gamma,\OOO'};\QQ)^{W_{P}}
\end{equation}
Second, the maps induced in rational cohomology by the $G$-equivariant projections $Y_{e,\OOO}\to Y_{e,\OOO''}$ and $Y_{\gamma,\OOO}\to Y_{\gamma,\OOO''}$,  and the characteristic isomorphisms
for $Y_{\gamma,\OOO}$ and $Y_{\gamma,\OOO''}$, fit into a commutative diagram.
If we use the identification of rational cohomologies in (\ref{eq:inc-rat}), then this commutative diagram becomes:
 \begin{equation}\label{eq:charcomm0}
\xymatrix{H^\bullet(Y_{e,\OOO'};\QQ)^{W_{P''}}\ar^{\kappa_{\OOO''}}[r]\ar@{^{(}->}[d] & H^\bullet(Y_{\gamma,\OOO'};\QQ)^{W_{P''}}\ar@{^{(}->}[d] \\
 H^\bullet(Y_{e,\OOO'};\QQ)^{W_{P}} \ar^{\kappa_{\OOO}}[r] &H^\bullet(Y_{\gamma,\OOO'};\QQ)^{W_{P}}.}
 \end{equation}
Here $\kappa_{\OOO''}$ and $\kappa_{\OOO}$ denote the characteristic isomorphisms
for the clutching construction applied to $\OOO''$ and $\OOO$ respectively.

Now assume that the characteristic isomorphism $\kappa_{\OOO''}$ is not integral.
By Lemma \ref{lem:invcohom} we can also identify the integral cohomologies of $Y_{\gamma,\OOO''}$
and $Y_{\gamma,\OOO}$ with the following two subgroups of $H^{\bullet}(Y_{\gamma,\OOO'};\ZZ)$:
\begin{equation}\label{eq:incint}
 H^\bullet(Y_{\gamma,\OOO'};\QQ)^{W_{P''}}
 \cap
 H^\bullet(Y_{\gamma,\OOO'};\ZZ)
 \subset
 H^\bullet(Y_{\gamma,\OOO'};\QQ)^{W_{P}}
 \cap
 H^\bullet(Y_{\gamma,\OOO'};\ZZ),
\end{equation}
so the non-integrality of $\kappa_{\OOO''}$ implies
(see Remark \ref{rmk:PD-integrality}):
\begin{equation}
 \label{eq:non-integrability-kappa-ooo''}
 \kappa_{\OOO''} ( H^\bullet(Y_{e,\OOO'};\QQ)^{W_{P''}}\cap H^\bullet(Y_{e,\OOO'};\ZZ))
 \not\subset
 H^\bullet(Y_{\gamma,\OOO'};\QQ)^{W_{P''}}\cap H^\bullet(Y_{\gamma,\OOO'};\ZZ).
 \end{equation}
Diagram (\ref{eq:charcomm0}) now implies that
$$
 \kappa_{\OOO} ( H^\bullet(Y_{e,\OOO'};\QQ)^{W_{P''}}\cap H^\bullet(Y_{e,\OOO'};\ZZ))
 \not\subset
 H^\bullet(Y_{\gamma,\OOO'};\QQ)^{W_{P''}}\cap H^\bullet(Y_{\gamma,\OOO'};\ZZ),
$$
from which it follows immediately (using
$
 \kappa_{\OOO} ( H^\bullet(Y_{e,\OOO'};\QQ)^{W_{P''}})
=H^\bullet(Y_{\gamma,\OOO'};\QQ)^{W_{P''}}$)
that
$$
 \kappa_{\OOO} ( H^\bullet(Y_{e,\OOO'};\QQ)^{W_{P}}\cap H^\bullet(Y_{e,\OOO'};\ZZ))
 \not\subset
 H^\bullet(Y_{\gamma,\OOO'};\QQ)^{W_{P}}\cap H^\bullet(Y_{\gamma,\OOO'};\ZZ).
$$
Therefore,
the characteristic
morphism for $Y_{\gamma,\OOO}$ is not integral. It then follows, by Proposition \ref{pro:nonint}, that
$\rho:\pi_1(G)\rightarrow \pi_1(\Diff(\OOO))$ is injective.
\end{proof}

As for the group $G$, it is elementary to check that it can be assumed to be adjoint. Indeed, let $G(\glie)$ be the simply
 connected Lie group with Lie algebra $\glie$. This is a compact group
 with finite center $Z$  \cite[Chapter V,\S 7]{BD}, and, by definition, $G(\glie)/Z\subset \mathrm{Aut}(\glie,[\cdot,\cdot ])$ is the adjoint group of the Lie algebra $\glie$,
 so the adjoint action of $G$ on $\glie$ gives a covering map $G\rightarrow G(\glie)/Z$ with kernel the (finite) center of $G$.
 It follows that there is a commutative diagram
\[
\xymatrix{G\ar[rr] \ar[rd] & & \Diff(\OOO)  \\
& G(\glie)/Z \ar[ru] }
\]
Thus, if  $\rho\colon \pi_1(G(\glie)/Z)\rightarrow \pi_1(\Diff(\OOO))$ is a monomorphism, then
  $\rho\colon \pi_1(G)\rightarrow \pi_1(\Diff(\OOO))$ is the composition of two monomorphisms, and therefore a monomorphism.

If $G$ is adjoint, then it is the product of compact, simple, adjoint groups \cite[Chapter V,\S 8]{BD}. The following result --- whose proof is deferred
until Section \ref{sec:genSchubert} as it requires
some basics of Schubert calculus on $Y_{\gamma,\OOO}$ ---
shows that it is enough to consider the case of simple groups:
\begin{prop}\label{pro:simplegroups}
If for any non-contractible loop in a compact, connected, simple Lie group and for any
coadjoint orbit of the group the characteristic isomorphism  is not integral, then
 \[\rho\colon \pi_1(G)\rightarrow \pi_1(\Diff(\OOO))\]
is injective for any $\OOO$ coadjoint orbit of a compact, connected,  semisimple Lie group $G$.
 \end{prop}

Consequently, Theorem \ref{thm:injdiff} will follow from the non-integrality of the characteristic isomorphism for the minimal orbits
of the simple adjoint groups $\PU(n+1)$, $\Sp(n)/Z$, $\SO(n)/Z$, $\mathrm{E}_6/Z$ and $\mathrm{E}_7/Z$ (note that the simply connected simple groups
$E_8$, $F_4$ and $G_2$
have trivial center \cite[Page 516, Table IV]{He}).

\section{Generalized Schubert calculus}\label{sec:genSchubert}

The rational cohomology ring of a coadjoint orbit is described algebraically as the ring of coinvariants;
this is no longer true in general for the integral cohomology ring \cite[Theorem 2.1]{To}. The
understanding of the integral cohomology is attained --- in a dual fashion --- by a description of a basis of the integral homology in terms of Schubert classes.
The subtle interaction between the algebraic and geometric descriptions of the cohomology is the subject of Schubert calculus.

The algebraic description of the rational cohomology of the bundles $Y_{\gamma,\OOO}$
is nothing but the characteristic isomorphism (\ref{eq:characteristicisom}). Therefore  the study of its integrality is necessarily linked to a description of the
integral homology of the bundles. This is the reason why we need to generalize some basic results of Schubert calculus to our bundle setting.

We shall start by recalling a presentation of our bundles $Y_{\gamma,\OOO}$ which incorporates the linearized/algebraic action of a complex torus. Then we
shall look at the module structure at fixed points of the action and its relation to the line bundles defined by roots. This
will be needed both to present a Bruhat/Schubert decomposition of $Y_{\gamma,\OOO}$, and to produce explicit formulas  for the integration
of polynomials (equivariant classes) on appropriate Schubert varieties via localization.

Next, we shall prove
 a  generalized Chevalley formula to integrate any polynomial
against any Schubert class. Compared to the localization formulas, the drawback of the
generalized Chevalley formula  will be
its complexity in terms of the Hasse diagram of the cell.
 However, this formula will be very useful for us from both
a theoretical and a computational viewpoint. On the one hand, it will allow us to prove Proposition \ref{pro:simplegroups} reducing
Theorem \ref{thm:injdiff} to the case of minimal orbits. On the other hand, we will be able to implement  the generalized Chevalley formula in an algorithm
which we will use to prove
the non-integrality of the characteristic isomorphisms for certain minimal orbits of $E_7$.
Interestingly enough, our main theorem   --- but just for regular orbits ---
will follow immediately from the generalized Chevalley formula:

\begin{corollary}\label{cor:regular} Let $G$ be a compact, connected, semisimple Lie group, let $\OOO'$ be a regular coadjoint orbit, and let
$\gamma\subset G$ be a non-contractible loop. Then  $\gamma\subset \Diff(\OOO')$ is not contractible.
 \end{corollary}

We shall close this section describing some special Dynkin subdiagrams
to which we can associate a smooth Schubert variety with a very simple module structure  at fixed points.

\subsection{The algebraic torus action on the bundles $Y_{\gamma,\OOO}$.}

The fundamental group of a compact, connected, semisimple Lie group $G$ can be described in terms of lattices \cite[Chapter V,\S7]{BD}.
 Recall that we denote by $\tlie$ the Lie algebra of the maximal torus $T\subset G$.
 The Lie algebra $\tlie$ can be naturally identified with the Lie algebra of the maximal torus in $G(\glie)$ covering $T$ and also with the Lie algebra of
 the maximal torus in the adjoint group $G(\glie)/Z$ covered by $T$.

 The kernels of the exponential maps for the maximal tori in $G(\glie)$ and $G(\glie)/Z$
 can be identified with the coroot and coweight lattices respectively:
\[\Lambda_r^\vee<\Lambda_w^\vee<\tlie\]
The quotient $\Lambda_w^\vee/\Lambda_r^\vee$ is a finite  group which is canonically isomorphic to the fundamental group of the adjoint group:
\begin{eqnarray*}
      \xi:\Lambda_w^\vee/\Lambda_r^\vee & \longrightarrow &\pi_1(G(\glie)/Z)\\
            z+\Lambda_r^\vee &\longmapsto  & [\mathrm{exp}(tz)],\, t\in [0,1]
\end{eqnarray*}
For any $z\in  \Lambda_w^\vee$ we shall write
$$Y_{z,\OOO}=Y_{\gamma_z,\OOO},$$
where
$\gamma_z$ is the loop $e^{2\pi\imag t}\mapsto\exp (tz)$ representing $\xi(z)$.
Since $\xi$ is an isomorphism, it will suffice for our purposes to study the fibrations
$Y_{z,\OOO}\to S^2$ as $z$ runs over a set of representatives of the cosets in $\Lambda_w^\vee/\Lambda_r^\vee$.

We fist recall how coadjoint orbits can be endowed with an algebraic torus action. This will we useful to discuss the corresponding problem
for the bundles $Y_{z,\OOO}$, and also to fix some notation.

Choose $q:G/T\to \OOO$ a $G$-equivariant projection and let $P$ be the centralizer of $x=q(eT)$,
so that $T\subset P$.

Let $\GG$ denote the adjoint complex Lie group
integrating the complexified Lie algebra $\glie_\CC$ and let $\hlie\subset \glie_\CC$ be the complexification of $\tlie$. Let $\TT\subset\GG$ be the complex torus with Lie algebra $\hlie$.
Fix positive roots $\Delta^+\subset \Delta$
for the adjoint action of $\hlie$, so that:
\[\mathfrak{g}_\CC=\hlie\oplus \sum_{\alpha\in \Delta^+}\mathfrak{g}_\alpha\oplus
 \sum_{\alpha\in \Delta^+}\mathfrak{g}_{-\alpha},\]
Then any subset   $J$ of the set $\Pi\subset \Delta^+$ of simple roots such determines a
parabolic subgroup $\PP\subset\GG$ defined
as the normalizer of the subalgebra
 \[
\mathfrak{p}=\mathfrak{h}\oplus \sum_{\alpha\in \Delta^+}\mathfrak{g}_\alpha\oplus
 \sum_{\alpha\in \Delta^+_J} \mathfrak{g}_{-\alpha},
\]
where $\Delta^+_J$ are those positive roots which are sums of roots in $J$. A fundamental result in Lie theory (see e.g. \cite[Section 4]{Bott2}
shows the existence of a
subset $J\subset\Pi$ whose associated parabolic subgroup $\PP$
has the property that
the natural inclusion $G\hookrightarrow\GG$ induces a $G$-equivariant diffeomorphism
$G/P\cong\GG/\PP$.
This endows $\OOO$ with a canonical $G$-invariant holomorphic structure with respect to which
$\GG$ acts holomorphically, transitively, and in such a way that the stabilizer of $eP\in G/P$
is $\PP$.

To exhibit the algebraic torus action  the complex group $\GG$ is replaced by the algebraic group $\mathrm{Aut}(\glie_\CC)$,
and $\PP$ is replaced by $N(\plie)$, the normalizer of $\plie$ in $\mathrm{Aut}(\glie_\CC)$. It is well-known (see for example
\cite[Page 6]{Bott})
that there is
a linear representation   $K:\mathrm{Aut}(\glie_\CC)\to \GL(V)$ and a vector
$v$ such that the stabilizer of the line $[v]\in \PP(V)$ is $N(\plie)$. Therefore the representation furnishes a projective embedding
\[G/P\cong \GG/\PP\cong K(\mathrm{Aut}(\glie_\CC))\cdot [v]\subset \PP(V),\]
for which the action of the complex torus $\TT$ on $\GG/\PP$ corresponds to
the restriction to $K(\mathrm{Aut}(\glie_\CC))\cdot [v]$ of the algebraic action of $\TT$
 on $\PP(V)$ induced by $K$.

Next we will see that the fibration $Y_{z,\OOO}\to \PP^1\cong S^2$ can be given a  complex structure,
with respect to which it supports a holomorphic action of $\CC^*\times\TT$. This can be done in a way
compatible with the algebraic structure of the fibers and with the action of $\TT$ on them.

The vector  $z\in  \Lambda_w^\vee\subset \tlie$ determines upon exponentiation the cocharacter
\begin{equation}\label{eq:weight-morphism}
 \phi_z: \CC^*\to \TT.
\end{equation}
The corresponding $\CC^*$-action on $\GG/\PP$ can be used to make the clutching construction
compatible with complex structures.
Rather than gluing two trivial bundles over disks, we take two copies of the
trivial holomorphic bundle over the complex line and identify the
complements of the fiber over zero by a biholomorphism:
\begin{equation}\label{eq:associatedbcx}
Y_{z,\OOO}=(\CC\times \GG/\PP) \sqcup (\CC\times \GG/\PP)/
 (u,g\PP)\sim (u^{-1}, \phi_{z}(u)g\PP),\quad u\in \CC^*,\,g\in \GG,
\end{equation}

It is clear that the left action of $\TT$ on $\GG/\PP$ induces an action of $\TT$
on $Y_{z,\OOO}$ by (holomorphic) bundle automorphisms. In fact, more is true.
\begin{lemma}\label{lem:torus-action} There is a $\CC^*\times \TT$-action on $Y_{z,\OOO}$ by bundle automorphisms.
 \end{lemma}
 \begin{proof}
We use a different construction of $Y_{z,\OOO}$  (see for example \cite[Section 4]{CMP}): we
consider the diagonal action of $\CC^*$ on $\CC^2\backslash \{0\}\times \GG/\PP$ associated
to the the standard (diagonal) action on $\CC^2\backslash \{0\}$
and the left action  on $\GG/\PP$ defined by
 the cocharacter $\phi_{-z}$. Then it is clear that we have a biholomorphic map:
 \[Y_{z,\OOO}\cong \CC^2\backslash \{0\}\times_{\CC^*}\GG/\PP.\]
We have an obvious product action of
$(\CC^*)^2\times \TT$ on $\CC^2\backslash \{0\}\times \GG/\PP$. As the action of the subtorus
 $\CC^*\times \{1\}\times \TT$ commutes with the above $\CC^*$-action, it descends to the quotient $Y_{z,\OOO}$.

It is useful to know how the $\CC^*\times \TT$-action reads in  the  (complex) clutching construction (\ref{eq:associatedbcx}).
Given $(\lambda,\lambda')\in \CC^*\times \TT$  we have:
\begin{equation}\label{eq:actionclutching}
(\lambda,\lambda')\cdot (u,g\PP)=(\lambda u,\lambda'g\PP)
\end{equation}
on the first copy of $\CC\times \GG/\PP$, and
\begin{equation}\label{eq:actionclutching-2}
(\lambda,\lambda')\cdot (u,g\PP)=(\lambda^{-1}u,\phi_{z}(\lambda)\lambda' g\PP)
\end{equation}
on the second one.
\end{proof}

The next proposition shows that both the complex structure on $Y_{z,\OOO}$ and the holomorphic
action of $\CC^*\times\TT$ defined in Lemma \ref{lem:torus-action} are algebraic.

\begin{prop}\label{pro:linaction} The $\CC^*\times \TT$-manifold $Y_{z,\OOO}$ admits projective embeddings that linearize the action.
 \end{prop}
\begin{proof}
Let   $K:\GG\to \GL(V)$ be a linear representation which defines a projective embedding of $\OOO$ in $\PP(V)$.
If we regard $\CC^2\backslash \{0\}\times V $ as the trivial bundle over $\CC^2\backslash \{0\}$,
then its quotient by the $\CC^*$-action defined by the composition
of the cocharacter $\phi_{-z}$ with $K$
is a holomorphic vector bundle $V_z\to \PP^1$. Its description in terms of
trivializations is:
\begin{equation}\label{eq:clutching-vector}V_z=(\CC\times V)\sqcup(\CC\times V)/
 (u,v)\sim (u^{-1}, K(\phi_{z}(u))v),\, u\in \CC^*,\,v\in V.
 \end{equation}
The holomorphic bundle $V_z$ supports an action of $\CC^*\times \TT$ defined as in the proof of Lemma \ref{lem:torus-action} (but replacing the $\CC^*$-space  $\GG/\PP$
by $V$). This action descends to the projectivized bundle $\PP(V_z)\to \PP^1$. By construction $Y_{z,\OOO}\subset \PP(V_z)$ is an embedding
of $\CC^*\times \TT$-manifolds, so it is enough to show that we can linearize the action
of $\CC^*\times \TT$ on $\PP(V_z)$, which is a standard result:
as the action on $\PP(V_z)$ comes from a linear action on $V_z$, it lifts to an action on the hyperplane line bundle $\OOO_{\PP(V_z)}(1)$. This
bundle might not be ample, but
\[\pi^*\OOO_{\PP^1}(k)\otimes \OOO_{\PP(V_z)}(1)\] is ample for $k$ large enough.
This is a standard fact, but for completeness we sketch a proof. Let $L_b=\OOO_{\PP^1}(1)\to \PP^1$
and $L_f=\OOO_{\PP(V_z)}(1)\to\PP(V_z)$. The Fubini--Study metric $h_b$ on $L_b$
is positive, meaning that the curvature
$\omega_b\in\Omega^{1,1}(\PP^1)$ satisfies $\imag\omega_b(u,\overline{u})>0$ for every holomorphic
$u\in T\PP^1\otimes\CC$ (see e.g. \cite[p. 29]{GH}). Similarly, the fiberwise Fubini--Study metric
$h_f$ on $L_f$ has curvature $\omega_f\in\Omega^{1,1}(\PP(V_z))$ which is positive when restricted
to the fibers of $\PP(V_z)\to\PP^1$. One then proves that for any big enough 
natural number $k$  the form $k\pi^*\omega_b+\omega_f$ is positive on the entire holomorphic
tangent bundle of $\PP(V_z)$.  Now $k\pi^*\omega_b+\omega_f$ is the curvature of the metric
$(\pi^*h_b)^{\otimes k}\otimes h_f$ on $\pi^*L_b^{\otimes k}\otimes L_f$.
By Kodaira's embedding theorem, it follows that
$\pi^*L_b^{\otimes k}\otimes L_f=\pi^*\OOO_{\PP^1}(k)\otimes \OOO_{\PP(V_z)}(1)$ is ample.

The action of the first factor
$\CC^*\subset \CC^*\times \TT$ on $\PP^1$ lifts to $\OOO_{\PP^1}(k)$, and we use this
lifted action to define the following $\CC^*\times \TT$-action on $\pi^*\OOO_{\PP^1}(k)$:
\[(\lambda,\lambda')\cdot \pi^*w=\pi^*\lambda w,\quad (\lambda,\lambda')\in \CC^*\times \TT,\,w\in \OOO_{\PP^1}(k).\]
Its tensor product with the linear action on  $\OOO_{\PP(V_z)}(1)$ defines a lift of the $\CC^*\times \TT$-action to $\pi^*\OOO_{\PP^1}(k)\otimes \OOO_{\PP(V_z)}(1)$,
and this proves the proposition.

\end{proof}

\subsection{Fixed points and the generalized Schubert decomposition}

To obtain the classical Schubert decomposition of a regular  coadjoint orbit $\OOO'\cong \GG/\BB$ one
considers the left action of the positive Borel subgroup  $\BB$ on $\GG/\BB$.

The resulting partition of $\GG/\BB$ into orbits has some standard properties which we are
going to recall. Before that, let us recall that once a set of simple
positive roots
$\Pi\subset\Delta^+$ has been chosen, one can define a partial order on the Weyl group $W$ (see e.g. \cite[Section 2]{BGG})
by saying that, given $w,w'\in W$, $w'\leq w$ if $l(w)-l(w')=k$, where $l$ is the length function in the Weyl group,
and there exist elements $s_1,\dots,s_k\in W$
such that each $s_i$ is a reflection associated to a positive root and
$w=s_1\dots s_kw'$. These are the properties of the action of $\BB$ on $\GG/\BB$
which we are going to use \cite{BGG}:
\begin{itemize}
 \item Assigning to each $w\in W$ the orbit $B_w:=\BB w\BB\subset\GG/\BB$
defines a bijection between the set of orbits and the Weyl group. The orbits
$B_w$ are called Bruhat cells.
 \item Each Bruhat cell $B_w$ is canonically biholomorphic to $\CC^{l(w)}$.
 \item The Schubert variety associated to $w$ is by definition $X_w:=\overline{B_w}$. We have
  $B_{w'}\subset X_w$ if and only if $w'\leq w$.
 \item The Schubert varieties are the closures of the cells of a CW decomposition of $\GG/\BB$
 with even dimensional cells, so in particular $\{[X_w]\}_{w\in W}$ is an
 additive basis of the integral homology of $\GG/\BB$.
 \end{itemize}

The left action of $\BB$ on $\GG/\PP$ defines a similar stratification \cite[Section 5]{BGG}, but this time the Bruhat
cells are parametrized by the right cosets $W/W_P$.
 More precisely, on each equivalence class in $W$ modulo $W_P$ there exist a unique element of minimal length,
 and its Bruhat cell is the only one which is not shrunk (i.e., it does not lose dimension)
 under the projection $\GG/\BB\to \GG/\PP$.

Our aim is to define an analogue of the Schubert decomposition for the fibrations $Y_{z,\OOO}$.
In doing so, one encounters the difficulty that
the spaces $Y_{z,\OOO}$ do not carry in any natural way an action of $\BB$ extending the action
on the fibers. However, the Bruhat decomposition of $\OOO$ can be defined in terms of the action of
the Cartan subgroup $\TT\subset\BB$ (see  e.g. \cite[II, \S 4.1]{BCG}).
This is good news, because the action of $\TT$ does extend to an algebraic action of $\CC^*\times\TT$ on $Y_{z,\OOO}$.
Hence, to generalize the Schubert decomposition to the bundles $Y_{z,\OOO}$ we will use the $\CC^*\times \TT$-action. To that end
we need to lift the action to the line bundles over $Y_{z,\OOO}$ defined by roots, and to determine the
$\CC^*\times \TT$-module structure of both the tangent space to
the fixed points of the action on $Y_{z,\OOO}$ and the fibers over these fixed points.

We start by fixing some notation for our bundles $Y_{z,\OOO}$. We single out the origin and point at infinity in the base:
\[0:=[0:1]\in \PP^1,\quad \infty:=[1:0]\]
We  also identify the fiber over $0$ and $\infty$ with $\GG/\PP$ using the product structure in the charts
of the clutching construction (\ref{eq:associatedbcx}):
\begin{equation}\label{eq:triv-fibers}
{Y_{z,\OOO}}|_{0}\cong \GG/\PP\cong {Y_{z,\OOO}}|_{\infty}
 \end{equation}
We denote by $W^P\subset W$ those elements with minimal length in its right coset in $W/W_P$. If $w\in W^P$, then
we denote the points corresponding to $w\PP$ under (\ref{eq:triv-fibers}) by:
\begin{equation}\label{eq:fibered-fixed--points}
 w(0)\in {Y_{z,\OOO}}|_{0},\, w(\infty)\in {Y_{z,\OOO}}|_{\infty}.
\end{equation}
By formulas (\ref{eq:actionclutching}) and (\ref{eq:actionclutching}) in Lemma \ref{lem:torus-action} these are exactly
the fixed points of the action of $\CC^*\times \TT$ on $Y_{z,\OOO}$.

Let $\zeta\in \Lambda_r^{W_P}$ be a root invariant by $W_P$. It integrates into a character:
\[\chi_\zeta\colon \PP\to \CC^*.\]
We let $\PP$ act  on  by the diagonal action associated to
the trivial action on $\CC^2\backslash \{0\}$, the right action on $\GG$, and
the action on $\CC^*$ defined by $\chi_\zeta$.
We also consider the diagonal action of  $\CC^*$  associated to the
diagonal action on $\CC^2\backslash \{0\}$, the (left) action on $\GG$ induced
 by the cocharacter $\phi_{-z}$, and the trivial action on $\CC$.
We use the product action of $\CC^*\times \PP$ to define the quotient space
\[L_\zeta:= (\CC^2\backslash \{0\}\times \GG\times \CC)/\CC^*\times \PP.\]
This space is a line bundle over $Y_{z,\OOO}$ with the following properties:
\begin{enumerate}
 \item[(i)]  Its restriction to both $ {Y_{z,\OOO}}|_{0}$ and ${Y_{z,\OOO}}|_{\infty}$ becomes under the identification (\ref{eq:triv-fibers})
 the classical line bundle defined by the character $\chi_\zeta$ (see e.g. \cite[Page 7]{Bott})
\[\LLL_\zeta:=\GG\times_{\chi_\zeta} \CC\]
on which $\TT$ acts by line bundle automorphisms.
\item[(ii)] By definition of the Chern-Weil map:
\[\kappa(\zeta)=c_1(L_\zeta).\]
\end{enumerate}
The line bundle $L_\zeta$ is a $\CC^*\times \TT$-space in a natural way:
\begin{lemma}\label{lem:bundle-action} There is a linear lift  of the $\CC^*\times \TT$-action on $Y_{z,\OOO}$ to the line bundle  $L_\zeta$
 which extends the $\TT$-action on  $\LLL_\zeta$.
\end{lemma}
 \begin{proof}
 Consider the diagonal action
of $\CC^*$ on $\CC^2\backslash \{0\}\times \GG\times \CC$ associated to the standard action on the first factor of $\CC^2\backslash \{0\}$ and the trivial
actions on $\GG$ and $\CC$. Consider also the diagonal action of $\TT$ associated to the trivial action on the first and third factors, and the left
action on $\GG$. As the product action of $\CC^*\times \TT$
commutes with the action of $\CC^*\times \PP$, it descends to an action on $L_\zeta$. This action clearly lifts the $\CC^*\times \TT$-action
defined in Lemma \ref{lem:torus-action} and extends the $\TT$-action on $\LLL_\zeta\cong L_\zeta|_{0}$.
\end{proof}

As one would naturally expect
 the $\CC^*\times \TT$-module structure  on the fibers of $L_\zeta$
 over fixed points $w(0),w(\infty)$ is related
to the $\TT$-module structure of the fibers of $\LLL_\zeta$ over $w\PP$.
Before proving that we will introduce some notation:
If $\zeta$ is a weight of a complex torus,
then we denote by $\CC_{\zeta}$ the module structure on $\CC$ defined by
the character $\chi_\zeta$. If  $z$ is a cocharacter of $\TT\subset \GG$, then we denote the morphism
\[\CC^*\times \TT\to \TT,\quad  (\lambda,\lambda')\mapsto \phi_z(\lambda)\lambda'\]
 by
\[\Phi_z:\CC^*\times \TT\to \TT.\]

\begin{prop}\label{pro:rootlocalization} Let $z\in \Lambda^\vee_w$ and $\PP\subset \GG$ a parabolic subgroup. Then for any root
$\zeta\in \Lambda_r^{W_P}$ and any $w\in W^P$  we have the
following identifications of  $\CC^*\times \TT$-modules:
\begin{equation}\label{eq:fibermodule}
{L_\zeta}|_{w(0)}\cong \CC_{{\pi_{2}}^*(w\cdot \zeta)},\quad
 {L_\zeta}|_{w(\infty)}\cong \CC_{\Phi_z^*(w\cdot \zeta)},
\end{equation}
where $\pi_2\colon  \CC^*\times \TT\to \TT$ is the second projection.
\end{prop}
\begin{proof}
Using the description of ${L_\zeta}$ as a quotient the $\CC^*\times \TT$-action at fixed points is:
\begin{align}\label{eq:action-clutching0}
 (\lambda,\lambda')(0,u_2,w,\varsigma)\CC^*\times \PP &=(0,u_2,\lambda'w,\varsigma)\CC^*\times \PP=(0,u_2,w(0),\zeta(w^{-1}\lambda'w)\varsigma)\CC^*\times \PP,
 \\\nonumber
 (\lambda,\lambda')(u_1,0,w,\varsigma)\CC^*\times \PP & =(\lambda u_1,0,\lambda'w,\varsigma)\CC^*\times \PP=
 (u_1,0,\phi_z(\lambda)\lambda'w,\varsigma)\CC^*\times \PP =\\ \nonumber
 &= (u_1,0,w,\zeta(w^{-1}\phi_z(\lambda)\lambda'w)\varsigma)\CC^*\times \PP,
\end{align}
 and therefore $\CC^*\times \TT$-module structure
 of  ${L_\zeta}|_{w(0)}$ and  ${L_\zeta}|_{w(\infty)}$ is given by (\ref{eq:fibermodule}).
\end{proof}

Next we want to determine the module structure on the tangent space of fixed points for the $\CC^*\times \TT$-action on $Y_{z,\OOO}$.
Each such module will have a `vertical' submodule --- coming via pullback
as in Proposition  \ref{pro:rootlocalization} from the
$\TT$-module structure of $T_{w\PP}\GG/\PP$ --- and an  additional 1-dimensional `horizontal' submodule, this coming from the $\CC^*$-action on $\PP^1$:

\begin{prop}\label{pro:tangentspace} Let $z\in \Lambda^\vee_w$ and $\PP\subset \GG$ a parabolic subgroup. If $w\in W^P$, then
the following  identifications of  $\CC^*\times \TT$-modules hold:
\begin{align}\label{eq:normalmodule0}
T_{w(0)}Y_{z,\OOO} &\cong \CC_{\pi_1^*(0)}\oplus \bigoplus_{\alpha\in \Delta^-\backslash{\Delta_J^-}}\CC_{{\pi_{2}}^*(w\cdot \alpha)},\\ \nonumber
T_{w(\infty)}Y_{z,\OOO} &\cong  \CC_{-\pi_1^*(0)}\oplus \bigoplus_{\alpha\in \Delta^-\backslash{\Delta_J^-}} \CC_{\Phi_z^*(w\cdot \alpha)},
 \end{align}
 where $\Delta_J^-$ is the subset of negative roots generated by the subset of simple roots $J\subset \Pi$ and
 $\pi_1\colon \CC^*\times  \TT\to \CC^*$ is the first projection.
\end{prop}
\begin{proof}
Let us first consider the regular case $\OOO'\cong \GG/\BB$. Under this identification the tangent bundle to $\OOO'$ corresponds to
the bundle associated to the adjoint action on $\glie_\CC/\blie$, which splits into a direct sum of line bundles associated to negative roots:
\begin{equation}\label{eq:tangent-roots}
T(\GG/\BB)=\GG\times_{\BB} (\glie_\CC/\blie)=\bigoplus_{\alpha\in \Delta^-} \LLL_\alpha.
\end{equation}
Using the charts of the clutching construction for $Y_{z,\OOO'}$ and the identification $\OOO'\cong \GG/\BB$,
the fibers ${Y_{z,\OOO'}}|_{0}$ and ${Y_{z,\OOO'}}|_{\infty}$ are also identified with $\GG/\BB$. So the $\TT$-module structure
of the their tangent spaces is given by (\ref{eq:tangent-roots}). Lemma \ref{lem:torus-action} implies that
${Y_{z,\OOO'}}|_{0}$ and ${Y_{z,\OOO'}}|_{\infty}$
are $\CC^*\times \TT$-spaces. It follows from Lemma \ref{lem:bundle-action} that for the $\CC^*\times \TT$-module structure on $\GG\times_{\BB} (\glie_\CC/\blie)$
induced by (\ref{eq:tangent-roots}) each line bundle $\LLL_\alpha$ is a $\CC^*\times \TT$-submodule, and that at fixed
points we obtain the identification of $\CC^*\times \TT$-modules:
\[{\LLL_{\alpha}}|_{w(0)}\cong {L_\alpha}|_{w(0)},\quad {\LLL_{\alpha}}|_{w(\infty)}\cong {L_\alpha}|_{w(\infty)}\]
Therefore by (\ref{eq:fibermodule}) in Proposition \ref{pro:rootlocalization} we obtain that as $\CC^*\times \TT$-modules:
\[T_{w(0)}(Y_{z,\OOO'}|_0)\cong \bigoplus_{\alpha\in \Delta^-} \CC_{{\pi_{2}}^*(w\cdot \alpha)},\quad
 T_{w(\infty)}(Y_{z,\OOO'}|_\infty)\cong \bigoplus_{\alpha\in \Delta^-} \CC_{\Phi_z^*(w\cdot \alpha)}.
 \]
The $\CC^*\times \TT$-action described in (\ref{eq:actionclutching}) and (\ref{eq:actionclutching-2}) implies
that we have a `horizontal' one dimensional submodules in $T_{w(0)}Y_{z,\OOO'}$ and $T_{w(\infty)}Y_{z,\OOO'}$  complementary
to the vertical tangent spaces and diffeomorphic to $\CC_{\pi_1^*(0)}$ and $\CC_{-\pi_1^*(0)}$, respectively. This completes the proof for
the regular orbit.

Let us fix $q: \OOO'\cong \GG/\BB\to \OOO$ a $\GG$-equivariant projection, so $\OOO\cong \GG/\PP$, $\BB\subset \PP$.
The proof for $\OOO\cong \GG/\PP$ only requires a technical adjustment. We have the identification $T(\GG/\PP)=\GG\times_{\PP} (\glie_\CC/\plie)$,
but the right hand side cannot be written as a sum of line bundles associated to roots  $\alpha\in \Delta^-\backslash \Delta_J^-$, as $\alpha$ may not be $W_P$-
invariant. However, the $\TT$-equivariant  projection $\GG/\BB\to \GG/\PP$  furnishes a canonical identification of $\TT$-modules
(see e.g. \cite[Proposition 17]{Tu}):
\[q^*T(\GG/\PP)\cong \GG\times_\BB (\glie_\CC/\plie).\]
Since one in the right hand side is isomorphic to $\bigoplus_{\alpha\in \Delta^-\backslash\Delta_J^-} \LLL_\alpha$ and since we can identify
as $\TT$-modules $T_{w\PP}(\GG/\PP)$ with $q^*T(\GG/\PP)|_{w\BB}$, the result follows for $\GG/\PP$.

Passing from the  orbit case to $Y_{z,\OOO}$ is done as for the regular orbit.

\end{proof}

Now we are ready to discuss the generalized Schubert decomposition for $Y_{z,\OOO}$:

\begin{definition}\label{def:Bruhat} Let $Y_{z,\OOO}$ be the holomorphic bundle determined by
the orbit $\OOO\cong \GG/\PP$ and  $z\in \Lambda_w^\vee$. Then for each $w\in W^P$  we define:
\begin{enumerate}
 \item The (vertical) Bruhat cell
 \[B_w:=B_w\subset \GG/\PP\cong   {Y_{z,\OOO}}|_{ 0}\subset Y_{z,\OOO}.\]
 \item The (fibered) Bruhat cell
 \[bB_w:=\CC\times B_w\subset \CC\times \GG/\PP\subset Y_{z,\OOO},\]
 where we are using the chart centered at $\infty\in \PP^1$.
\end{enumerate}

 We define the Schubert varieties  $X_w$ and $bX_w$ to be the closure of the corresponding Bruhat cells. By construction $bX_w$ has the stratification:
 \[bX_w=\bigsqcup_{w'\leq w}B_{w'}\sqcup \bigsqcup_{w''<w}bB_{w''}\]

\end{definition}

\begin{prop} The Bruhat cells of $Y_{z,\OOO}$
define a holomorphic stratification whose closure --- the Schubert varieties --- define a CW complex decomposition
\end{prop}
\begin{proof}
  By Proposition \ref{pro:linaction} the action of $\CC^*\times \TT$ on $Y_{z,\OOO}$ is algebraic.
  Results of Bialinicky-Birula on algebraic actions of tori (see e.g \cite[II,\S 4.1]{BCG}) imply that
  any one dimensional subtorus $\CC^*\subset \CC^*\times \TT$ gives rise to a (minus) Bialinicky-Birula decomposition of $Y_{z,\OOO}$, and that
  one can always find one dimensional subtori whose fixed point set is $Y^{\CC^*\times \TT}_{z,\OOO}$. In that case the Bialinicky-Birula decomposition
  has the following properties:
  \begin{enumerate}
   \item[(i)] There is one cell for each fixed point.
   \item[(ii)] Each cell is characterized as the $\CC^*\times \TT$-invariant submanifold of $Y_{z,\OOO}$ containing the fixed point and whose tangent
  space at the fixed point is the submodule of negative weights.
  \item[(iii)] The cells are biholomorphic to affine spaces and their closures define a CW-complex decomposition of $Y_{z,\OOO}$.
  \end{enumerate}

  We start by noting that by Definition \ref{def:Bruhat} there is one Bruhat cell containing each fixed point, and that such cell is a $\CC^*\times \TT$-invariant
  submanifold.
  Therefore, we just need to find a subtorus $\CC^*\subset \CC^*\times \TT$ whose fixed point set is $Y^{\CC^*\times \TT}_{z,\OOO}$,
   and such for each fixed point its submodule of negative weights is
  exactly the tangent space of the corresponding Bruhat cell.

  The Lie algebra of the maximal compact subgroup (torus) of $\CC^*\times \TT$ is
  $i\RR\oplus \tlie\subset\CC\oplus \hlie$.
  We pick $x\in \tlie$ in the interior of the negative Weyl chamber generating a compact 1-dimensional subgroup of $T$. We define $\CC^*$
  to be the subtorus whose Lie algebra contains $i\oplus x$.

  For the action on $\GG/\PP$ of the 1-dimensional subtorus of $\TT$ determined by $x$
  we have (see \cite[II,Example 4.2]{BCG}\footnote{The result
  in \cite[II,Example 4.2]{BCG}
  is  stated just for $\GG/\BB$ (and for the plus Bialinicky-Birula decomposition). The general case follows from how the Bruhat decompositions behave with respect to the $\TT$-equivariant
  submersion $\GG/\BB\to \GG/\PP$}):
   \begin{itemize}
   \item the set of fixed points is  $(\GG/\PP)^\TT$, i.e., $w\PP$, $w\in W^P$;
   \item the negative submodule at $w\PP$ is exactly $T_{w\PP}B_w$.
  \end{itemize}
  This, together with (\ref{eq:actionclutching}), (\ref{eq:actionclutching-2}) in Lemma
  \ref{lem:torus-action} and (\ref{eq:normalmodule0}) in Proposition \ref{pro:tangentspace}, implies
  that the fixed point set of the $\CC^*$-action defined by $i\oplus x$ is
  $Y^{\CC^*\times \TT}_{z,\OOO}$ ($w(0),w(\infty)$, $w\in W^P$),
  and that at each fixed point the negative submodule is the tangent space to
  the corresponding Bruhat cell.
\end{proof}

\subsection{Generalized Chevalley formula}
Let  $\OOO\cong \GG/\PP$ be a coadjoint orbit. The classical Chevalley formula computes
 for $\zeta \in \Lambda_w^{W_P}$ and $w\in W^P$ the cap product of the integral cohomology class $\kappa(\zeta)$
with the integral homology class $[X_w]$ in terms of information on the
Bruhat order of the Weyl group (see for example \cite[\S 4, Proposition 3]{BGG}):
\begin{equation}\label{eq:capclass}
  \kappa(\zeta)\cap [X_w]=
\sum_{w'\overset{s}{\rightarrow}w,\,w'\in W^P} (w'\cdot \zeta)(h_{\alpha})[X_{w'}],
\end{equation}
where the expression `$w'\overset{s}{\rightarrow}w,\,w'\in W^P$' below the summation
sign means (here and everywhere else in the paper)
that we sum over all reflections $s=s_{\alpha}\in W$ such that:
\begin{enumerate}
\item $\alpha$ is a positive root,
\item $w':=sw\in W^P$,
\item $l(w')=l(w)-1$,
\end{enumerate}
where $l$ is the length function on $W$, and  $h_\alpha$ denotes the coroot associated to $\alpha$, which recall is defined as
\begin{equation}
\label{eq:def-h}
h_\alpha:=2\frac{\langle \alpha,\cdot\rangle}{\langle \alpha,\alpha\rangle}.
\end{equation}

We present below a generalization to our setting (the sign change is because in our convention
to define a line bundle out of a character is the opposite as in \cite{BGG}):

\begin{prop}\label{pro:cap} Let $Y_{z,\OOO}$ be the bundle defined
by $z\in \Lambda_w^\vee$ and the coadjoint orbit $\OOO\cong \GG/\PP$.
Then for any  $\zeta\in  \Lambda_w^{W_P}$ and
any $w\in W^P$ we have the following cap product formula for the class $\kappa(\zeta)$
and the integral class $[bXw]$:
\begin{equation}\label{eq:cap}
\kappa(\zeta)\cap [bX_w]=-(w\cdot\zeta)(z)[X_w]-
\sum_{w'\overset{s}{\rightarrow}w,\,w'\in W^P} (w'\cdot \zeta)(h_{\alpha})[bX_{w'}].
\end{equation}

\end{prop}
\begin{remark}\label{rem:rational -vs-integral}
 Observe that if $\zeta$ is a root then $\kappa(\zeta)$ is  the integral class $c_1(L_\zeta)$,
but if $\zeta$ is a weight then $\kappa(\zeta)$ may be a non-integral rational class.
\end{remark}

\begin{proof}
The proof for the orbit $\OOO\cong \GG/\PP$  \cite[\S 4, Proposition 3]{BGG} goes by constructing an appropriate projective embedding of $\OOO$ w.r.t. which
$\LLL_\zeta^\vee=\LLL_{-\zeta}$
is the restriction of the hyperplane bundle\footnote{Note the difference with \cite{BGG} due to the different convention for the definition of $\LLL_\zeta$.}.
Chevalley's formula is obtained by looking at a natural holomorphic section
of ${\LLL_{-\zeta}}|_{X_w}$ whose divisor of zeroes can easily
be related to the stratification
of $X_w$ by Bruhat cells. A crucial fact is that
the singularities of the Schubert variety $X_w$ have complex codimension
at least 2, which implies the existence of a Poincar\'e duality isomorphism $H_{\dim_{\RR}X_w-2}(X_w)\simeq H^2(X_w)$ identifying the first Chern class of $\LLL_{-\zeta}|_{X_w}$ 
and the homology class represented by the vanishing locus of a holomorphic section intersected with $X_w$.
Our proof is very similar, with the only difference that in our bundle setting we will
have to consider a natural  meromorphic (not necessarily holomorphic) section of a certain line bundle over $bX_w$.

It will be convenient for our argument to work with the regular coadjoint orbit $\OOO'\cong \GG/\BB$, $\BB\subset \PP$. So we consider the submersion
$$q\colon Y_{z,\OOO'}\to Y_{z,\OOO}$$
and we denote for every $w\in W$ by
$$[X_w']\in H_{2l(w)}(\OOO';\ZZ),\qquad [bX_w']\in H_{2l(w)+2}(Y_{z,\OOO'};\ZZ)$$
the homology classes associated to $w$.
A straightforward extension of Corollary 5.3 in \cite{BGG}
implies that:
\begin{equation}\label{eq:pushhom}
 q_*[bX'_w]=\begin{cases}
        [bX_w]& \mathrm{if}\,\, w\in W^P\\
0 & \mathrm{if}\,\, w\notin W^P
           \end{cases},
           \qquad
 q_*[X'_w]=\begin{cases}
        [X_w]& \mathrm{if}\,\, w\in W^P\\
0& \mathrm{if}\,\, w\notin W^P
           \end{cases}
\end{equation}

We shall now compute $q^*\kappa(\zeta)\cap [bX'_w]\in H_{2l(w)}(Y_{z,\OOO'};\QQ)$ for an arbitrary
$w\in W$.

Because
(\ref{eq:cap}) is linear on $\zeta$, we may assume that $\zeta$ is a root and that it is a regular dominant weight.
Let us denote by $K:\GG\rightarrow \GL(V)$
 the representation of $\GG$ with maximal
 weight $\zeta$. Because $\zeta$ is regular the $\GG$-orbit of  the line $[V_\zeta]$ corresponding to the 1-dimensional
eigenspace of $\zeta$ defines an embedding  $\OOO'\hookrightarrow \mathbb{P}(V)$. Moreover, the restriction of
  the hyperplane bundle $\OOO_{\PP(V)}(1)$ to (the image of) $\OOO'$ can be identified with
  $\LLL_{-\zeta}$  \cite[\S 4, Proposition 3]{BGG}:
  \[\OOO_{\PP(V)}(1)|_{\OOO'}\equiv\GG\times_{\BB}V_\zeta=\LLL_{-\zeta}.\]

Next, we construct a holomorphic vector bundle $V_z\to \PP^1$ gluing two copies
 of $\CC\times V$ as in (\ref{eq:clutching-vector}) in  Proposition \ref{pro:linaction}.
 If we let $\PP(V_z)$ denote its projectivization, then we obtain an embedding
 \[Y_{z,\OOO'}\hookrightarrow \PP(V_z),\]
 so that the restriction of $\OOO_{\PP(V_z)}(1)$ to (the image of) $Y_{z,\OOO'}$
 is identified with $L_{-\zeta}$.

We define a meromorphic section $\tau_{w}$ of $\OOO_{\PP(V_z)}(1)$ as follows: let $v_{\zeta_i}$, $i\in I$, be a basis of $V$ of eigenvectors.
On the domain of the chart of $L_{-\zeta}$ centered at $\infty\in \PP^1$ the
section is given by restricting the 1-form $v_{w\cdot \zeta}^*\in V^*$:
\[ (u,[v])\mapsto {v_{w\cdot\zeta}^*|}_{\CC v},\quad (u,v)\in \CC\times V.\]
There is just one choice of meromorphic section in the chart centered at $0\in \PP^1$ so that
it agrees with the above holomorphic section in the overlap:
\begin{equation}\label{eq:secoverinfty}(u,[v])\mapsto {u^{\w\cdot \zeta(z)}v_{w\cdot \zeta}^*|}_{\CC v},\quad (u,v)\in \CC\times V.
 \end{equation}

To describe the divisor of zeros and poles of ${\tau_w}|_{bX_w}$ we first restrict our attention to the chart centered at $\infty$. As $\tau_w$  is holomorphic there,
it has no poles. For each fixed $u\in \CC$
the restriction of $\tau_w$ to $\{u\}\times X_w$
is non-vanishing exactly on vectors over the Bruhat cell $B_w$ and its divisor of
zeroes is \cite[\S 4, Proposition 3]{BGG}:
\[\sum_{w'\overset{s}{\rightarrow}w} (w'\cdot\zeta)(h_{\alpha})[X_{w'}'],\]
where the expression "$w'\overset{s}{\rightarrow}w$" below the summation
sign means  that we sum over all reflections $s=s_{\alpha}\in W$ such that:
(1) $\alpha$ is a positive root, and (2)
$w':=sw$ satisfies $l(w')=l(w)-1$.
Therefore it follows that the divisor of zeros of ${\tau_w}|_{bX_w}$ contains:
\[\sum_{w'\overset{s}{\rightarrow}w} (w'\cdot\zeta)(h_{\alpha})[bX_{w'}'].\]
By (\ref{eq:secoverinfty}) the remaining part of the divisor of zeroes and poles of ${\tau_w}|_{bX_w}$ is:
\[(w\cdot\zeta)(z)[X_w'].\]
Therefore we conclude
\begin{equation}
\label{eq:Chev-red}
q^*\kappa(\zeta)\cap [bX_w']=-(w\cdot \zeta)(z)[X_w']-
\sum_{w'\overset{s}{\rightarrow}w} (w'\cdot \zeta)(h_{\alpha})[bX_{w'}'].
\end{equation}
We now translate this into a formula for the general coadjoint orbit $\OOO$. Take
some $w\in W^P$. By (\ref{eq:pushhom}) we have $[bX_w]=q_*[bX_w']$, so by
the product formula we have:
\begin{align*}
\kappa(\zeta)\cap [bX_w] &= \kappa(\zeta)\cap q_*[bX_w']=q_*(q^*\kappa(\zeta)\cap [bX_w']) \\
&=q_*\left(-(w\cdot \zeta)(z)[X_w']-
\sum_{w'\overset{s}{\rightarrow}w} (w'\cdot \zeta)(h_{\alpha})[bX_{w'}']\right) \\
&=-(w\cdot \zeta)(z)q_*[X_w']-
\sum_{w'\overset{s}{\rightarrow}w} (w'\cdot \zeta)(h_{\alpha})q_*[bX_{w'}'] \\
&=-(w\cdot\zeta)(z)[X_w]-
\sum_{w'\overset{s}{\rightarrow}w,\,w'\in W^P} (w'\cdot \zeta)(h_{\alpha})[bX_{w'}],
\end{align*}
where in the last equality we have used again (\ref{eq:pushhom}).
\end{proof}

\begin{remark}\label{rem:order}
If $f\in \Sym_\QQ(\Lambda_w)$ represents a cohomology class in $\OOO\cong \GG/\PP$, then
is possible to apply (\ref{eq:cap}) recursively to cap  $\kappa(f)$  against a (fibered) Schubert class.
Note that such a polynomial will be $W_P$-invariant in the quotient ring $\Sym_\QQ(\Lambda_w)/I_+$, so it may
not be itself $W_P$-invariant, and therefore its monomials need not be $W_P$-invariant. This implies
that we have to do the cap product of each monomial
inside a regular orbit. The price to pay is the use the Bruhat order of $W$ and not the much simpler order induced on $W/W_P$.
\end{remark}

A first interesting consequence of the generalized Chevalley formula is the proof of our main theorem
in the case of regular orbits:

 \begin{proof}[Proof of Corollary \ref{cor:regular}]
 By Proposition \ref{pro:nonint} it is enough to show that if $\OOO'$ is a regular orbit of $G$ and $\gamma\subset G$ is a non-contractible loop, then
  the characteristic isomorphism for $Y_{\gamma,\OOO}$ is not integral.

  Let us fix a $G$-equivariant identification $\OOO'\cong G/T$ and let us pick
  $z\in \Lambda_w^\vee$ so that the loop $\gamma_z$ is homotopic to $\gamma$. Because $\gamma$ is not contractible $z\notin \Lambda^\vee_r$.
  This is equivalent to the existence of a weight $\zeta\in \Lambda_w$
 so that $\zeta(z)\notin \ZZ$.

 The Borel isomorphism identifies $H^2(\OOO';\ZZ)$ with the lattice of weights $\Lambda_w$, so $\kappa$ will be non-integral if
 $\kappa(\zeta)\in H^2(Y_{z,\OOO'};\QQ)$
 is not an integral class.  The only fibered degree two Schubert class is $[bX_e]$. If we apply
 the generalized Chevalley formula (\ref{eq:cap}) to $[bX_e]$ and $\zeta$, then we obtain:
 \[\int_{[bX_e]}\kappa(\zeta)=-\zeta(z)\notin \ZZ,\]
 which proves the corollary.
   \end{proof}

It is also a consequence of the generalized Chevalley formula that to prove the non-integrality of the characteristic isomorphism
is it enough to consider the case of simple groups:
\begin{proof}[Proof of Proposition \ref{pro:simplegroups}]
Let us assume that the characteristic isomorphism is not integral
for every coadjoint orbit of a simple group and for every non-contractible loop in the group.

Let $\OOO\cong G/P$, $T\subset P$, be a coadjoint orbit of be an (adjoint) semisimple Lie group $G$,  and let $z\in \tlie$ be a coweight with is not a coroot.
We can assume without loss of generality that $G=G_1\times G_2$ ($T=T_1\times T_2$),  with $G_1$ simple, that $z=z_1+z_2\in \tlie=\tlie_1\oplus \tlie_2$, with $z_1$
a coweight which is not a co-root, and that $\OOO=\OOO_1\times \OOO_2$ is a product of coadjoint orbits on each factor group,
$\OOO_1\cong G_1/P_1$, $T_1\subset P_1$.

As $z_1$ is not a co-root, by assumption the characteristic isomorphism is not integral for $Y_{z_1,\OOO_1}$. Therefore there exist  $w_1\in {W_1}^{P_1}$
and $f_1\in \QQ[\tlie_1]$ homogeneous with $\kappa(f_1)\in H^*(\OOO_1;\ZZ)$, such that:
\begin{equation}\label{eq:contractionorbit}
 \int_{[bX_{w_1}]}\kappa(f_1)\notin \ZZ.
 \end{equation}

If we regard $f_1\in \QQ[\tlie]$, then we claim that $\kappa(f_1)\in H^\bullet(\OOO_1\times \OOO_2;\QQ)$  is an integral class.
The reason is that the classical Chevalley formula for coadjoint orbits (\ref{eq:capclass})  and the fact
that the Bruhat order on $W=W_1\times W_2$ is the product order  imply:
\[\int_{[X_{w'}]}\kappa(f_1)=\int_{[X_{w'_1}]}\kappa(f_1),\, w'\in W,\, w'=w'_1w_2'.\]

Let us also regard $bX_{w_1}$ as a fibered Schubert variety of $Y_{z,\OOO}$. By the generalized Chevalley formula we also get that the pairing
$\int_{[bX_{w_1}]}\kappa(f_1)$ in $Y_{z,\OOO}$ equals (\ref{eq:contractionorbit}) (this is because
for any monomial $\zeta_1\subset \tlie^*_1$ of $f_1$ we have $\zeta_1(z_1+z_2)=\zeta_1(z_1)$).
Thus, we conclude  that the characteristic isomorphism  for $Y_{z,\OOO}$ is not integral.
\end{proof}

\subsection{Smooth Schubert varieties and Dynkin subdiagrams}
Our goal is finding appropriate polynomials and Schubert varieties so that the
the cap products
\begin{equation}\label{eq:cap-product-Schubert}
\kappa(f)\cap [bX_w]:=\int_{[bX_w]}\kappa(f)
 \end{equation}
allow us to prove the non-integrality of the characteristic isomorphism. Schubert varieties are in general
non-smooth $\CC^*\times \TT$-spaces. However, for smooth ones the cap product (\ref{eq:cap-product-Schubert})
is actually an integral that can be computed by localization.
To that end we will be describing
a collection of smooth Schubert varieties in $Y_{z,\OOO}$ (in $\OOO$) for which the module
structure of the tangent spaces at the fixed points can be easily determined, and that will be enough
for our purposes \footnote{It is very likely that the smoothness of the varieties we will describe is a
consequence of the known smoothness criteria
(see e.g. \cite{BP}), although these often apply just to Schubert varieties inside regular orbits.}.

\begin{definition}\label{def:projgraph}
Let $\GG$ be a simple (connected) Lie group, let $\Pi$ be a set of simple roots and
let $J$ be a subset of $\Pi$. Denote by $\PP$ the parabolic subgroup determined by $J$.
We say that a subdiagram $\Gamma\subset \Gamma_\Pi$
 is projective (relative to $J$) if $\Gamma$ is isomorphic to $A_k$, for some $k\in \NN$ (in particular
 all roots have the same norm so there are no multiple edges), and the only root of $\Gamma$ not in $J$ is in the
boundary of $\Gamma$.
\end{definition}

We collect the following list of projective subdiagrams for latter use.

\begin{example}\label{ex:subgraphs} We refer to Dynkin diagrams of classical simple Lie groups as appearing in \cite[Page 293]{OV}.
\begin{enumerate}
 \item For any $\Gamma_\Pi$, $J\subset \Gamma_\Pi$, and $\alpha \notin J$  the subdiagram $\{\alpha\}$ is projective.
\item If $\Gamma_\Pi=B_n$ and  $J=\Gamma_\Pi\backslash \{\alpha_r\}$, $r<n$, then any connected subdiagram  having
as leftmost node the $r$-th root and as rightmost node $\alpha_l$, $l<n$, is projective.
\item  If $\Gamma_\Pi=D_n$ and $J=\Gamma_\Pi\backslash \{\alpha_r\}$, $r<n$,  then projective subdiagrams
 are those connected subdiagrams with bivalent nodes having the $r$-th root at their boundary.
 \end{enumerate}
\end{example}

Projective subdiagrams determine smooth Schubert varieties with simple module structure at fixed points:

\begin{prop}\label{pro:graphword} Using the notation of Definition \ref{def:projgraph}, let
$\Gamma\subset \Pi$ be a projective subdiagram relative to $J$. Let
$\alpha_{(1)},\dots,\alpha_{(k)}$ be the only ordering of the roots of $\Gamma$ so that
$\alpha_{(j)}$ and $\alpha_{(j+1)}$, $1\leq j\leq k-1$, are joined
by an edge and $\alpha_{(1)}\notin J$.
If we define $w_\Gamma:=s_{(k)}\circ \cdots \circ s_{(1)}$, where $s_{(j)}$ is the reflection
 defined by $\alpha_{(i)}$, then the following holds:
\begin{enumerate}
 \item The Schubert variety $X_{w_\Gamma}\subset \GG/\PP\cong \OOO$ is biholomorphic to projective space $\PP^k$.
 \item The only elements of $W^P$ smaller or equal than  $w_\Gamma$ are:
 \[w_{(j)}:=s_{(j)}\circ \cdots \circ s_{(1)},\,j=0,\dots,k,\,\, (w_{(0)}=e).\]
\item For any coweight $z\in\Lambda_w^\vee$ the (smooth) Schubert variety $bX_{w_\Gamma}\subset Y_{z,\OOO}$ has fixed points
 \[w_{(j)}(0),w_{(j)}(\infty),\,  0\leq j\leq k,\]
 (see equation (\ref{eq:fibered-fixed--points}))
 with $\CC^*\times \TT$-module structure on tangent spaces isomorphic to:
 \begin{align}\label{eq:schubertproj0}
T_{w_{(j)}(0)}bX_{w_\Gamma} &\cong  \CC_{\pi_1^*(0)}\oplus\bigoplus_{h=1}^j\CC_{\pi_2^*(\alpha_{(h)}+\cdots+\alpha_{(j)})}
\oplus \bigoplus_{h=j+1}^k\CC_{-\pi_2^*(\alpha_{(j+1)}+\cdots+\alpha_{(h)})},\\ \nonumber
T_{w_{(j)}(\infty)}bX_{w_\Gamma} &\cong \CC_{-\pi_1^*(0)}\oplus\bigoplus_{h=1}^j\CC_{\Phi_z^*(\alpha_{(h)}+\cdots+\alpha_{(j)})}
\oplus \bigoplus_{h=j+1}^k\CC_{-\Phi_z^*(\alpha_{(j+1)}+\cdots+\alpha_{(h)})}.
\end{align}

\end{enumerate}

\end{prop}

\begin{proof}

For an arbitrary Dynkin diagram any word in $k$ distinct simple roots
has length $k$ \cite{Ti} (see also Lemma \ref{lem:tits}). We provide the argument for $w_\Gamma$ because it contains
some useful computations that we will need to prove other items of this proposition. What makes the computations for $w_\Gamma$ simpler is that this is
a word in consecutive roots and the identities:
\begin{equation}\label{eq:equal-norm}2\frac{\langle \alpha_{(j+1)},\alpha_{(j)}\rangle}{\langle \alpha_{(j+1)},\alpha_{(j+1)}\rangle}=-1.
 \end{equation}
By \cite[Lemma 2.2]{BGG} we need to find $k$ positive roots which $\w_{\Gamma}^{-1}$ sends to negative ones.
Let  $\gamma_i:=s_{(k)}\circ\cdots\circ s_{(i+1)}(\alpha_{(i)})$
$1\leq i\leq k$.  An induction argument shows that:
\begin{equation}\label{eq:negroots0}
 \gamma_i=\sum_{h=i}^k\alpha_{(h)},
\end{equation}
and hence $\gamma_i$, $1\leq j\leq k$, are distinct roots. An analogous calculation shows
\[
 w^{-1}_\Gamma\cdot \gamma_i=-\left(\sum_{h=1}^{i}\alpha_{(h)}\right),
\]
and thus $w_\Gamma$ has length $k$.

Next, we want to identify $X_{w_\Gamma}$ with $\PP^k$ and to describe its partition into Bruhat cells.
Consider the weight
\[
\mu(h_\alpha)=\frac{2\langle \alpha,\mu\rangle}{\langle \alpha,\alpha\rangle}:=\begin{cases}0 & \mathrm{if}\,\, \alpha\in J\\
                                                                   1& \mathrm{if}\,\,\alpha\in \Pi\backslash J
    \end{cases}
    \]
(see (\ref{eq:def-h}) for the definition of the coroot $h_{\alpha}$).
The representation $K:\GG\to \GL(V)$ with maximal weight $\mu$ embeds $\GG/\PP$  in $\PP(V)$
as the orbit of the eigenline $[V_\mu]$.
We fix a non-zero $v\in V_\mu$  and consider the vectors $w_{(j)}\cdot v\in V_{w_{(j)}\cdot \mu}$, $j=0,\dots,k$.
Because $\alpha_{(1)}\notin J$ we get:
\begin{equation}\label{eq:maxweightact}
 w_{(j)}\cdot \mu=\mu-\sum_{h=1}^j\alpha_{(h)}.
\end{equation}
As these are different weights, the vectors $w_{(j)}\cdot v$, $0\leq j\leq k$, are linearly independent and span
$V_{w_\Gamma}:=\oplus_{j=0}^k V_{w_{(j)}\cdot \mu}$.
We claim that we have the disjoint union
\begin{equation}\label{eq:bruhat-for-projective}
 \PP(\oplus_{h=0}^{j}V_{w_h\cdot \mu})=\BB\cdot [\w_{(j)}\cdot v]\sqcup \PP(\oplus_{h=0}^{j-1} V_{w_h\cdot \mu}).
 \end{equation}
Note that this would automatically imply item (2) in the proposition.

As the complex torus $\TT\subset \BB$ stabilizes  $[\w_{(j)}\cdot v]$, we only need to look at the action
of the nilpotent factor $\BB/\TT$. A nilpotent group
is diffeomorphic to the product of the 1-parameter subgroups associated to any basis of its Lie algebra.
Regarding the infinitesimal action of the nilpotent Lie algebra,
by \cite[Proposition 1.1]{BGG} it is enough to consider the action of positive roots sent to negative ones by $w_{(j)}$. It follows from (\ref{eq:negroots0})
that these roots are \[\gamma^j_i=\sum_{h=i}^j\alpha_{(h)},\, 0\leq i\leq j.\]
As we have the identity $\gamma_{i}^j+w_{(j)}\cdot\mu=w_{(i-1)}\cdot \mu$, we obtain
\[\gamma_{i}^j\cdot V_{w_{(j)}\cdot\mu}= V_{w_{(i-1)}\cdot \mu}.\]
The decomposition in (\ref{eq:bruhat-for-projective}) follows immediately by integrating the infinitesimal
action of these root spaces.

By  (\ref{eq:fibered-fixed--points}) the set of fixed points  $X_{w_\Gamma}^{\TT}$ is $[w_{(j)}\cdot v]$, $0\leq j\leq k$. To describe the $\TT$-module structure of their tangent spaces we note
that the basis $w_{(j)}\cdot v\in V_{w_{(j)}\cdot \mu}$, $j=0,\dots,k$  furnishes an isomorphism
of $\TT$-modules:
\begin{equation}\label{eq:projtmod}
\bigoplus_{j=0}^k\CC_{w_{(j)}\cdot \mu}\cong \bigoplus_{j=0}^k V_{w_{(j)}\cdot \mu},
 \end{equation}
and these implies:
\begin{equation}\label{eq:mod-orbit}T_{w_{(j)}}X_{w_\Gamma}\cong \bigoplus_{h=1}^j\CC_{\alpha_{(h)}+\cdots+\alpha_{(j)}}
\oplus \bigoplus_{h=j+1}^k\CC_{-(\alpha_{(j+1)}+\cdots+\alpha_{(h)})}
\end{equation}

The formulas
for the $\CC^*\times \TT$-module structure at $T_{w_{(j)}(0)}bX_{w_\Gamma},T_{w_{(j)}(\infty)}bX_{w_\Gamma}\subset Y_{z,\OOO}$ described
in (\ref{eq:schubertproj0}), are easily deduced from (\ref{eq:mod-orbit})
 by working with the bundle
over $\PP^1$ associated to the representation $K$ and the $\CC^*$-action determined by the cocharacter $\phi_{-z}$.

\end{proof}
\begin{remark} The only property of $A_k$ used in the proof of Proposition \ref{pro:graphword} is equation
(\ref{eq:equal-norm}). Thus if $\Gamma$ is a subdiagram isomorphic to $C_k$ with
root ordering $\alpha_{(1)},\dots,\alpha_{(k)}$ as in Proposition \ref{pro:graphword},
then the same conclusions hold.
\end{remark}

\section{Grassmann and isotropic Grassmann fibrations}\label{sec:unitary}

The minimal coadjoint orbits of the projective unitary group $\PU(n+1)$ and of the compact symplectic group $\Sp(2n)$ can be identified with
the Grassmannians and isotropic Grassmannians, respectively: $\mathrm{Gr}_r(\CC^{n+1})$, $\mathrm{IGr}_r(\CC^{2n})$, $1\leq r\leq n$.
The goal of this section is to prove the following result:

\begin{theorem}\label{thm:ngrass} Let $G$ be either $\PU(n+1)$ or $\Sp(n)/Z$ and let $\gamma\subset G$ be a non-contractible loop based at
the identity. Let $Y_{\gamma,r}$
denote the bundle coming from the clutching construction applied to the minimal coadjoint orbit $\mathrm{Gr}_r(\CC^{n+1})$ or $\mathrm{IGr}_r(\CC^{2n})$, $1\leq r\leq n+1$.

The characteristic isomorphism
\[\kappa:H^\bullet(Y_{e,r};\QQ)\to H^\bullet(Y_{\gamma,r};\QQ)\]
 is not integral.
\end{theorem}
As we will see, to detect the non-integrality of $\kappa$ will be enough to consider cohomology classes
of degree less or equal than four.

Observe that Theorem \ref{thm:ngrass} for $G=\PU(n+1)$  implies Theorem \ref{thm:injdiffflags} in the Introduction.

We need a precise description of the quotient of coweight and coroot lattices,  the weight lattice, and the action of the Weyl group
on the weight lattice, for both $\PU(n+1)$ and $\Sp(n)$. This material is drawn from
 \cite[Pages 292--289]{OV}\footnote{In some cases our generators for $\Lambda_w^\vee/\Lambda_r^\vee$ will differ
 from those in \cite{OV}}. It will be presented simultaneously for both groups with common notation
 that hopefully will lead to no confusion.
 We warn the reader that when the data for both group differs the information for
$\PU(n+1)$ will come in the first place.

Let $T_{\U}\subset \U(n+1)$ be determinant one diagonal matrices with entries in $S^1$ and let $T\subset \PU(n+1)$ be the
image of $T_{\U}$ under the projection $\U(n+1)\to \PU(n+1)$. Let
 $e_i\in \tlie_{\U}$, $1\leq i\leq n+1$, denote the matrix with 1 in the position $(ii)$ and 0 elsewhere. We denote by
\begin{equation}\label{eq:generators-lattice-1}
 \epsilon_1,\dots,\epsilon_{n+1}\in \tlie^*_{\U}
\end{equation}
the basis dual to $e_1,\dots,e_{n+1}$. This
a basis of the rational lattice of $\tlie^*_{\U}$  so that
$\tlie^*$ is the quotient by the line spanned by $\sum_{k=1}^{n+1}\epsilon_k=0$. A multiple of the Killing form on $\tlie^*$
is induced by the following degenerate pairing on $\tlie^*_{\U}$:
\begin{equation}
\langle \epsilon_i,\epsilon_j\rangle=\begin{cases} \tfrac{n}{n+1} &\mathrm{if}\,\, i=j\\
                                         -\tfrac{1}{n+1} &\mathrm{if}\,\, i\neq j
                                       \end{cases}
  \end{equation}

We realize $\Sp(n)$ as the intersection $\SU(2n)\cap \mathrm{Sp}(2n,\CC)$, where the latter is the group of
symmetries of the standard complex symplectic
form on $\CC^{2n}$.  Our maximal torus $T$ consists of diagonal matrices
with entries in $S^1$, i.e. elements of the form
\[\theta_1e_1+\cdots\theta_ne_n+\theta_1^{-1}e_{n+1}+\cdots +\theta_n^{-1}e_{2n},\]
where $\theta_i\in S^1$ and $e_i$ is the matrix whose only non-zero entry is 1 in the position $(ii)$.
Let
\begin{equation}\label{eq:generators-lattice-2}
\epsilon_1,\dots,\epsilon_n \in \tlie^*
\end{equation}
be the basis dual to $e_1,\dots,e_n$. This is an orthonormal basis for a multiple of the Killing form.

We take as sets of simple roots for $\PU(n+1)$ and $\Sp(n)$:
\begin{equation}\label{eq:simple-roots-basis}
 \PU(n+1):\,\alpha_i=\epsilon_i-\epsilon_{i+1},\,i=1,\dots,n,\qquad\qquad \Sp(n):\,\alpha_i=\epsilon_i-\epsilon_{i+1},\,i=1,\dots,n,\,\alpha_n=2\epsilon_n
\end{equation}
For both groups the corresponding fundamental weights are:
\begin{equation}\label{eq:suweightparab}
 \zeta_i=\sum_{k=1}^i\epsilon_i,\,i=1,\dots,n.
\end{equation}
The quotients of coweight and coroot lattices are the cyclic groups
\begin{equation}\label{eq:groupgrass}
 \PU(n+1):\, \Lambda_w^\vee/\Lambda_r^\vee \cong \ZZ_{n+1},\qquad\qquad  \Sp(n):\,\Lambda_w^\vee/\Lambda_r^\vee \cong \ZZ_{2},
 \end{equation}
and we select the generators:
\begin{equation}\label{eq:gen-unitary}
 \PU(n+1):\,  z_0=\frac{1}{n+1}\left(-ne_{n+1}+\sum_{k=1}^n e_k\right),   \qquad\qquad    \Sp(n):\,   z_0=\frac{1}{2}\sum_{k=1}^{n}e_k.
\end{equation}
The Weyl groups for $\PU(n+1)$ and $\Sp(n)$  are symmetric group and the hyperoctahedral group
in the generators (\ref{eq:generators-lattice-1}) and (\ref{eq:generators-lattice-2}), respectively. So far we have denoted
the reflection on a root by $s$. From this point on, we will modify
 our notation for reflections. If a set of generators of
the rational lattice of $\tlie^*$ has been fixed, then notation $\sigma$ will be used
when the simple reflection acts as a transposition on the fixed set of generators.
More specifically,  for $\PU(n+1)$ the Weyl group is generated by the simple reflections $\sigma_1,\dots,\sigma_n$, where the $\sigma_i$ transposes
$\epsilon_i$ and $\epsilon_{i+1}$. The reflections on the simple roots of $\Sp(n)$ are
$\sigma_1,\dots,\sigma_{n-1},s_n$, where $\sigma_i$ transposes
$\epsilon_i$ and $\epsilon_{i+1}$ and
\[s_n(\epsilon_i)=\begin{cases}\epsilon_i &\mathrm{if}\,\, i\neq n\\
                         -\epsilon_n &\mathrm{if}\,\, i=n.
                        \end{cases}
\]

Let  $\PP_r$ be the maximal parabolic subgroup of either $\PGL(n+1,\CC)$ or $\Sp(n,\CC)$
determined by the subset $\Pi\backslash \{\alpha_r\}$, $1\leq r\leq n$. The corresponding
flag variety is either the Grassmannian  $\mathrm{Gr}_r(\CC^{n+1})$, or  the isotropic Grassmannian $\mathrm{I}\mathrm{Gr}_r(\CC^{2n})$ (the Grassmannian of
isotropic subspaces of
dimension $r$ in $\CC^{2n}$ with its standard complex symplectic form).
The corresponding residual Weyl Groups $W_r$ are generated by:
 \begin{equation}\label{eq:residual-Weyl-unit}
\PU(n+1):\,\sigma_1,\dots,\sigma_{r-1},\sigma_{r+1},\dots,\sigma_n,\qquad\qquad \Sp(n):\, \sigma_1,\dots,\sigma_{r-1},\sigma_{r+1},\dots,\sigma_{n-1},s_n.
 \end{equation}

\begin{proof}[Proof of Theorem \ref{thm:ngrass}]
For both $\PU(n+1)$ and $\Sp(n)$ we consider the coweight $z=dz_0$, $d\in \ZZ$.
By (\ref{eq:gen-unitary}) and (\ref{eq:groupgrass})  it is enough to show
that the characteristic isomorphism for $Y_{dz_0,r}$ is not integral if $d$ is not is a
multiple of $n+1$ for $\PU(n+1)$ and  of 2 for $\Sp(n)$.

For both $\PU(n+1)$ and $\Sp(n)$ equation (\ref{eq:residual-Weyl-unit}) implies that
any symmetric polynomial in the variables $\epsilon_1,\dots,\epsilon_r$ corresponds to an integral cohomology class of the
minimal coadjoint orbit (note that if $r<n$, then $s_n$ acts trivially in the variables $\epsilon_1,\dots,\epsilon_r$). Equation (\ref{eq:suweightparab}) implies that
the weight lattice is generated by $\epsilon_1,\dots,\epsilon_n$. Therefore
by  \cite[Theorem 2.1]{To} the integral cohomology up to degree four of
the minimal orbit is generated by the following symmetric polynomials in the variables $\epsilon_1,\dots,\epsilon_r$:
\[c_1=\sum_{i=1}^r\epsilon_i,\quad c_1^2,\quad c_2=\sum_{1\leq i<j\leq r}\epsilon_i\epsilon_j.\]

Applying the generalized Chevalley formula (\ref{eq:cap})  to $c_1$ and $[bX_e]$, and using (\ref{eq:gen-unitary}) we obtain:
\begin{equation}\label{eq:cap-symmetric-one}
\PU(n+1):\,\int_{[bX_e]}\kappa(c_1)=-c_1(z)=-\frac{dr}{n+1},\qquad\qquad \Sp(n):\, \int_{[bX_e]}\kappa(c_1)=-c_1(z)=-\frac{dr}{2}.
 \end{equation}

The subgraph with a single node $\alpha_r$ is projective relative to  $\Pi\backslash \{\alpha_r\}$ (see Example \ref{ex:subgraphs}, (1)),
so $bX_{w_{\alpha_r}}$ is a complex submanifold of the minimal orbit of (complex) dimension two. To cap $\kappa(c_2)$ against $[bX_{w_{\alpha_r}}]$ we will not use localization, but rather the generalized
Chevalley formula again. The reason is that
only $e$ precedes $\alpha_r$ in the Bruhat order of the Weyl group
(see Remark \ref{rem:order} and Lemma \ref{lem:tits}). To apply the generalized Chevalley formula we note that the choice of simple roots  (\ref{eq:simple-roots-basis})
together with the fixed inner products imply the following relation between the $\epsilon_i$'s
and coroots (see (\ref{eq:def-h})):
\begin{equation}\label{eq:evaluation-dual-root}
 \epsilon_i(h_{\alpha_r})=\begin{cases} 0 &\mathrm{if}\,\, i<r\\
                                        1 &\mathrm{if}\,\, i=r.
                          \end{cases}
\end{equation}
If we apply once the generalized Chevalley formula  (\ref{eq:cap}) to $\kappa(c_2)$ and $[bX_{w_{\alpha_r}}]$,
then we get a first summand which is a linear
polynomial in the variables $\epsilon_1,\dots,\epsilon_{r-1}$ multiplying the vertical Schubert variety $[X_{\alpha_r}]$.
When we cap again that linear polynomial against $[X_{\alpha_r}]$,  by equation (\ref{eq:evaluation-dual-root})  the result is zero.  The same argument
 proves that in the second summand in (\ref{eq:cap})  monomials
not containing $\epsilon_r$ do not contribute to the integral. Therefore the integral becomes:
\[\int_{[bX_{w_{\alpha_r}}]}{\kappa(c_2)}=
\epsilon_r(h_{\alpha_r})\int_{[bX_e]}\kappa(\epsilon_1+\cdots +\epsilon_{r-1})=-(\epsilon_1(z)+\cdots +\epsilon_{r-1}(z))\]
Thus by (\ref{eq:gen-unitary}):
\begin{equation}\label{eq:cap-symmetric-two}
\PU(n+1):\,\int_{[bX_{w_{\alpha_r}}]}{\kappa(c_2)}=\frac{d(r-1)}{n+1},\qquad\qquad \Sp(n):\,\int_{[bX_{w_{\alpha_r}}]}{\kappa(c_2)}=-\frac{d(r-1)}{2}.
\end{equation}
Equations (\ref{eq:cap-symmetric-one}) and (\ref{eq:cap-symmetric-two}) for $\PU(n+1)$ imply that
$\frac{dr}{n+1},\frac{d(r-1)}{n+1}\in \ZZ$, and so does their difference $\frac{d}{n+1}$. The same applies to $\Sp(n)$ to show that  $\frac{d}{2}\in\ZZ$,
and this finishes the proof of the theorem.
\end{proof}

We finish this section proving Theorem \ref{thm:injdiffflags} in the Introduction. There, the starting point
is  a rank $n+1$ holomorphic vector bundle  $E\to\PP^1$. The assertion is that for two such vector bundles, there exists a compatible diffeomorphism between
the associated bundles of flags of type
$r=(r_1,\dots,r_l)$ if and only if the degrees of the vector bundles differ by a multiple of $n+1$.

The frame bundle associated to $E$ is classified by a loop $\gamma: S^1\to \GL(n+1,\CC)$. The bundle of flags of type $r$ associated to $E$ comes from the principal bundle
classified by the loop $p(\gamma)$, where $p$ is the natural projection $p:\GL(n+1,\CC)\to \PGL(n+1,\CC)$.

The degree of $E$ is the homotopy class of $\gamma$ with the normalization provided by the determinant map.
More precisely, it coincides with the degree of $\mathrm{det}\circ \gamma:S^1\to S^1$.  Also, the kernel of the map induced by $p$ on fundamental groups
consists of homotopy classes of loops whose degree is a multiple of $n+1$.

Therefore, by Theorem \ref{thm:ngrass} (for $\PU(n+1)$), Proposition \ref{pro:min-orbits} (reduction to minimal orbits) and
Proposition \ref{pro:nonint} (the non-integrality
of $\kappa$ obstructs the existence of compatible isomorphisms),
there is a compatible diffeomorphism between the bundles of flags of type $r$ associated to the rank $n+1$ vector bundles $E,E'$ if and only if the loops $p(\gamma),p(\gamma')$
are homotopic. Equivalently, if the degrees of $E,E'$ differ by a multiple of $n+1$, as we wanted to prove.

\section{Orthogonal Grassmann fibrations I}
\label{s:orthogonal-I}

The minimal coadjoint orbits of the odd dimensional orthogonal group $\SO(2n+1)$ are
the orthogonal Grassmannians $\mathrm{OGr}_r(\CC^{2n+1})$, $1\leq r\leq n$. The
purpose of this section is to prove the following:

\begin{theorem}\label{thm:orthoggrass} Let $\gamma\subset \SO(2n+1)$ be a non-contractible loop based at the identity and let $Y_{\gamma,r}$
denote the bundle coming from the clutching construction applied to the coadjoint orbit $\mathrm{OGr}_r(\CC^{2n+1})$, $1\leq r\leq n$.

The characteristic isomorphism
\[\kappa:H^\bullet(Y_{e,r};\QQ)\to H^\bullet(Y_{\gamma,r};\QQ).\]
is not integral.
\end{theorem}

The analogous theorem  for even dimensional orthogonal groups is treated in the next section,
 as it presents additional technical complications.

The algebraic structure of the orthogonal groups is far more involved that the one of the special unitary and symplectic groups.
For example, the action of the Weyl group  on any basis of the weight lattice is quite different from that of the symmetric group.
Our strategy will boil down to
careful choices of both a polynomial $f$ representing an integral class in the orthogonal
Grassmannian $\mathrm{O}\mathrm{Gr}_r(\CC^{2n+1})$ and
 a projective subdiagram $\Gamma$, so that we can compute by localization the integral $\int_{bX_{w_\Gamma}}\kappa(f)\notin \ZZ$:

\begin{theorem}\label{thm:integralorthoggrass} Let $\gamma\subset \SO(2n+1)$ be a non-contractible loop based at the identity. Choose a set of simple
roots $\Pi=\{\alpha_1,\dots,\alpha_n\}$ with Dynkin diagram $\Gamma_\Pi$. For any $1\leq r\leq n$  there exists
\begin{enumerate}
\item a polynomial $f$ in $\mathrm{Sym}_\QQ(\Lambda_w)$ which by the Borel isomorphism corresponds to a class in
$H^{2(n-r+1)}(\mathrm{O}\mathrm{Gr}_r(\CC^{2n+1});\ZZ)$, and
\item a projective subdiagram $\Gamma$ relative to $\Gamma_\Pi\backslash \{\alpha_r\}$ such that:
\end{enumerate}
\begin{equation}\label{eq:intprojodd}
 \int_{bX_{w_\Gamma}}{\kappa(f)}=(-1)^{n-r+1}\frac{1}{2}.
\end{equation}
\end{theorem}

Of course, Theorem \ref{thm:orthoggrass} is an immediate consequence of Theorem \ref{thm:integralorthoggrass}.

The choice of $f$ will be non-trivial because another manifestation of the complexity of the orthogonal groups is that
the symmetric $\ZZ$-algebra spanned by the weight lattice does not describe the full integral cohomology of the orthogonal Grassmannian \cite[Theorem 2.1]{To}.
And, naturally, a polynomial with fractional coefficients in weights is more likely to detect the non-integrality of the characteristic isomorphism.
Our choice of $f$ will be based on the analysis of the Borel isomorphism up to degree four --- where it detects integral cohomology \cite[Theorem 2.1]{To}.
The subdiagram $\Gamma$ will be an obvious projective subdiagram of $\Gamma_\Pi$ relative to $\Pi\backslash \{\alpha_r\}$ of length $n-r$.

\subsection{An integral cohomology class in the orthogonal Grassmannian}

We shall describe a polynomial in weight variables with non-integral coefficients  whose image by the Borel isomorphism defines an integral cohomology class
in the orthogonal Grassmannian $\mathrm{O}\mathrm{Gr}_r(\CC^{2n+1})$.
Our main ingredients are a presentation of the inverse image by the Borel isomorphism of the integral cohomology of the regular orbit \cite{TW,N2},
together with basic properties of the $\ZZ$-algebra of symmetric polynomials.

We start with a precise description of the quotient of coweight and coroot lattices,  the weight lattice the action of the Weyl group
on the weight lattice $\SO(2n+1)$
 \cite[Pages 292--289]{OV}.

 We realize $\mathrm{SO}(2n+1)$ as the intersection  $\SU(2n+1)\cap \mathrm{SO}(2n+1,\CC)$, where the latter is the group of matrices leaving the
quadratic form $z_1z_{n+1}+\cdots z_nz_{2n}+z_{2n+1}^2$ invariant.
Let  $T$ be the diagonal matrices with entries in $S^1$, i.e., matrices of the form
\[\theta_1e_1+\cdots\theta_ne_n+\theta_1^{-1}e_{n+1}+\cdots +\theta_n^{-1}e_{2n}+e_{2n+1},\]
where  $e_i$ denotes the matrix with 1 in the position $(ii)$ and 0 elsewhere.
Let   $\epsilon_1,\dots,\epsilon_n\in \tlie^*$ be the basis dual to $e_1,\dots,e_n$. This is an orthonormal basis for a multiple of the Killing form.
We take as set of simple roots:
\begin{equation}\label{eq:simple-roots-basis-odd}
\alpha_i=\epsilon_i-\epsilon_{i+1},\,i=1,\dots,n-1,\,\alpha_n=\epsilon_n
\end{equation}
The corresponding fundamental weights are:
\begin{equation}\label{eq:suweightparab-odd}
 \zeta_i=\sum_{k=1}^i\epsilon_i,\,i=1,\dots,n-1,\,\zeta_n=\frac{1}{2}\sum_{k=1}^n \epsilon_k.
\end{equation}
We fix as generator of
\begin{equation}\label{eq:groupgrass-odd}
 \Lambda_w^\vee/\Lambda_r^\vee \cong \ZZ_{2}
 \end{equation}
the coweight
\begin{equation}\label{eq:gen-unitary-odd}
   z_0=e_n.
\end{equation}
The Weyl group for $\SO(2n+1)$ is  the hyperoctahedral group  in $\epsilon_1,\dots,\epsilon_n$ (each reflection associated to a simple root
acting as in the case of $\Sp(n)$). Note, however,
that $\epsilon_1,\dots,\epsilon_n$ is not a basis of the weight lattice. In particular, by (\ref{eq:suweightparab-odd}) the action
of the reflection defined by $\alpha_n$ on the $n$-th fundamental weight is:
\begin{equation}\label{eq:weylbn}
s_n(\zeta_n)=\zeta_{n-1}-\zeta_n.
\end{equation}

Let  $\PP_r\subset \SO(2n+1)$ be the maximal parabolic subgroup determined by the subset $\Pi\backslash \{\alpha_r\}$. The corresponding
flag variety can be identified with the orthogonal Grassmannian $\mathrm{O}\mathrm{Gr}_r(\CC^{2n+1})$
(the Grassmannian of isotropic subspaces of dimension $r$ in $\CC^{2n+1}$ for the fixed quadratic form).
The residual Weyl Group $W_r$ is generated by:
 \begin{equation}\label{eq:residual-Weyl}
 \sigma_1,\dots,\sigma_{r-1},\sigma_{r+1},\dots,\sigma_{n-1},s_n.
 \end{equation}

Let us denote by $\delta_i$ the complete symmetric polynomial of degree $i$.
If we set $\epsilon=(\epsilon_1,\dots,\epsilon_n)$ and $\epsilon'=(\epsilon_1,\dots,\epsilon_r)$, then by (\ref{eq:residual-Weyl})
the polynomial $\delta_{n-r+1}(\epsilon')$ is $W_r$-invariant. Since it is a polynomial on roots with integer coefficients  its
image by the Borel isomorphism represents a class in $H^{2(n-r+1)}(\mathrm{O}\mathrm{Gr}_r(\CC^{2n+1});\ZZ)$. Looking at low degrees we see
that the integral class represented by this polynomial is divisible by 2:
\begin{itemize}
 \item If $r=n$, then by definition
\[\delta_1(\epsilon')=2\zeta_n.\]
\item If $r=n-1$, then a short computation gives
\[\delta_2(\epsilon')=2\zeta_n(\zeta_{n-1}-\zeta_n)+\frac{1}{2}(\epsilon_1^2+\cdots +\epsilon_n^2).\]
By (\ref{eq:weylbn}) the first summand in the RHS is $W_{n-1}$-invariant; the second one is $W$-invariant (it defines the trivial class in the ring of coinvariants).
\end{itemize}

In fact, this divisibility holds in full generality:
\begin{prop}\label{pro:intpolynbn}
Let $n,r$ be non-zero natural numbers with $r\leq n$. The image of the polynomial $\frac{1}{2}\delta_{n-r+1}(\epsilon')$
 by the Borel isomorphism belongs to
$H^{2(n-r+1)}(\mathrm{O}\mathrm{Gr}_r(\CC^{2n+1});\ZZ)$.
\end{prop}
\begin{proof}
As we noticed the $W_r$-invariant polynomial $\frac{1}{2}\delta_{n-r+1}(\epsilon')$ corresponds by the Borel isomorphism to a rational class.
Because  $H^\bullet(\mathrm{SO}(2n+1)/P_r;\ZZ)=H^\bullet(\mathrm{SO}(2n+1)/T;\ZZ)^{W_r}$, we just need to show that the image of $\frac{1}{2}\delta_{n-r+1}(\epsilon')$
 by the Borel isomorphism for the regular orbit is an integral cohomology class. Polynomials corresponding to integral classes in the regular orbit
were originally described by Toda-Watanabe \cite[Theorem 2.1]{TW}. The more recent description in \cite[Theorem 4.1]{N2} is as follows:

The quotient ring
\[\ZZ[\epsilon_1,\dots,\epsilon_n,\varpi_1,\dots,\varpi_{2n}]/I,\]
where $I$ is the ideal spanned by the polynomials
\begin{equation}\label{eq:cohom-relations-odd}
c_i(\epsilon)-2\varpi_i,\,1\leq i\leq n,\quad \varpi_j,\, n<j\leq 2n,\quad \varpi_{2k}+\sum_{j=1}^{2k-1}(-1)^j\varpi_j\varpi_{2k-j},\,1\leq k\leq n,
 \end{equation}
 canonical embeds in the ring of rational coinvariants:
 \begin{equation}\label{eq:bocohomology}
 \ZZ[\epsilon_1,\dots,\epsilon_n,\varpi_1,\dots,\varpi_{2n}]/I\to \QQ[\epsilon_1,\dots,\epsilon_n]/I^+.
 \end{equation}
Its image are exactly those polynomials which corresponds to integers classes by the characteristic isomorphism.

Let $\epsilon''$ denote the set of variables $(\epsilon_{r+1},\dots,\epsilon_n)$. We claim that:
\begin{equation}\label{eq:elemcomplete}
 \delta_{n-r+1}(\epsilon')=\sum_{j=1}^{n-r+1}(-1)^{n-r+1-j} c_{n-r+1-j}(\epsilon'')\delta_j(\epsilon).
 \end{equation}
 We remark that the sum starts with $j=1$, so the homogeneous polynomial $\delta_j(\epsilon)$ is never a constant.

Let us assume that (\ref{eq:elemcomplete}) holds. The fundamental theorem on symmetric polynomials allows to rewrite the relation $c_i(\epsilon)-2\varpi_i=0$ in (\ref{eq:cohom-relations-odd})
as
\begin{equation}\label{eq:relationsymm}
\delta_j(\epsilon)=2Q_j(\varpi),\quad 1\leq j\leq n-r+1,
\end{equation}
where the polynomials $Q_j$ have integral coefficients.
Thus, by (\ref{eq:elemcomplete}), (\ref{eq:relationsymm}) and the presentation of the integral cohomology in (\ref{eq:bocohomology}), we conclude that
$\frac{1}{2}\delta_{n-r+1}(\epsilon')$ corresponds to an integral cohomology class.

To finish the proof of the proposition we have to show that (\ref{eq:elemcomplete}) holds. The generating functions for elementary and complete
symmetric polynomials on a set of variables $p=(p_1,\dots,p_s)$ are, respectively:
\[{\mathbf c}_p(t)=\prod_{i=1}^s(1+p_it),\quad {\boldsymbol\delta}_p(t)=\prod_{i=1}^s\frac{1}{1-p_it}.\]
So we have the immediate equality:
 \[{\boldsymbol\delta}_{\epsilon'}(t)={\mathbf c}_{\epsilon''}(-t){\boldsymbol\delta}_{\epsilon}(t),
\]
which for each degree becomes:
\begin{equation}\label{eq:elemcomplete-full}
 \delta_i(\epsilon')=\sum_{j=0}^i(-1)^{i-j} c_{i-j}(\epsilon'')\delta_j(\epsilon).
\end{equation}
Equation (\ref{eq:elemcomplete}) now follows from (\ref{eq:elemcomplete-full}) because  $c_{n-r+1}(\epsilon'')=0$ (this is the smallest degree
for which the elementary symmetric polynomial vanishes),
and this proves the proposition.
\end{proof}

\begin{remark} If $r=1$, then $n$ is the smallest degree for which there is a $W_1$-invariant polynomial
in weight variables with integer coefficients that is non-divisible, and that corresponds to a divisible integral cohomology class. The Grassmannian
of isotropic lines is the standard (non-degenerate) odd dimensional quadric $Q\subset \PP^{2n}$. The monomial $\epsilon_1$ represents
the generator of $H^2(Q;\ZZ)$, which by the Lefschetz hyperplane theorem is the restriction of the generator of $H^2(\PP^{2n};\ZZ)$. The Lefschetz hyperplane theorem
still grants that $\epsilon_1^{n-1}$ represents the generator of $H^{2n-2}(Q;\ZZ)$, but cannot say anything about $\epsilon_1^n=\delta_1(\epsilon')$.
Note that this is a simple example showing that the Hard Lefschetz
isomorphism is false over the integers.
\end{remark}

\subsection{Proof of Theorem \ref{thm:orthoggrass}}
\label{ss:proof-thm:orthoggrass}
Let us fix $z=dz_0\in\Lambda_w^\vee$. By Proposition \ref{pro:intpolynbn} the polynomial $f=\frac{1}{2}\delta_{n-r+1}(\epsilon')$
represents an integral class in the cohomology of the Grassmannian
$\mathrm{OG}_{r}(\CC^{2n+1})$.

We distinguish two cases:

\underline{Case 1: $r<n$.} We define $\Gamma$ to be the subdiagram starting
at $\alpha_r$ and ending at $\alpha_{n-1}$. By Example \ref{ex:subgraphs}, (2),
$\Gamma$ is projective relative to $\Pi\backslash \{\alpha_r\}$. Therefore
 by
Proposition \ref{pro:graphword}  $bX_{w_\Gamma}\subset Y_{dz_0,r}$ is
a smooth subvariety of complex dimension
$n-r+1$.
For notational simplicity we shall prove the equality:
\begin{equation}\label{eq:inttwice}
 \int_{bX_{w_\Gamma}}2\kappa(f)=(-1)^{n-r+1}d,
\end{equation}
which by (\ref{eq:gen-unitary-odd}) implies the theorem.

As $bX_{w_\Gamma}$ is smooth and $f$ is a polynomial on weights we can compute the integral (\ref{eq:inttwice}) by localization. Strictly speaking,
Proposition  \ref{pro:rootlocalization} is stated for $W_r$-invariant weights and not for $W_r$-invariant polynomials. However, if $q: Y_{z,\OOO'}\to Y_{z,\OOO}$
denotes the projection (as usual $\OOO'\cong \GG/\BB$), the diagram of
of $\CC^*\times \TT$-manifolds
 \begin{equation}\label{eq:charcomm}
\xymatrix{w(0)\ar@{^{(}->}[r]\ar^{q}[d] & Y_{z,\OOO'}\ar^{q}[d] \\
  w(0)\ar@{^{(}->}[r] & Y_{z,\OOO}}
 \end{equation}
is commutative and the arrow between fixed points is an isomorphism. Hence the pullback/restriction of $\kappa(f)$ to $w(0)\in Y_{z,\OOO}$
is canonically identified with the pullback/restriction of $q^*(\kappa(f))$ to $w(0)\in Y_{z,\OOO'}$, where Proposition  \ref{pro:rootlocalization}
can be applied to each monomial.

We let $\epsilon_0\in \mathrm{Hom}(\CC,\CC)$ denote the form dual to the vector $i\in i\RR\subset \CC$ generating the Lie algebra of the first factor of $\CC^*\times\TT $.
Recall that the coweight $z$ defines a morphism  $\Phi_z:\CC^*\times \TT\to \TT$, $(\lambda,\lambda')\mapsto (\phi_z(\lambda)\lambda')$.
Therefore we have $\Phi_z^*\epsilon_i=\epsilon_i+\epsilon_i(z)\epsilon_0$ (pullbacks by the projections
of $\CC^*\times \TT$ onto its factors are implicit).
By  (\ref{eq:gen-unitary-odd}) we have:
\begin{equation}\label{eq:infinityaction}
\Phi_z^*\epsilon_i=\begin{cases} \epsilon_i&\mathrm{if} \,\,1\leq i< n,\\
                      \epsilon_n+d\epsilon_0&\mathrm{if} \,\, i=n.
                     \end{cases}
 \end{equation}
Let us denote $\Phi_z^*\epsilon_i$ by  $\overline{\epsilon}_i$, so that
$$\overline{\epsilon}_i=\epsilon_i\quad\text{for $1\leq i\leq n$},\qquad
\text{and }\overline{\epsilon}_n=\epsilon_n+d\epsilon_0.$$

According to Propositions \ref{pro:rootlocalization}  and   \ref{pro:graphword} and equation (\ref{eq:infinityaction}) we have:
\begin{equation}\label{eq:intcompletesymbo}
 \int_{bX_\Gamma}2{\kappa(f)}=
\sum_{j=0}^{n-r}\left(\frac{\delta_{n-r+1}(\epsilon_1,\dots,\epsilon_{r-1},\epsilon_{r+j})}
{\epsilon_0\prod_{i\in \{r,\dots,n\},i\neq r+j}(\epsilon_i-\epsilon_{r+j})}-
\frac{\delta_{n-r+1}(\overline{\epsilon}_1,\dots,\overline{\epsilon}_{r-1},\overline{\epsilon}_{r+j})}
{\epsilon_0\prod_{i\in \{r,\dots,n\},i\neq r+j}(\overline{\epsilon}_i-\overline{\epsilon}_{r+j})}\right).
\end{equation}
For the discussion that follows it will be convenient to refer to the $2(n-r+1)$ summands in the RHS above in the abstract form $\tfrac{Q_i}{R_i}$. We may then put the expression in common denominator
$\sum_i\tfrac{Q_i}{R_i}=\tfrac{Q}{R},$
where $Q$ and $R$ are
polynomials in $\epsilon_0\,\dots,\epsilon_n$
with coefficients in $\ZZ[d]$ (recall that $d$ is for us a formal variable), and where, more concretely,
\[
 R=\epsilon_0\prod_{r\leq l<i\leq n-1}(\epsilon_i-\epsilon_l)
\prod_{k=r}^{n-1}(\epsilon_n-\epsilon_k)(\overline{\epsilon}_n-\overline{\epsilon}_{k}).
\]
Therefore:
\[Q= \int_{bX_\Gamma}2 \kappa(f)\cdot R.\]
Now the idea to compute the value of the integral is to expand
both $Q$ and $R$ and then compare the coefficients of a particular monomial.
We will consider the following monomial:
\[\tau:=\epsilon_n^{2(n-r)}\epsilon_{n-1}^{n-r-1}\epsilon_{n-2}^{n-r-2}\cdots \epsilon_{r+2}^2\epsilon_{r+1}\epsilon_0,\]
which appears with coefficient 1 in the expansion of $R$ (this can be easily seen by looking
at the exponents in the monomial $\tau$ and observing that in order to obtain $\tau$ in the expansion
of $R$, there is only one possible choice of summand
in each linear factor $(\epsilon_i-\epsilon_l)$, $(\epsilon_n-\epsilon_k)$ and $(\overline{\epsilon}_n-\overline{\epsilon}_{k})$).
Hence the value of the integral $\int_{bX_\Gamma}2 \kappa(f)$ is equal to the coefficient of $\tau$ in $Q$.

To obtain $Q$ we need to put the fractions $\tfrac{Q_i}{R_i}$ with common denominator $R$,
$$\frac{Q_i}{R_i}=\frac{Q_iS_i}{R_iS_i},\qquad\text{ with $R_iS_i=R$, so that }\qquad
Q=\sum_iQ_iS_i.$$
Now, we have $\deg_{\epsilon_n}R_i>0$ for every $i$, and hence necessarily $\deg_{\epsilon_n}S_i<\deg_{\epsilon_n}R$.
Since $\deg_{\epsilon_n}\tau=\deg_{\epsilon_n}R$, it follows that the only summands $Q_iS_i$
which contribute non trivially to $\mathrm{Coeff}_\tau Q$
are those in which $\deg_{\epsilon_n}Q_i>0$,
%
%
namely, the numerators of the two fractions corresponding to the fixed points $w_\Gamma(0)$ and $w_\Gamma(\infty)$. Therefore:
\begin{equation}\label{eq:coeffr0}\begin{split}
\mathrm{Coeff}_\tau Q=\mathrm{Coeff}_\tau
\left((-1)^{n-r}\prod_{k=r}^{n-1}(\overline{\epsilon}_n-\overline{\epsilon}_{k})
\prod_{r\leq l<i\leq n-1}(\epsilon_i-\epsilon_{l})\delta_{n-r+1}(\epsilon_1,\dots,\epsilon_{r-1},\epsilon_n)\right)-\\
\mathrm{Coeff}_\tau\left((-1)^{n-r}\prod_{k=r}^{n-1}(\epsilon_n-\epsilon_k)\prod_{r\leq l<i\leq n-1}(\epsilon_i-\epsilon_{l})
\delta_{n-r+1}(\overline{\epsilon}_1,\dots,\overline{\epsilon}_{r-1},\overline{\epsilon}_n))\right).\end{split}
\end{equation}
The right hand side of (\ref{eq:coeffr0}) is
\[
(-1)^{n-r}\mathrm{Coeff}_{\epsilon_n^{2(n-r)}\epsilon_0}
\left(\prod_{k=r}^{n-1}(\overline{\epsilon}_n-\overline{\epsilon}_{k})\epsilon_n^{n-r+1}-
\epsilon_n^{n-r}{(\overline{\epsilon}_n)}^{n-r}(\overline{\epsilon}_1+\cdots +\overline{\epsilon}_{r-1}+\overline{\epsilon}_n)\right),
\]
which by (\ref{eq:infinityaction}) equals
\[(-1)^{n-r}(d(n-r)-d(n-r+1))=(-1)^{n-r+1}d.\]
This finishes the proof of the first case of the theorem.

\underline{Case 2: $r=n$.} If we now set $bX_\Gamma=bX_e$ and also $f=\frac{1}{2}\delta_1(\epsilon)=\zeta_n$, then
the generalized Chevalley formula implies
equation (\ref{eq:inttwice}).

\section{Orthogonal Grassmann fibrations II}
\label{s:orthogonal-II}

The minimal coadjoint orbits of the even dimensional orthogonal group $\SO(2n)$ are
the orthogonal Grassmannians $\mathrm{OGr}_r(\CC^{2n})$, $1\leq r\leq n$. The purpose of this section is to prove the following:

\begin{theorem}\label{thm:orthoggrasseven} Let $\gamma\subset \SO(2n)$ be a non-contractible loop based at the identity and let $Y_{\gamma,r}$
denote the bundle coming from the clutching construction applied to the coadjoint orbit $\mathrm{OGr}_r(\CC^{2n})$, $1\leq r\leq n$.

The characteristic isomorphism
\[\kappa:H^\bullet(Y_{e,r};\QQ)\to H^\bullet(Y_{\gamma,r};\QQ).\]
is not integral.
\end{theorem}

We follow the same approach as for the odd orthogonal groups. The additional complications
come from the more involved algebraic structure, which is reflected
 in a more complicated Dynkin diagram. In particular we shall be needing three projective subdiagrams and two polynomials to detect the non-integrality
of the characteristic isomorphisms.

We start with a precise description of the quotient of coweight and coroot lattices,
the weight lattice and the action of the Weyl group
on the weight lattice of $\SO(2n+1)$
 \cite[Pages 292--289]{OV}.

 We realize $\mathrm{SO}(2n)$ as the intersection $\SU(2n)\cap \mathrm{SO}(2n,\CC)$, $n>1$,
 where we identify the latter is the group of matrices leaving the
quadratic form $z_1z_{n+1}+\cdots +z_nz_{2n}$ invariant. We select as maximal
torus $T$ the
diagonal matrices in $\SU(2n)\cap \mathrm{SO}(2n,\CC)$ with entries in $S^1$, i.e., the matrices  of the form
\[\theta_1e_1+\cdots\theta_ne_n+\theta_1^{-1}e_{n+1}+\cdots +\theta_n^{-1}e_{2n},\quad \theta_i\in S^1,\]
where $e_i$ denotes the matrix with 1 in the position $(ii)$ and 0 elsewhere.
Let   $\epsilon_1,\dots,\epsilon_n\in \tlie^*$ be the basis dual to $e_1,\dots,e_n$. This is an orthonormal basis for a multiple of the Killing form.
We take as set of simple roots:
\begin{equation}\label{eq:simple-roots-basis-even}
\alpha_i=\epsilon_i-\epsilon_{i+1},\,1\leq i\leq n-1,\quad \alpha_n=\epsilon_{n-1}+\epsilon_n.
\end{equation}
The corresponding fundamental weights are:
\begin{equation}\label{eq:suweightparab-even}
\zeta_i=\sum_{k=1}^i\epsilon_k,\,1\leq i\leq n-2, \quad\zeta_{n-1}=\frac{1}{2}(\epsilon_1+\cdots+\epsilon_{n-1}-\epsilon_n),\quad
\zeta_n=\frac{1}{2}\sum_{k=1}^n\epsilon_k.
\end{equation}
We fix as generators of
\begin{equation}\label{eq:groupgrass-even}
\Lambda_w^\vee/\Lambda_r^\vee\cong
 \begin{cases} \ZZ_{2}\oplus \ZZ_2=\langle z_0\rangle \oplus  \langle z_1\rangle &\mathrm{if}\,\,  n\in 2\ZZ \\
\ZZ_{4}=\langle z_1\rangle &\mathrm{if}\,\, n\notin 2\ZZ.
\end{cases}
 \end{equation}
the coweights
\begin{equation}\label{eq:gen-unitary-even}
   z_0=e_n,\quad z_1=\frac{1}{2}\sum_{k=1}^n e_k.
\end{equation}
The vectors  $\epsilon_1,\dots,\epsilon_n$ are not a basis of the weight lattice. However,
 they are appropriate to describe the Weyl group of $\SO(2n)$. The first $n-1$ simple reflections
 are the transpositions $\sigma_1,\dots,\sigma_{n-1}$, and the last simple reflection
 $s_n$ acts as follows:
\begin{equation}\label{eq:weylbneven}
s_n(\epsilon_i)=\begin{cases} \epsilon_i &\mathrm{if}\,\,1\leq i\leq n-2\\
                 -\epsilon_n &\mathrm{if}\,\,i=n-1\\
                 -\epsilon_{n-1} &\mathrm{if}\,\, i=n.
                \end{cases}
\end{equation}

Let  $\PP_r\subset \SO(2n)$ be the maximal parabolic subgroup determined by the subset $\Pi\backslash \{\alpha_r\}$. The corresponding
flag variety can be identified with the orthogonal Grassmannian $\mathrm{O}\mathrm{Gr}_r(\CC^{2n})$
(the Grassmannian of isotropic subspaces of dimension $r$ in $\CC^{2n}$ for the fixed quadratic form).
The residual Weyl Group $W_r$ is generated by:
 \begin{equation}\label{eq:residual-Weyl-even}
 \sigma_1,\dots,\sigma_{r-1},\sigma_{r+1},\dots,\sigma_{n-1},s_n.
 \end{equation}

  The following theorem ---
whose proof we defer until the end of the section --- contains the integration formulas we need.

\begin{theorem}\label{pro:dn} Let $1\leq r\leq n$, let $z\in \Lambda_w^\vee$ and let $\Gamma,\Gamma'$ and $\Gamma''$ be the
following projective graphs relative to $\Pi\backslash\{\alpha_r\}$:
\begin{itemize}
 \item For $r\neq n-1$, $\Gamma'$ is the  subdiagram of containing the $n-r$ roots  connecting  $\alpha_r$  with $\alpha_{n}$.
 \item For $r< n-1$,  $\Gamma$ is the subdiagram containing the $n-r$ roots connecting $\alpha_r$ with $\alpha_{n-1}$.
 \item For $r=n-1$,  $\Gamma''$ is the subdiagram  containing  $\alpha_{n-1}$  and $\alpha_{n-2}$.
\end{itemize}
There exist  polynomials $f,g\in \mathrm{Sym}_\QQ(\Lambda_w)$, where $f$ is defined for $r\neq n-1$ and $g$ for $r=n-1$,
that correspond to a class in $H^\bullet(\mathrm{O}\mathrm{Gr}_r(\CC^{2n});\ZZ)$ and such that:

\begin{enumerate}
 \item If $z=dz_0$  then:
\begin{equation}\label{eq:gen1subdiagramnasty}
 \int_{bX_{w_{\Gamma'}}}{\kappa(f)}=(-1)^{n-r}\frac{d}{2},\quad r\neq n-1.
\end{equation}
\item If $z=dz_1$ then:
\begin{align}\label{eq:subdiagramfine}
 \int_{bX_{w_\Gamma}}{\kappa(f)} &=(-1)^{n-r+1}\frac{dn}{4},\quad r<n-1,\\ \label{eq:subdiagramnasty}
 \int_{bX_{w_{\Gamma'}}}{\kappa(f)} &=(-1)^{n-r+1}\left(\frac{dn}{4}-\frac{d}{2}\right),\quad r\neq n-1,\\ \label{eq:subdiagramnastier}
 \int_{bX_{w_{\Gamma''}}}{\kappa(g)} &=\frac{d}{2}-\frac{d(n-2)}{2},\quad r=n-1.
\end{align}
\end{enumerate}
\end{theorem}

Assuming that the above integration formulas hold, we can deduce the main result of this section:

\begin{proof}[Proof of Theorem \ref{thm:orthoggrasseven}]
We assume that the characteristic isomorphism is integral and  consider the following cases:

\underline{Case 1: $z=dz_0$.}
If $r\neq n-1$ then  by (\ref{eq:gen1subdiagramnasty}) $\tfrac{d}{2}\in \ZZ$.

If $r=n-1$ then the fundamental weight $\zeta_{n-1}$ is $W_{n-1}$ invariant and by (\ref{eq:cap}) and
(\ref{eq:gen-unitary-even}):
\[\int_{[bX_e]}{\kappa(\zeta_{n-1})}=\frac{d}{2}\in \ZZ,\]
and this proves the theorem in case 1.

\underline{Case 2: $z=dz_1$.}
If $r<n-1$ by (\ref{eq:subdiagramfine}) and (\ref{eq:subdiagramnasty})
   $\tfrac{dn}{4},\tfrac{d}{2}\in \ZZ$. By (\ref{eq:groupgrass-even}) this proves the theorem if $r<n-1$.

If $r=n$    by (\ref{eq:cap}) and
(\ref{eq:gen-unitary-even}):
\[\int_{[bX_e]}{\kappa(\zeta_n)}=-\frac{dn}{4}\in \ZZ,\]
Since by  (\ref{eq:subdiagramnasty}) $\tfrac{d}{2}-\tfrac{dn}{4}\in \ZZ$, the result follows again from (\ref{eq:groupgrass-even}).

If $r=n-1$ we have
\[\int_{[bX_e]}{\kappa(\zeta_{n-1})}=-\frac{d(n-2)}{4}\in \ZZ.\]
Together with (\ref{eq:subdiagramnastier}) this means that $\tfrac{d(n-2)}{4},\tfrac{d}{2}\in\ZZ$, and this finishes the proof of the theorem.
\end{proof}

\begin{proof}[Proof of Theorem \ref{pro:dn}]

The polynomial we consider for $r\neq n-1$ is the same as for the odd
orthogonal groups: $f=\tfrac{1}{2}\delta_{n-r+1}(\epsilon')$,
$\epsilon'=(\epsilon_1,\dots,\epsilon_r)$.

According to \cite[Theorem 4.5]{N2}  the integral cohomology ring of
the regular orbit can be identified with the image of the quotient ring
 \begin{equation}\label{eq:dnbocohomology}
  \ZZ[\epsilon_1,\dots,\epsilon_n,\varpi_1,\dots,\varpi_{2n-21}]/I\to \QQ[\epsilon_1,\dots,\epsilon_n]/I^+\cong H^\bullet(\mathrm{SO}(2n)/T;\QQ),
\end{equation}
where $I$ is the ideal spanned by the polynomials
\[c_i-2\varpi_i,\,1\leq i\leq n-1,\quad c_n,\quad\varpi_j,\,n\leq j\leq 2n-2,\quad \varpi_{2k}+\sum_{j=1}^{2k-1}(-1)^j\varpi_j\varpi_{2k-j},\,1\leq k\leq n-1.\]
In particular we have the relations $c_i-2\varpi_i=0$, $1\leq i\leq n$. If $r\neq n-1$, so $W_r$ acts by permutations
of the variables $\epsilon_1,\dots,\epsilon_r$,
then the proof of Proposition \ref{pro:intpolynbn}
applies verbatim, and thus $f$ represents an integral cohomology class.

\noindent
\underline{Proof of formula (\ref{eq:gen1subdiagramnasty})}. Suppose that $z=dz_0$ and $r\neq n-1$
(subdiagram $\Gamma'$).
Let $\epsilon_0\in \mathrm{Hom}(\CC,\CC)$ denote the same 1-form as in the proof for odd
orthogonal groups (Section \ref{s:orthogonal-I}), and set:
\begin{equation}\label{eq:infinityaction2}
\overline{\epsilon}_i=\begin{cases} \epsilon_i,\,1\leq i< n,\\
                      \epsilon_n+d\epsilon_0.
                     \end{cases}
 \end{equation}
By (\ref{eq:gen-unitary-even}) we have $\overline{\epsilon}_i=\Phi_z^*\epsilon_i$.
We compute $\int_{bX_{w_{\Gamma'}}}2\kappa(f)$ by localization.
According to  Propositions
 \ref{pro:graphword}  and  \ref{pro:rootlocalization}
 the integral equals:
\begin{equation}\label{eq:gen1loc}
\begin{split}
&-\left(\sum_{j=0}^{n-r-1}\frac{\delta_{n-r+1}(\epsilon_1,\dots,\epsilon_{r-1},\epsilon_{r+j})}
{\epsilon_0\prod_{i\in \{r,\dots,n-1\},i\neq r+j}(\epsilon_i-\epsilon_{r+j})(\epsilon_n+\epsilon_{r+j})}\right)+
\frac{\delta_{n-r+1}(\epsilon_1,\dots,\epsilon_{r-1},-\epsilon_n)}
{\epsilon_0\prod_{i=r}^{n-1}(\epsilon_i+\epsilon_n)}+\\
&+\left(\sum_{j=0}^{n-r-1}\frac{\delta_{n-r+1}(\overline{\epsilon}_1,\dots,\overline{\epsilon}_{r-1},\overline{\epsilon}_{r+j})}
{\epsilon_0\prod_{i\in \{r,\dots,n-1\},i\neq r+j}(\epsilon_i-\epsilon_{r+j})(\overline{\epsilon}_n+\overline{\epsilon}_{r+j})}\right)
-\frac{\delta_{n-r+1}(\overline{\epsilon}_1,\dots,\overline{\epsilon}_{r-1},-\overline{\epsilon}_{n})}
{\epsilon_0\prod_{i=r}^{n-1}(\overline{\epsilon}_i+\overline{\epsilon}_n)},
\end{split}
\end{equation}
The idea to compute this fraction will be the same as that in
Subsection \ref{ss:proof-thm:orthoggrass}. Namely, we will write it as $Q/R$ for
some polynomials $Q$ and $R$ in the variables $\epsilon_i$, and we will look for some
monomial $\tau$ which appears with coefficient $1$ in the expansion of $R$; then
$Q/R=\mathrm{Coeff}_\tau Q$.

We may write (\ref{eq:gen1loc}) in the from $Q/R$ with
\[
 R=\epsilon_0\prod_{r\leq l<i\leq n-1}(\epsilon_i-\epsilon_{l})
\prod_{k=r}^{n-1}(\epsilon_n+\epsilon_k)\prod_{k=r}^{n-1}(\overline{\epsilon}_n+\overline{\epsilon}_k).
\]
The monomial
\[\tau:=\epsilon_n^{2(n-r)}\epsilon_{n-1}^{n-r-1}\epsilon_{n-2}^{n-r-2}\cdots \epsilon_{r+2}^2\epsilon_{r+1}\epsilon_0\]
appears with coefficient 1 in the expansion of $R$.
The same arguments that in Subsection \ref{ss:proof-thm:orthoggrass} allowed us to identify the quantities in (\ref{eq:intcompletesymbo}) and in (\ref{eq:coeffr0}) (namely, considering the degree in $\epsilon_n$ of the numerators and denominators of the fractions in (\ref{eq:intcompletesymbo})), imply that (\ref{eq:gen1loc}), which is equal to $\mathrm {Coeff}_\tau Q$, can be identified with:
\[\begin{split} & \mathrm{Coeff}_\tau
\left(\prod_{k=r}^{n-1}(\overline{\epsilon}_n+\overline{\epsilon}_{k})
\prod_{r\leq l<i\leq n-1}(\epsilon_i-\epsilon_{l})\delta_{n-r+1}(\epsilon_1,\dots,\epsilon_{r-1},-\epsilon_n)\right)-\\
 & \mathrm{Coeff}_\tau\left(\prod_{k=r}^{n-1}(\epsilon_n+\epsilon_k)\prod_{r\leq l<i\leq n-1}(\epsilon_i-\epsilon_{l})
\delta_{n-r+1}(\overline{\epsilon}_1,\dots,\overline{\epsilon}_{r-1},-\overline{\epsilon}_n))\right),\end{split}\]
which on its turn equals to:
\begin{align*}
(-1)^{n-r+1} & \mathrm{Coeff}_{\epsilon_n^{2(n-r)}\epsilon_0}
\left(\prod_{k=r}^{n-1}(\overline{\epsilon}_n+\overline{\epsilon}_{k})\epsilon_n^{n-r+1}+
\epsilon_n^{n-r}\overline{\epsilon}_n^{n-r}(\overline{\epsilon}_1+
\cdots\overline{\epsilon}_{r-1}-\overline{\epsilon}_n)\right)=\\
&=(-1)^{n-r+1}d(n-r)+ \\
&\qquad\qquad (-1)^{n-r+1}\mathrm{Coeff}_{\epsilon_n^{n-r}\epsilon_0}\left(\left(\epsilon_n+d\epsilon_0\right)^{n-r}
(\epsilon_1+\cdots+\epsilon_{r-1}-\epsilon_n-d\epsilon_0)\right)\\
&=(-1)^{n-r+1}(d(n-r)-d(n-r)-d)=(-1)^{n-r}d.
\end{align*}
This finishes the proof of formula (\ref{eq:gen1subdiagramnasty}).

In the proofs of the remaining three formulas we will follow the same strategy as before: using localization we will equate the relevant integral to a sum of rational functions, and the numerical value of the sum will be computed writing it in common denominator $Q/R$, and choosing an appropriate monomial $\tau$ satisfying $\mathrm{Coeff}_{\tau}R=1$, so that $Q/R=\mathrm{Coeff}_{\tau}Q$.
We will also prove that, among the rational functions in the initial formula given by localization, only very few of them actually contribute to $\mathrm{Coeff}_{\tau}Q$, and this will lead immediately to the desired formulas.


\noindent\underline{Proof of formula (\ref{eq:subdiagramfine})}
Assume $z=dz_1$ and $r<n-1$  (subdiagram $\Gamma$).
We modify accordingly the definition in (\ref{eq:infinityaction2}) to  \[\overline{\epsilon}_i=\epsilon_i+\frac{d}{2}\epsilon_0,\]
so that $\overline{\epsilon}_i=\Phi_z^*\epsilon_i$  by (\ref{eq:gen-unitary-even}).
 By  Propositions  \ref{pro:graphword}  and  \ref{pro:rootlocalization}:
\begin{equation}\label{eq:intcompletesecond}
 \int_{bX_{w_\Gamma}}2{\kappa(f)}=
\sum_{j=0}^{n-r}\left(\frac{\delta_{n-r+1}(\epsilon_1,\dots,\epsilon_{r-1},\epsilon_{r+j})}
{\epsilon_0\prod_{i\in \{r,\dots,n\},i\neq r+j}(\epsilon_i-\epsilon_{r+j})}-
\frac{\delta_{n-r+1}(\overline{\epsilon}_1,\dots,\overline{\epsilon}_{r-1},\overline{\epsilon}_{r+j})}
{\epsilon_0\prod_{i\in \{r,\dots,n\},i\neq r+j}({\epsilon}_i-{\epsilon}_{r+j})}\right).
\end{equation}
We write the RHS as $Q/R$, where
\[
 R=\epsilon_0\prod_{r\leq l<i\leq n}(\epsilon_i-\epsilon_{l}),
\]
and we set
$\tau:=\epsilon_n^{n-r}\epsilon_{n-1}^{n-r-1}\cdots \epsilon_{r+2}^2\epsilon_{r+1}^1\epsilon_0,$
which satisfies $\mathrm{Coeff}_{\tau}R=1$, so (\ref{eq:intcompletesecond}) is equal to
$\mathrm{Coeff}_{\tau}Q$.

Using the same considerations on the degree w.r.t. $\epsilon_n$ we deduce that the only summands in
(\ref{eq:intcompletesecond}) that contribute to $\mathrm{Coeff}_{\tau}Q$ are those corresponding to
%
%
the fixed points $w_\Gamma(0)$ and $w_\Gamma(\infty)$. The same considerations using $\deg_{\epsilon_0}$ instead of $\deg_{\epsilon_n}$ imply that the summand corresponding to the point
 $w_\Gamma(0)$ does not contribute to $\mathrm{Coeff}_{\tau}Q$. Therefore the integral coincides with
\[
\mathrm{Coeff}_\tau\left(-(-1)^{n-r}
\prod_{r\leq l<i\leq n-1}(\epsilon_i-\epsilon_l)
\delta_{n-r+1}(\overline{\epsilon}_1,\dots,\overline{\epsilon}_{r-1},\overline{\epsilon}_{n})\right),
\]
which  equals
\[(-1)^{n-r+1}
\mathrm{Coeff}_{\epsilon_0\epsilon_n^{(n-r)}}
(\overline{\epsilon}_n^{n-r}(\overline{\epsilon}_1+\cdots\overline{\epsilon}_{r-1}+\overline{\epsilon}_n))
=(-1)^{n-r+1}\frac{dn}{2}.\]
This proves (\ref{eq:subdiagramfine}).

\noindent\underline{Proof of formula (\ref{eq:subdiagramnasty})}
Assume $z=dz_1$ and $r\neq n-1$ (subdiagram $\Gamma'$).
The proof  is very similar to that of formula (\ref{eq:gen1subdiagramnasty}). In
(\ref{eq:gen1loc}) the denominators are the same for either definition of $\overline{\epsilon}_i$. Therefore by making the same choice of common
denominator and monomial we deduce that:
\begin{align*}
\int_{bX_{w_{\Gamma'}}} & 2\kappa(f)=(-1)^{n-r+1}\mathrm{Coeff}_{\epsilon_n^{2(n-r)}\epsilon_0}
\left(\prod_{k=r}^{n-1}(\overline{\epsilon}_n+\overline{\epsilon}_{k})\epsilon_n^{n-r+1}+
\epsilon_n^{n-r}\overline{\epsilon}_n^{n-r}
(\overline{\epsilon}_1+\cdots\overline{\epsilon}_{r-1}-\overline{\epsilon}_n)\right)\\
&= (-1)^{n-r+1}d(n-r)+ \\ &\qquad+(-1)^{n-r+1}\mathrm{Coeff}_{\epsilon_n^{n-r}\epsilon_0}\left(\left(\epsilon_n+\frac{d}{2}\epsilon_0\right)^{n-r}
\left(\epsilon_1+\cdots+\epsilon_{r-1}-\epsilon_n+\frac{d(r-2)}{2}\epsilon_0\right)\right) \\
&=(-1)^{n-r+1}\left(d(n-r)-\frac{d(n-r)}{2}+\frac{d(r-2)}{2}\right)=(-1)^{n-r+1}\left(\frac{dn}{2}-d\right).
\end{align*}

\noindent\underline{Proof of formula (\ref{eq:subdiagramnastier})}
Assume $z=dz_1$ and $r=n-1$ (subdiagram $\Gamma''$).
The polynomial $g$ will be constructed using the same
ideas applied for $f$. We consider the variables
\[\eta=(\eta_1,\dots,\eta_{n-1},\eta_n):=(\epsilon_1,\dots,\epsilon_{n-1},-\epsilon_n).\]
As $W_{n-1}$ acts on this set of variables as the symmetric group $S_{n-1}$, the corresponding complete symmetric polynomials
are $W_{n-1}$ invariant (note that  $\tfrac{1}{2}\delta_1(\eta)=\xi_{n-1}$).

We claim that for $m\geq 1$, $\tfrac{1}{2}\delta_m(\eta)$ represents an integral
class in the cohomology of the orthogonal Grassmannian
$\mathrm{OGr}_{n-1}(\CC^{2n})$. The claim follows from these three facts: (1) defining
 generating functions ${\boldsymbol\delta}$ and ${\mathbf c}$ as in Proposition \ref{pro:intpolynbn},
 we have the identity
\[{\boldsymbol\delta}_{\eta'}(t)={\mathbf c}_{\eta''}(-t){\boldsymbol\delta}_\eta(t),\quad \text{where }
\eta'=(\epsilon_1,\dots,\epsilon_{n-1})\text{ and }\eta''=\epsilon_n;\]
 (2) we also have ${\boldsymbol\delta}_{\eta'}(t)={\boldsymbol\delta}_{\epsilon'}(t)$, and (3) the elements $\frac{1}{2}\delta_{m}(\epsilon)$ represent classes in
 $H^\bullet(\SO(2n)/T;\ZZ)$.

We set $g=\frac{1}{2}\delta_3(\eta)$.
 By  Propositions  \ref{pro:graphword}  and  \ref{pro:rootlocalization},   the evaluation of
 \[\int_{bX_{w_{\Gamma''}}}2\kappa(g)\]
 by localization is

\[
\begin{split}
&\frac{\delta_3(\epsilon_1,\dots,\epsilon_{n-2},\epsilon_{n-1},-\epsilon_{n})}
{\epsilon_0(\epsilon_{n-1}-\epsilon_{n})(\epsilon_{n-2}-\epsilon_n)}-
\frac{\delta_3(\epsilon_1,\dots,\epsilon_{n-2},\epsilon_{n},-\epsilon_{n-1})}
{\epsilon_0(\epsilon_{n-1}-\epsilon_{n})(\epsilon_{n-2}-\epsilon_{n-1})}+
\frac{\delta_3(\epsilon_1,\dots,\epsilon_{n-1},\epsilon_{n},-\epsilon_{n-2})}
{\epsilon_0(\epsilon_{n-2}-\epsilon_{n})(\epsilon_{n-2}-\epsilon_{n-1})}
\\
-&\frac{\delta_3(\overline{\epsilon}_1,\dots,\overline{\epsilon}_{n-2},\overline{\epsilon}_{n-1},-\overline{\epsilon}_{n})}
{\epsilon_0(\epsilon_{n-1}-\epsilon_{n})(\epsilon_{n-2}-\epsilon_n)}+
\frac{\delta_3(\overline{\epsilon}_1,\dots,\overline{\epsilon}_{n-2},\overline{\epsilon}_{n},-\overline{\epsilon}_{n-1})}
{\epsilon_0(\epsilon_{n-1}-\epsilon_{n})(\epsilon_{n-2}-\epsilon_{n-1})}-
\frac{\delta_3(\overline{\epsilon}_1,\dots,\overline{\epsilon}_{n-1},\overline{\epsilon}_{n},-\overline{\epsilon}_{n-2})}
{\epsilon_0(\epsilon_{n-2}-\epsilon_{n})(\epsilon_{n-2}-\epsilon_{n-1})}
\end{split}
\]
We write the previous sum of six rational functions in the form $Q/R$ with \[R=\epsilon_0(\epsilon_{n-1}-\epsilon_{n})(\epsilon_{n-2}-\epsilon_{n-1})(\epsilon_{n-2}-\epsilon_{n}).\]
Setting $\tau=\epsilon_0\epsilon_{n}^2\epsilon_{n-1}$ we have $\mathrm{Coeff}_{\tau}R=-1$.
Among the six fractions, the only ones that
contribute to $\mathrm{Coeff}_{\tau}Q$
are the ones corresponding to the three fixed points over $\infty$.

The contribution of the fraction corresponding to the fixed point $e(\infty)$ is equal to:
\[
\mathrm{Coeff}_{\tau}\left(\epsilon_{n-1}\left(-\epsilon_n-\frac{d}{2}\epsilon_0\right)^2
\left(\epsilon_1+\cdots +\epsilon_{n-1}-\epsilon_n+\frac{d(n-2)}{2}\epsilon_0\right)\right)=\frac{d(n-2)}{2}-d.\]
The fraction corresponding to the fixed point $\sigma_{n-1}(\infty)$ contributes as:
\[
\mathrm{Coeff}_{\tau}\left(-\epsilon_{n}\left(\epsilon_n+\frac{d}{2}\epsilon_0\right)\left(-\epsilon_{n-1}-\frac{d}{2}\epsilon_0\right)
\left(\epsilon_1+\cdots+\epsilon_{n-2}-\epsilon_{n-1}+\epsilon_n+\frac{d(n-2)}{2}\epsilon_0\right)\right)\]
which equals
\[\frac{d(n-2)}{2}.\]
Finally, the fraction corresponding to the fixed point $w(\infty)$ contributes as:
\[\begin{split}
&\mathrm{Coeff}_{\tau}\left(-\epsilon_{n-1}\left(\epsilon_n+\frac{d}{2}\epsilon_0\right)^2
\left(\epsilon_1+\cdots+\epsilon_{n-3}-\epsilon_{n-2}+\epsilon_{n-1}+\epsilon_n+\frac{d(n-2)}{2}\epsilon_0\right)\right)+\\
+&\mathrm{Coeff}_{\tau}\left(\epsilon_{n}\left(\epsilon_n+\frac{d}{2}\epsilon_0\right)\left(\epsilon_{n-1}+\frac{d}{2}\epsilon_0\right)
\left(\epsilon_1+\cdots+\epsilon_{n-3}-\epsilon_{n-2}+\epsilon_{n-1}+\epsilon_n+\frac{d(n-2)}{2}\epsilon_0\right)\right)
\end{split}
\]
which equals
\[-\frac{d(n-2)}{2}-d+\frac{d(n-2)}{2}+d=0.\]
Therefore, summing up the three contributions we obtain
\[\int_{bX_{w_{\Gamma''}}}\kappa(g)=\frac{d}{2}-\frac{d(n-2)}{2},\]
which is what we wanted to prove.
\end{proof}

\section{Exceptional groups}\label{s-exotic}

Let us repeat once again that, for the purposes of proving our main Theorem \ref{thm:injdiff},
among the simply connected exceptional groups  we only need to consider $E_6$ and $E_7$,
since these are the only ones with nontrivial center.

The goal of this section is to prove the following result, which is the last piece in the proof of Theorem \ref{thm:injdiff}:

\begin{theorem}\label{thm:e6} Let $\gamma$  be a non-contractible loop based at the identity in either $E_6/Z$ or $E_7/Z$, let $\OOO$ be
a minimal coadjoint orbit of the group containing $\gamma$, and let $Y_{\gamma,\OOO}$
denote the bundle coming from the clutching construction applied to  $\OOO$.

The characteristic isomorphism
\[\kappa:H^\bullet(Y_{e,\OOO};\QQ)\to H^\bullet(Y_{\gamma,\OOO};\QQ).\]
is not integral.
\end{theorem}

The calculations for all minimal orbits of $E_6$ and for five of the minimal orbits of $E_7$
are simple enough to be handled without computer. However, for two of the minimal orbits of $E_7$ we will have to use computers in two different ways:
\begin{enumerate}
 \item To compute invariant polynomials which represent integral classes.
 \item To implement an algorithm computing the generalized Chevalley formula (\ref{eq:cap})
 for certain Schubert varieties with very elementary Hasse diagram.
\end{enumerate}

Let us discuss first the question of the Schubert varieties. Let $G$ be a compact, connected simple Lie group and let $T\subset G$ be a maximal torus. Let
$\alpha_1,\dots,\alpha_n$ be a set of simple roots with Weyl group $W$. Fix $r$,  $1\leq r\leq n$,
let $f$ be a $W_r$-invariant polynomial
of degree $l+1$ and let $w\in W^r$ be a word of length $l$.
To compute
\[\int_{[bX_w]}\kappa(f)\]
 applying recursively the generalized Chevalley formula (\ref{eq:cap})
\[\kappa(\zeta)\cap [bX_w]=-(w\cdot\zeta)(z)[X_w]-
\sum_{w'\overset{s}{\rightarrow}w,\,w'\in W^P} (w'\cdot \zeta)(h_{s})[bX_{w'}]\]
 (see (\ref{eq:def-h}) for the definition of the coroot $h_s$)
 we have to know all $w'\in W$ such that $w'\leq w$, i.e, we need to know
 the Hasse diagram of $w$. In this sense it is useful to consider first those $w\in W$ whose Hasse diagram is easy to describe. The
 natural choice are elements $w\in W$ given by words
 \[s_{i_1}\cdots s_{i_{l-1}}s_r\]
with the following additional properties:
\begin{itemize}
 \item The indices $i_1,\dots,i_{l-1},r$ are in the reduced expression  are pairwise different.
  \item For each $j$ the simple reflection $s_{i_j}$ does not commute with
  at least one simple reflection to its right in $s_{i_1}\cdots s_{i_{l-1}}s_r$
  (which can be $s_r$).
\end{itemize}
\begin{lemma}\label{lem:tits}
 If $w\in W$ is given by a word as above, then the following holds:
 \begin{enumerate}
  \item There is a bijection between subwords of $s_{i_1}\cdots s_{i_{l-1}}s_r$ and elements  $w'\in W$ such that $w'\leq w$.
In particular the Hasse diagram for $w$ has $l!$ intervals from $e$ to $w$.
  \item $w\in W^r$, i.e, $X_w\subset \OOO$ is a Schubert variety of complex dimension $l$.
 \end{enumerate}
\end{lemma}
\begin{proof}

A result of Tits \cite{Ti} asserts that given a word on the simple reflections $s_i$  (more generally, a word in a Coxeter group) one can pass to any reduced
expression by a sequence of moves which are cancellation of adjacent equal simple reflections  and braid moves of the form
\[\overbrace{s_is_j{\cdots}s_is_j}^{m(i,j)}\mapsto \overbrace{s_js_i\cdots s_js_i}^{m(i,j)},\]
where $m(i,j)$ is the order of $s_is_j\in W$.  In particular the subset of simple reflections appearing
in any reduced expression is the same. Since $s_{i_1}\cdots s_{i_{l-1}}s_r$ has different letters we conclude
that it is a reduced expression, that different subwords must represent different Weyl group elements, and that
any other reduced expression of  $s_{i_1}\cdots s_{i_{l-1}}s_r$  is the result
of a sequence of  transpositions of simple reflections
associated to non-adjacent roots in the Dynkin diagram.

To show that $w\in W^r$ we use that $w$ has a unique parabolic factorization $w=uv$, where $u\in W^r$, $v\in W_r$ and
$l(w)=l(u)+l(v)$ \cite[Proposition 1.10]{Hu}. In particular reduced expressions of $u$ and $v$ give rise to a reduced expression of $w$.
But if $v$ is not the identity, then we obtain a reduced expression for $w$ where the rightmost reflection is $s_j$, $j\neq r$.
By the above theorem of Tits,  these two reduced expression are related by transpositions of commuting reflections, which implies
that the reflection $s_j$ must commute with all reflections to its right in $w$, which contradicts our assumption.
Therefore, we conclude that $v$ is the identity and thus $w\in W^r$.
\end{proof}

\subsection{Proof of Theorem \ref{thm:e6} for $E_6/Z$}

Calculations for $E_6/Z$ will not require the use computers.

We start with a precise description of the quotient of coweight and coroot lattices,  the weight lattice and the action of the Weyl group
 \cite[Pages 292--289]{OV}.

The roots and weights for the dual $\tlie^*$ of the Lie algebra of a maximal torus $T\subset E_6$ can be expressed in a way similar to the case $\SU(6)$.
Let $\epsilon_1,\dots,\epsilon_6,\epsilon$ be a basis of a rational vector space and consider the quotient space by the line spanned by
$\sum_{k=1}^6\epsilon_k$. Consider the inner product on the quotient space induced by the degenerate pairing:
\[
 \langle \epsilon_i,\epsilon_j\rangle =\begin{cases}
                          5/6 &\mathrm{if}\,i=j,\\
-1/6 &\mathrm{if}\ i\neq j,
                         \end{cases}\qquad \langle \epsilon,\epsilon_i\rangle =0,\qquad \langle\epsilon,\epsilon\rangle =1/2.
\]
 Then $\tlie^*$ together with a multiple of the Killing form  can be identified
with the quotient space so that one can choose the following
set of simple roots:
\begin{equation}\label{eq:simpleroots}
 \alpha_i=\epsilon_i-\epsilon_{i+1},\, i=1,\dots,5, \qquad \alpha_6=\epsilon_4+\epsilon_5+\epsilon_6+\epsilon.
\end{equation}
In terms of this identification, the associated fundamental weights can be computed as
\begin{equation}\label{eq:fundweights}
 \zeta_i=\epsilon_1+\cdots+\epsilon_i+\mathrm{min}(i,6-i)\epsilon,\, i=1,\dots,5,\qquad
    \zeta_6=2\epsilon.
\end{equation}
As a generator of $\Lambda_w^\vee/\Lambda_r^\vee\cong \ZZ_3$ we select the following
combination of coroots:
\begin{equation}\label{eq:fundgroupe6}
 z_0=\frac{1}{3}(h_{\alpha_1}-h_{\alpha_2}+h_{\alpha_4}-h_{\alpha_5}).
\end{equation}
To describe the Weyl group action it is convenient to introduce the following basis of
$\tlie^*$  (see for example \cite[p. 266]{TW}) where
the root ordering is different):
\begin{equation}\label{eq:weylvariables}
\begin{cases}
t_1 &=\zeta_1=\epsilon_1+\epsilon,\\
t_2 &=-\zeta_1+\zeta_2=\epsilon_2+\epsilon,\\
t_3 &=-\zeta_2+\zeta_3=\epsilon_3+\epsilon,\\
t_4 &=-\zeta_3+\zeta_4+\zeta_6=\epsilon_4+\epsilon,\\
t_5 &=-\zeta_4+\zeta_5+\zeta_6=\epsilon_5+\epsilon,\\
t_6 &=-\zeta_5+\zeta_6=\epsilon_6+\epsilon,\\
t &=\zeta_6=2\epsilon.
\end{cases}
\end{equation}
These vectors do not span the weight lattice over the integers,
but they are useful for our purposes because in contrast to what happens with the basis
of $\tlie^*$ given by fundamental weights, the action of the simple reflections expressed in
terms of $\{t_1,\dots,t_6,t\}$ looks rather close to a permutation group (see for example \cite[Table (4.5)]{TW}): for any $1\leq i\leq 5$,
the reflection $\sigma_i$ exchanges $t_i$ and $t_{i+1}$, and fixes all other elements in the basis
$\{t_1,\dots,t_6,t\}$; only the action of $s_6$ is slightly more involved:
\begin{equation}\label{eq:weylaction2}
 s_6(t_i)=\begin{cases} t_{i}& \mathrm{if}\,\, i=1,2,3,\\
-(t_5+t_6-t)&\mathrm{if}\,i=4,\\
-(t_4+t_6-t)& \mathrm{if}\,i=5,\\
-(t_4+t_5-t)& \mathrm{if}\,i=6,
               \end{cases}
\end{equation}
\[s_6(t)=2t-(t_4+t_5+t_6).\]

Let $\PP_i\subset \EE_6$ be the maximal parabolic subgroup associated to the subset $\Pi\backslash \{\alpha_i\}$ of the Dynkin diagram $\Pi$, let $z=dz_0\in \Lambda_w^\vee$,
and let $\OOO=E_6/\PP_i$. We distinguish the following three cases:

\underline{Case 1:  $\PP_1,\PP_2,\PP_4$ and $\PP_5$ (arguments using degree 2 cohomology).}
The degree 2 integral cohomology of $\OOO$ is generated by the image of the weight $\zeta_i$. By (\ref{eq:fundgroupe6}) we have:
\begin{equation}\label{eq:weighteval}
 \zeta_i(z_0)=\begin{cases}
   1/3& \mathrm{if}\,\,i=1,4,\\
-1/3& \mathrm{if}\,\,i=2,5,\\
0& \mathrm{if}\,\,i=3,6.\\
                                         \end{cases}
\end{equation}

 By (\ref{eq:weighteval}) and the generalized Chevalley
formula applied to $[bX_e]\subset Y_{z,\OOO}$, we have:
\[\int_{[bX_e]}\kappa(\zeta_i)=\pm\frac{d}{3},\qquad  \mathrm{if}\,\,i=1,2,4,5.
\]

\underline{Case 2:  $\PP_3$ (arguments using degree 4 cohomology).}
  By (\ref{eq:weylvariables}) and (\ref{eq:fundgroupe6}) we have:
\[ t_i(z)=\begin{cases}\frac{d}{3}& \mathrm{if}\,\,i=1,3,4,6,\\
         -\frac{2d}{3}& \mathrm{if}\,\,i=2,5,
                             \end{cases}\]
and $t(z)=0$.
 The residual Weyl group $W_3$ acts on the variables $t_1,t_2,t_3$ by permutations.
In particular the degree $2$ elementary symmetric polynomial
$f=c_2(t_1,t_2,t_3)$ represents an integral cohomology class of $\OOO\cong E_6/\PP_3$. We
will compute the integral of $\kappa(f)$ on the smooth Schubert variety $bX_{\sigma_3}\subset Y_{z,\OOO}$ applying the generalized Chevalley formula:

\begin{align}\nonumber
 \int_{[bX_{\sigma_3}]}\kappa(f)& =-\int_{X_{\sigma_3}}\kappa(t_1)\cdot t_2(z)-\int_{X_{\sigma_3}}\kappa(t_1)\cdot t_3(z)-
 \int_{X_{\sigma_3}}\kappa(t_2)\cdot t_3(z)-\\\nonumber & -\int_{bX_{e}}\kappa(t_1)\cdot t_2(h_{\alpha_3})-\int_{bX_{e}}\kappa(t_1)\cdot t_3(h_{\alpha_3})
 -\int_{bX_{e}}\kappa(t_2)\cdot t_3(h_{\alpha_3})=\\\nonumber
& =t_1(h_{\alpha_3})\cdot t_2(z)+t_1(h_{\alpha_3})\cdot t_3(z)+ t_2(h_{\alpha_3})\cdot t_3(z)+\\\nonumber
&+t_1(z)\cdot t_2(h_{\alpha_3})+t_1(z)\cdot t_3(h_{\alpha_3})+t_2(z)\cdot t_3(h_{\alpha_3})=\\\nonumber
&= 0\cdot \left(-\frac{2d}{3}\right)+0 \cdot \frac{d}{3}+0\cdot \frac{d}{3}+\frac{d}{3}\cdot 0+\frac{d}{3}\cdot 1+\left(-\frac{2d}{3}\right)\cdot 1=-\frac{d}{3},
\end{align}
which proves Case 2.

\underline{Case 3:  $\PP_6$ (arguments using degree 6 cohomology).}
The residual Weyl group $W_6$ acts on $t_1,\dots,t_6$ by permutations, so symmetric polynomials in these variables
represent integral classes on $\OOO\cong E_6/\PP_6$. One can check that
 $\kappa(c_2(t_1,\dots,t_6))$ also represents an integral cohomology class in $Y_{z,\OOO}$. However, we will show that
for the degree $3$ elementary symmetric polynomial  $f=c_3(t_1,\dots,t_6)$ we have:
\[\int_{bX_{\sigma_3s_6}}\kappa(f)=\frac{2d}{3}\]
To compute the integral by localization we relabel the 1-form $\epsilon_0\in \Hom(\CC,\CC)$ dual to
$i\in i\RR\subset \CC$ as $t_0$ and introduce the notation
\begin{equation}
\label{eq:qaction-2}
\bar{t}_i:=\Phi_z^*t_i=t_i+t_i(z)t_0,\qquad\text{where }
 t_i(z)=\begin{cases}\frac{d}{3}& \mathrm{if}\,\,i=1,3,4,6,\\
         -\frac{2d}{3}& \mathrm{if}\,\,i=2,5,
                             \end{cases}
\end{equation}
and $\bar{t}=t$, and we extend the definition of $\bar{(\cdot)}$ to be compatible with the ring structure: namely, for any polynomial $P$ we set $$\overline{P(t_1,\dots,t_6,t)}:=P(\bar{t}_1,\dots,\bar{t}_6,\bar{t}).$$

By Proposition \ref{pro:tangentspace} the $\CC^*\times \TT$-module structure of the tangent spaces at fixed points is:
\begin{align}\nonumber
 T_{(0,e)}bX_{\sigma_3s_6} &=\CC_{\pi_1*}\oplus \CC_ {-\pi_2^*\alpha_6}\oplus \CC_ {-\pi_2^*(\alpha_6+\alpha_3)}=
t_0(t_4+t_5+t_6-t)(t_3+t_5+t_6-t),\\ \nonumber
 T_{(0,s_6)}bX_{\sigma_3s_6} &=\CC_{\pi_1*}\oplus \CC_ {\pi_2^*\alpha_6}\oplus \CC_ {-\pi_2^*\alpha_3}=
-t_0(t_4+t_5+t_6-t)(t_3-t_4),\\ \nonumber
 T_{(0,\sigma_3s_6)}bX_{\sigma_3s_6}&=\CC_{\pi_1*}\oplus \CC_ {\pi_2^*(\alpha_6+\alpha_3)}\oplus \CC_ {\pi_2^*\alpha_3}=
t_0(t_3+t_5+t_6-t)(t_3-t_4),\\ \nonumber
 T_{(\infty,e)}bX_{\sigma_3s_6}&=\CC_{-\pi_1*}\oplus \CC_ {-z^*\alpha_6}\oplus \CC_ {-z^*(\alpha_6+\alpha_3)}=
-t_0(t_4+t_5+t_6-t)(t_3+t_5+t_6-t),\\ \nonumber
 T_{(0,s_6)}bX_{\sigma_3s_6}&=\CC_{-\pi_1*}\oplus \CC_ {z^*\alpha_6}\oplus \CC_ {-z^*\alpha_3}=
t_0(t_4+t_5+t_6-t)(t_3-t_4),\\ \nonumber
 T_{(0,\sigma_3s_6)}bX_{\sigma_3s_6}&=\CC_{-\pi_1*}\oplus \CC_ {z^*(\alpha_6+\alpha_3)}\oplus \CC_ {z^*\alpha_3}=
-t_0(t_3+t_5+t_6-t)(t_3-t_4).
\end{align}

The restriction of the polynomial $\kappa(f)$ to the fibers over the fixed points defines the following equivariant cohomology classes:
\begin{align}\nonumber
 \kappa(f)|_{(0,e)} &=\pi_2^*f=c_3(t_1,\dots,t_6),\\ \nonumber
\kappa(f)|_{(0,s_6)}&=\pi_2^*(s_6\cdot f)=c_3(t_1,t_2,t_3,s_6(t_4),s_6(t_5),s_6(t_6)),\\ \nonumber
 \kappa(f)|_{(0,\sigma_3s_6)}&=\pi_2^*(\sigma_3s_6\cdot f)=
c_3(t_1,t_2,t_4,\sigma_3s_6(t_4),\sigma_3s_6(t_5),\sigma_3s_6(t_6)),\\ \nonumber
 \kappa(f)|_{(\infty,e)}&=z^*f=c_3(\bar{t}_1,\dots,\bar{t}_6),\\ \nonumber
\kappa(f)|_{(\infty,s_6)}&=z^*(s_6\cdot f)=
c_3(\bar{t}_1,\bar{t}_2,\bar{t}_3,\overline{s_6(t_4)},\overline{s_6(t_5)},\overline{s_6(t_6)})\\ \nonumber
 \kappa(f)|_{(\infty,\sigma_3s_6)}&=z^*(\sigma_3s_6\cdot f)=
c_3(\bar{t}_1,\bar{t}_2,\bar{t}_4,\overline{\sigma_3s_6(t_4)},\overline{\sigma_3s_6(t_5)},\overline{\sigma_3s_6(t_6)})
\end{align}
Therefore the integral is
\begin{equation}\label{eq:deg6integral}
 \begin{split}
  \int_{bX_{\sigma_3s_6}}\kappa(f) &=\frac{(t_3-t_4)c_3(t_1,\dots,t_6)-
(t_3+t_5+t_6-t)c_3(t_1,t_2,t_3,s_6(t_4),s_6(t_5),s_6(t_6))}
{t_0(t_3-t_4)(t_4+t_5+t_6-t)(t_3+t_5+t_6-t)}\\
& +\frac{(t_4+t_5+t_6-t))c_3(t_1,t_2,t_4,\sigma_3s_6(t_4),\sigma_3s_6(t_5),\sigma_3s_6(t_6))-
(t_3-t_4)c_3(\bar{t}_1,\dots,\bar{t}_6)}
{t_0(t_3-t_4)(t_4+t_5+t_6-t)(t_3+t_5+t_6-t)}\\
 & +\frac{(t_3+t_5+t_6-t)c_3(\bar{t}_1,\bar{t}_2,\bar{t}_3,\overline{s_6(t_4)},\overline{s_6(t_5)},\overline{s_6(t_6)})}
{t_0(t_3-t_4)(t_4+t_5+t_6-t)(t_3+t_5+t_6-t)}\\
& -\frac{(t_4+t_5+t_6-t)c_3(\bar{t}_1,\bar{t}_2,\bar{t}_4,\overline{\sigma_3s_6(t_4)},
\overline{\sigma_3s_6(t_5)},\overline{\sigma_3s_6(t_6)})}{t_0(t_3-t_4)(t_4+t_5+t_6-t)(t_3+t_5+t_6-t)}.
 \end{split}
\end{equation}

We will compute this expression follow the same strategy as in Sections \ref{s:orthogonal-I} and \ref{s:orthogonal-II}. Let us denote for convenience
\[\tau=t_0t_3t^2,\]
let $R=t_0(t_3-t_4)(t_4+t_5+t_6-t)(t_3+t_5+t_6-t)$. Note that $R$ is the denominator of all summands in the RHS of (\ref{eq:deg6integral}), and that $\mathrm{Coeff}_{\tau}R=1$. Let $Q$ be the sum of the numerators in the RHS of (\ref{eq:deg6integral}). Then
the value of (\ref{eq:deg6integral}) is equal to $\mathrm{Coeff}_\tau Q$.

Because $t_0$ appears in $\tau$ but
does not appear in the summands which correspond to fixed points over $0$, these
do not contribute to $\mathrm{Coeff}_{\tau}Q$. Consequently, we only have to consider the
three summands corresponding to fixed points over $\infty$. Because $t$ does not appear
in the summand corresponding to $(\infty,e)$, only the last two summands are relevant.

We will concentrate first on
\[A=(t_3+t_5+t_6-t)c_3(\bar{t}_1,\bar{t}_2,\bar{t}_3,\overline{s_6(t_4)},\overline{s_6(t_5)},\overline{s_6(t_6)})\]
Since $\tau$ is neither divisible by $t_5$ nor by $t_6$, we have
\begin{multline*}
\Coeff_{\tau}A=\Coeff_{\tau}t_3c_3(\bar{t}_1,\bar{t}_2,\bar{t}_3,\overline{s_6(t_4)},\overline{s_6(t_5)},\overline{s_6(t_6)})
- \\ - \Coeff_{\tau}tc_3(\bar{t}_1,\bar{t}_2,\bar{t}_3,\overline{s_6(t_4)},\overline{s_6(t_5)},\overline{s_6(t_6)}).
\end{multline*}
Now we have
\begin{multline*}
\Coeff_{\tau}t_3c_3(\bar{t}_1,\bar{t}_2,\bar{t}_3,\overline{s_6(t_4)},\overline{s_6(t_5)},\overline{s_6(t_6)})
=\\=
\sum_{4\leq j<k\leq 6}\Coeff_{\tau}t_3(\bar{t}_1+\bar{t}_2+\bar{t}_3)\overline{s_6(t_j)}\overline{s_6(t_k)}+
\Coeff_{\tau}t_3\overline{s_6(t_4)}\overline{s_6(t_5)}\overline{s_6(t_6)}
%
\end{multline*}
Indeed, $\tau$ is divisible by $t^2$ while none of the elements $\bar{t}_1$, $\bar{t}_2$ and $\bar{t}_3$
contains the variable $t$, so all contributions to $\tau$ in the expansion of
$t_3c_3(\bar{t}_1,\bar{t}_2,\bar{t}_3,\overline{s_6(t_4)},\overline{s_6(t_5)},\overline{s_6(t_6)})$
come from terms which get either two or three entries from the terms
$\overline{s_6(t_4)},\overline{s_6(t_5)},\overline{s_6(t_6)}$ in the polynomial
$c_3(\dots)$.

By (\ref{eq:qaction-2}) we have
$t_3\overline{s_6(t_j)}\overline{s_6(t_k)}(\bar{t}_1+\bar{t}_2+\bar{t}_3)=
t_3\overline{s_6(t_j)}\overline{s_6(t_k)}({t}_1+t_2+t_3)$, which contains neither the
variables $t_0$ nor $t$, so we have
$$\Coeff_{\tau}t_3(\bar{t}_1+\bar{t}_2+\bar{t}_3)\overline{s_6(t_j)}\overline{s_6(t_k)}=0$$
for every $j,k$. Similar arguments give:
\begin{align*}
\Coeff_{\tau}& t_3\overline{s_6(t_4)}\overline{s_6(t_5)}\overline{s_6(t_6)} \\
&=\Coeff_{t_0t^2}\left(t-(t_5+t_6)+\frac{d}{3}t_0\right)\left(t-(t_4+t_6)-\frac{2d}{3}t_0\right)\left(t-(t_4+t_5)+\frac{d}{3}t_0\right) \\
&=\Coeff_{t_0t^2}\left(t+\frac{d}{3}t_0\right)\left(t-\frac{2d}{3}t_0\right)\left(t+\frac{d}{3}t_0\right)=0.
\end{align*}
In conclusion,
$$\Coeff_{\tau}t_3c_3(\bar{t}_1,\bar{t}_2,\bar{t}_3,\overline{s_6(t_4)},\overline{s_6(t_5)},\overline{s_6(t_6)})=0.$$

As for the second summand in $\Coeff_\tau A$, we have
\begin{align*}
\Coeff_{\tau}tc_3&(\bar{t}_1,\bar{t}_2,\bar{t}_3,\overline{s_6(t_4)},\overline{s_6(t_5)},\overline{s_6(t_6)}) =
\Coeff_{t_0t_3t}c_3(\bar{t}_1,\bar{t}_2,\bar{t}_3,\overline{s_6(t_4)},\overline{s_6(t_5)},\overline{s_6(t_6)})\\
&=
\Coeff_{t_0t}c_2(\bar{t}_1,\bar{t}_2,\overline{s_6(t_4)},\overline{s_6(t_5)},\overline{s_6(t_6)}) \\
&=
\Coeff_{t_0t}c_2\left(\frac{d}{3}t_0,-\frac{2d}{3}t_0,
t+\frac{d}{3}t_0,t-\frac{2d}{3}t_0,t+\frac{d}{3}t_0\right)=-d,
\end{align*}
where the second equality follows from observing that $\bar{t}_3=t_3+dt/3$, while the variable
$t_3$ does not appear in any of the other polynomials
$\bar{t}_1,\bar{t}_2,\overline{s_6(t_4)},\overline{s_6(t_5)},\overline{s_6(t_6)}$,
and the third equality follows from writing each $\overline{s_6(t_i)}$ as a polynomial in the $t_j$'s and $t$ combining (\ref{eq:weylaction2}) and (\ref{eq:qaction-2}),
and ignoring all variables different from $t_0$ and $t$.


In conclusion,
$$\Coeff_\tau A=d.$$

The other relevant summand in (\ref{eq:deg6integral}) is
$$B=-(t_4+t_5+t_6-t)e_3(\bar{t}_1,\bar{t}_2,\bar{t}_4,\overline{\sigma_3s_6(t_4)},
\overline{\sigma_3s_6(t_5)},\overline{\sigma_3s_6(t_6)}).$$
We have
$$\Coeff_\tau B=\Coeff_{t_0t_3t}c_3(\bar{t}_1,\bar{t}_2,\bar{t}_4,\overline{\sigma_3s_6(t_4)},
\overline{\sigma_3s_6(t_5)},\overline{\sigma_3s_6(t_6)}).$$
Among the variables $t_0,t_3,t$ only $t_0$ appears in $\bar{t}_1,\bar{t}_2,\bar{t}_4$,
and since $\Coeff_{t_0}(\bar{t}_1+\bar{t}_2+\bar{t}_4)=0$, we have
\begin{align*}
\Coeff & _{t_0t_3t} c_3(\bar{t}_1,\bar{t}_2,\bar{t}_4,\overline{\sigma_3s_6(t_4)},
\overline{\sigma_3s_6(t_5)},\overline{\sigma_3s_6(t_6)})=
\Coeff_{t_0t_3t}\overline{\sigma_3s_6(t_4)}\overline{\sigma_3s_6(t_5)}\overline{\sigma_3s_6(t_6)} \\
&=\Coeff_{t_0t_3t}
\left(t-(t_5+t_6)+\frac{d}{3}t_0\right)\left(t-(t_3+t_6)-\frac{2d}{3}t_0\right)\left(t-(t_3+t_5)+\frac{d}{3}t_0\right)\\
&=\Coeff_{t_0t_3t}
\left(t+\frac{d}{3}t_0\right)\left(t-t_3-\frac{2d}{3}t_0\right)\left(t-t_3+\frac{d}{3}t_0\right)=-\frac{d}{3}.
\end{align*}
Hence,
$$\Coeff_\tau B=-\frac{d}{3}.$$
Adding the two computations we deduce
$$\Coeff_\tau Q=\Coeff_\tau A+\Coeff_\tau B=\frac{2d}{3},$$
which is what we wanted to prove.

\subsection{Proof of Theorem \ref{thm:e6} for $E_7/Z$}

We start with a precise description of the quotient of coweight and coroot lattices, the weight lattice and the action of the Weyl group
 \cite[Pages 292--289]{OV}.

The roots and weights for the dual of the Lie algebra of a maximal torus $T$
inside $E_7$ can be expressed in a way very similar to the case of  $E_6$.
Let $\epsilon_1,\dots,\epsilon_8$ be a basis of a rational vector space and
consider the quotient space by the line spanned by
$\sum_{k=1}^8\epsilon_k$. Let us consider the inner product on the
quotient space induced by the degenerate pairing:
\[
 \langle \epsilon_i,\epsilon_j\rangle =\begin{cases}
                          7/8 &\mathrm{if}\,i=j\\
-1/8 &\mathrm{if}\ i\neq j.
                         \end{cases}
\]
Then $\tlie^*$ together with a multiple of the Killing form  can be identified
with the  quotient space and one can choose the following
set of simple roots:
\begin{equation}\label{eq:simplerootse7}
 \alpha_i=\epsilon_i-\epsilon_{i+1},\, i=1,\dots,6, \qquad \alpha_7=\epsilon_5+\epsilon_6+\epsilon_7+
 \epsilon_8.
\end{equation}
In terms of this identification the associated fundamental weights are
\begin{equation}\label{eq:fundweightse7}
 \zeta_i=\epsilon_1+\cdots+\epsilon_i+\mathrm{min}(i,8-i)\epsilon_8,\, i=1,\dots,6,\qquad
    \zeta_7=2\epsilon_8.
\end{equation}
As a generator of $\Lambda_w^\vee/\Lambda_r^\vee\cong \ZZ_2$ we select the following combination of coroots:
\begin{equation}\label{eq:fundgroupe7}
 z_0=\frac{1}{2}(h_{\alpha_1}+h_{\alpha_3}+h_{\alpha_7}).
\end{equation}
To study the action of the Weyl group we introduce the following basis of $\tlie^*$
(see \cite[(1.1)]{TW}):
\begin{equation}\label{eq:weylvariablese7}
\begin{cases}t_1&=\zeta_1=\epsilon_1+\epsilon_8,\\
 t_2&=-\zeta_1+\zeta_2=\epsilon_2+\epsilon_8,\\
t_3&=-\zeta_2+\zeta_3=\epsilon_3+\epsilon_8,\\
t_4&=-\zeta_3+\zeta_4=\epsilon_4+\epsilon_8\\
t_5&=-\zeta_4+\zeta_5+\zeta_7=\epsilon_5+\epsilon_8,\\
t_6&=-\zeta_5+\zeta_6+\zeta_7=\epsilon_6+\epsilon_8,\\
t_7&=-\zeta_6+\zeta_7=\epsilon_7+\epsilon_8,\\
t&=\zeta_7=2\epsilon_8.
\end{cases}
\end{equation}
As happened with the analogous definitions for $E_6$, this basis
is not a basis over the integers of the weight lattice. But again the basis
has the advantage of leading to a simple description of the action of the Weyl group.
Namely (see \cite[Table (1.3)]{TW}), for $1\leq i\leq 6$ the simple reflection
$\sigma_i$ exchanges $t_i$ and $t_{i+1}$ and fixes the other elements of
$\{t_1,\dots,t_7,t\}$, and
\begin{equation}\label{eq:weylactione7}
 s_7(t_j)=\begin{cases} t_{j}& \mathrm{if}\,\, j=1,2,3,4,\\
-(t_6+t_7-t)&\mathrm{if}\,j=5,\\
-(t_5+t_7-t)& \mathrm{if}\,j=6,\\
-(t_5+t_6-t)& \mathrm{if}\,j=7
               \end{cases}
\end{equation}
\[s_7(t)=2t-(t_5+t_6+t_7).\]

Let $\PP_i\subset \EE_6$ be the maximal parabolic subgroup associated to the subset  $\Pi\backslash \{\alpha_i\}$, let $z=dz_0\in \Lambda_w^\vee$
and let $\OOO=E_7/\PP_i$:

\underline{Case 1:  $\PP_1,\PP_3$ and $\PP_7$ (arguments using degree 2 cohomology).}
The degree 2 integral cohomology of $\OOO$ is generated by the
image of the weight $\zeta_i$. By (\ref{eq:fundgroupe7}) we have:
\begin{equation}\label{eq:weightevale7}
 \zeta_i(z_0)=\begin{cases}
   1/2& \mathrm{if}\,\,i=1,3,7\\
0& \mathrm{if}\,\,i=2,4,5,6.
                                         \end{cases}
\end{equation}
%
By (\ref{eq:weightevale7}) and the generalized Chevalley
formula applied to $[bX_e]\subset \OOO\cong\mathrm{E}_7/\PP_i$ we have:
\[\int_{[bX_e]}\kappa(\zeta_i)=\pm\frac{d}{2},\,\,  \mathrm{if}\,\,i=1,3,7,
\]
which proves Case 1 of the theorem.

\underline{Case 2:  $\PP_2$ and $P_4$ (arguments using degree 4 cohomology).}
Using (\ref{eq:fundgroupe7}) and (\ref{eq:weylvariablese7})
we deduce
\begin{equation}\label{eq:qactione7}
 t_i(z)=\begin{cases}\frac{d}{2}& \mathrm{if}\,\,i=1,3,5,6,7,\\
         -\frac{d}{2}& \mathrm{if}\,\,i=2,4.
                             \end{cases}
\end{equation}
and $t(z)=0$.

The residual Weyl group $W_2$ acts on the variables $t_1,t_2$ by permutations, so
the elementary symmetric polynomial  $f=c_2(t_1,t_2)=t_1t_2$ represents an
integral cohomology class of $\OOO\cong  E_7/\PP_2$.
We have:
\begin{align}\nonumber
 \int_{[bX_{\sigma_2}]}\kappa(f)& =-\int_{X_{\sigma_2}}\kappa(t_1)\cdot t_2(z)-\int_{bX_{e}}\kappa(t_1)\cdot t_2(h_{\alpha_2})=\\\nonumber
& =t_1(h_{\alpha_2})\cdot t_2(z)+t_1(z)\cdot t_2(h_{\alpha_2})=\\\nonumber &= 0\cdot \left(-\frac{d}{2}\right)+\frac{d}{2}\cdot 1=\frac{d}{2}.
\end{align}

The residual Weyl group $W_4$ acts on the variables $t_1,t_2,t_3,t_4$
by permutations, so
the elementary symmetric polynomial  $f=c_2(t_1,t_2,t_3,t_4)$ represents an
integral cohomology class of $\OOO=E_7/\PP_4$.
We have:
\begin{align}\nonumber
 \int_{[bX_{\sigma_4}]}\kappa(f)& =-\int_{X_{\sigma_4}}\kappa(t_1)\cdot (t_2(z)+t_3(z)+t_4(z))-\int_{X_{\sigma_4}}\kappa(t_2)\cdot (t_3(z)+t_4(z))\\
 \nonumber &\qquad -\int_{X_{\sigma_4}}\kappa(t_3)\cdot t_4(z)-\int_{bX_{e}}\kappa(t_1)\cdot (t_2(h_{\alpha_4})+t_3(h_{\alpha_4})+t_4(h_{\alpha_4}))\\
 \nonumber &\qquad -\int_{bX_{e}}\kappa(t_2)\cdot (t_3(h_{\alpha_4})+t_4(h_{\alpha_4}))-\int_{bX_{e}}\kappa(t_3)\cdot t_4(h_{\alpha_4})\\
 \nonumber &=t_1(h_{\alpha_4})\cdot \frac{d}{2}+t_2(h_{\alpha_4})\cdot 0-t_3(h_{\alpha_4})\cdot \frac{d}{2}+
 \frac{d}{2}\cdot (t_2(h_{\alpha_4})+t_3(h_{\alpha_4})+t_4(h_{\alpha_4}))+\\ \nonumber
 &\qquad
-\frac{d}{2}\cdot (t_3(h_{\alpha_4})+t_4(h_{\alpha_4}))-\frac{d}{2}\cdot t_4(h_{\alpha_4})\\\nonumber &
 =\frac{d}{2}(t_1(h_{\alpha_4})+t_2(h_{\alpha_4})-t_3(h_{\alpha_4}))-\frac{d}{2} t_4(h_{\alpha_4})=0-\frac{d}{2}=-\frac{d}{2},
\end{align}
and this proves Case 2.

\underline{Case 3:  $\PP_5$ (computed assisted calculations I).}

The following polynomial
 \[
\begin{split}
f &=\zeta_1^2\zeta_2\zeta_3-\zeta_1\zeta_2^2\zeta_3-\zeta_1^2\zeta_3^2+\zeta_1\zeta_2\zeta_3^2+\zeta_1^2\zeta_3\zeta_4-
\zeta_1\zeta_2\zeta_3\zeta_4+\zeta_2^2\zeta_3\zeta_4-\zeta_2\zeta_3^2\zeta_4-\zeta_1^2\zeta_4^2+\\
&+\zeta_1\zeta_2\zeta_4^2-\zeta_2^2\zeta_4^2+\zeta_2\zeta_3\zeta_4^2-2\zeta_1^2\zeta_2\zeta_5+2\zeta_1\zeta_2^2\zeta_5
-2\zeta_2^2\zeta_3\zeta_5+2\zeta_2\zeta_3^2\zeta_5+\zeta_1^2\zeta_4\zeta_5-\\
&-\zeta_1\zeta_2\zeta_4\zeta_5+\zeta_2^2\zeta_4\zeta_5
-\zeta_2\zeta_3\zeta_4\zeta_5-\zeta_3^2\zeta_4\zeta_5+\zeta_3\zeta_4^2\zeta_5+
\zeta_1^2\zeta_5^2-\zeta_1\zeta_2\zeta_5^2+\zeta_2^2\zeta_5^2-\\
&+\zeta_2\zeta_3\zeta_5^2+
\zeta_3^2\zeta_5^2
-\zeta_3\zeta_4\zeta_5^2+\zeta_1^2\zeta_4\zeta_7-
\zeta_1\zeta_2\zeta_4\zeta_7+\zeta_2^2\zeta_4\zeta_7-\zeta_2\zeta_3\zeta_4\zeta_7+\zeta_3^2\zeta_4\zeta_7-\\
&-\zeta_3\zeta_4^2\zeta_7-\zeta_1^2\zeta_7^2+\zeta_1\zeta_2\zeta_7^2-\zeta_2^2\zeta_7^2+\zeta_2\zeta_3\zeta_7^2-\zeta_3^2\zeta_7^2+\zeta_3\zeta_4\zeta_7^2.
\end{split}
\]
is  $W_5$-invariant. As it has integral coefficients in the weight variables, it represents an integral
cohomology class in $Y_{z,\OOO}$ for $\OOO\cong E_7/\PP_5$. Furthermore,
%
\begin{equation}\label{eq:integralw5}
 \int_{bX_{\sigma_4s_6\sigma_5}}\kappa(f)=-\frac{d}{2},
\end{equation}
and this proves Case 3 of the theorem.

In contrast to the previous cases, to find the invariant polynomial $f$
we relied on computer calculations using the procedure {\tt InvE7} for {\it Singular}\footnote{Our code is typed for Singular-4.1.0 which 
is the latest distribution for Linux. It
is not compatible with some of the earlier distributions such as the one for Cygwin.}
(the reader will find the procedure {\tt InvE7} in the the {\it ancillary file} {\tt InvE7.txt}
which can be downloaded from the {\tt arXiv} page for the present paper; this applies
to the other {\it Singular} procedures which we will refer to below). Similarly, to check
equality (\ref{eq:integralw5}) we use the procedure {\tt GChevE7} located in the file
{\tt GChevE7.txt}. We next explain how to use these procedures to obtain the results above,
and afterwards we briefly describe the algorithms on which the procedures are based.

To use the procedures, enter at the Linux prompt line:
$${\tt singular\,\,InvE7.txt\,\,GChevE7.txt}$$

The procedure {\tt InvE7} computes polynomials  on the weights of E7, of a given degree, which are invariant under the residual
Weyl group for a choice of maximal parabolic subgroup.
It requires two integer arguments: the first one is the degree of the invariant polynomials and the second one, which refers to a vertex in the Dynkin diagram of E7 following the labelling in  \cite[Page 293]{OV}, specifies the parabolic subgroup for which the invariant polynomials are
required. The output are polynomials in the variables {\tt a,b,c,d,e,f,g} which correspond to the
fundamental weights $\zeta_1,\zeta_2,\zeta_3,\zeta_4,\zeta_5,\zeta_6,\zeta_7$ respectively.
Typing
$${\tt >\,\,InvE7(4,5);}$$
gives as output 8 polynomials which are homogeneous of degree 4 and $W_5$-invariant.
Our polynomial $f$ is the fourth one in the list.

The procedure {\tt GChevE7} implements the generalized Chevalley formula (\ref{eq:cap}) for the
fibration $Y_{z_0,\OOO}$, which corresponds to $d=1$.
It requires two arguments: a homogeneous polynomial of arbitrary degree $\delta$
in the variables {\tt a,...,g}, which now correspond to the
variables $t_1,\dots,t_7$ respectively, and a list of $\delta-1$ integers, each belonging
to the set $\{1,\dots,7\}$, referring from left to right to a word in the simple reflections $\sigma_1,\dots,\sigma_6,s_7$
(see formula (\ref{eq:weylactione7}) and the lines above it) specifying a cell in the generalized Bruhat
decomposition; such word is required to satisfy the properties given just before Lemma \ref{lem:tits}.

To reproduce our computations type:
$${\tt >\,\,poly\,\,q=InvE7(4,5)[4];}$$
which defines {\tt q} to be the fourth of the invariant polynomials given by
{\tt InvE7(4,5)}. As explained above, the variables {\tt a,...,g} on which {\tt q} is given
correspond to the weight variables
$\zeta_1,\dots,\zeta_7$. Hence, to apply {\tt GChevE7} to {\tt q} we need first to change the variables
from $\zeta_1,\dots,\zeta_7$ to $t_1,\dots,t_7$. For that, type:
$${\tt >\,\,poly\,\,p=cvinv(q);}$$
(the procedure {\tt cvinv} is also defined in the file {\tt InvE7.txt}).
Next type
\begin{align*}
&{\tt >\,\,list\,\,l=(4,6,5);}\\
&{\tt >\,\,GChevE7(p,l);}
\end{align*}
The output is $-\tfrac{1}{2}$, which proves (\ref{eq:integralw5})
(recall that both sides of the equality (\ref{eq:integralw5}) are
linear on $d$ and $z_0$ corresponds to a generator of $\Lambda_w^{\vee}/\Lambda_r^{\vee}$).

The above procedures are based partly on existing {\it Singular}
procedures and partly on some standard algorithms which we next describe.

{\tt InvE7} is based on the {\tt invariant\underline{ }basis} procedure in
{\it Singular}, to which we supply the concrete matrices corresponding to the simple reflections $\sigma_1,\dots,\sigma_6,s_7$; these matrices are contained in the code in {\tt InvE7.txt}.
The file {\tt InvE7.txt} also contains the description of the command {\tt cvinv} that
given a polynomial in the weight basis $\zeta_1,\dots,\zeta_7$ returns the same polynomial
in the variables to $t_1,\dots,t_7$.



{\tt GChevE7} implements in \emph{Singular} an algorithm based on the generalized Chevalley formula
to compute the cap product
\begin{equation}
 \label{eq:cap-chev}
 \kappa(f)\cap[bX_w]
\end{equation}
on the bundle $Y_{z,\OOO}\to \PP^1$
for
\begin{itemize}
 \item  $z=z_0\in \tlie\subset \mathfrak{e}_7$,
 \item any word $w\in W$ in the simple reflections $\sigma_1,\cdots,\sigma_6,s_7$ of length $\delta-1$ satisfying the properties given just before Lemma \ref{lem:tits},
 \item and any homogeneous polynomial $f\in \QQ[\tlie]$ of degree $\delta$.
\end{itemize}
Recall the generalized Chevalley formula:
\begin{equation}
\label{eq:GC-1}
\kappa(\zeta)\cap [bX_w]=-(w\cdot \zeta)(z)[X_w]-
\sum_{w'\overset{s}{\rightarrow}w,\,w'\in W^P} (w'\cdot \zeta)(h_{s})[bX_{w'}].
\end{equation}
Since (\ref{eq:cap-chev}) depends linearly on $f$, to compute
(\ref{eq:cap-chev}) it suffices to consider the case in which $f$ is a monomial.
The associativity formula for the cup/cap products gives
$$\kappa(f_1\zeta)\cap[bX_w]=\kappa(f_1)\cap(\kappa(\zeta)\cap [bX_w]),$$
and hence to compute  (\ref{eq:cap-chev}) for a monomial $f$ we may apply
recursively the Chevalley formula to compute  (\ref{eq:cap-chev}) in $\delta$ steps.
If on each application of  (\ref{eq:GC-1}) we select just one summand
on its RHS,
then  (\ref{eq:cap-chev}) can be written as a sum of contributions labeled
by a segment of length $\delta-1$ in the
Hasse diagram of $w$, together with an integer in $\{1,\dots,\delta\}$ which marks
the step on which we cap a horizontal Schubert cell
to a vertical Schubert cell (this corresponding to picking the first summand in the RHS of (\ref{eq:GC-1})).
Because  $w$ satisfies the properties given just before Lemma \ref{lem:tits}, it follows that
there are $(\delta-1)!$ such segments.

The computation in {\tt GChevE7} is done by the {\tt capspl} procedure.
First, there is a loop summing  segments of length $\delta-1$ in the Hasse diagram of the word $w$; each segment
is  described by a list $L_{(1)}$ of integers marking  the position from left to right of
the letter to be
erased at each step.
The contribution of each segment is computed recursively in $\delta$ steps.
To explain this, let us assume for simplicity that we pass from a horizontal to a vertical cell in the last step. In the
$l$-th step the inputs are a degree $\delta-l+1$ polynomial $f_{(l)}$ and a list  of $\delta-l$ integers $L_{(l)}$.
By means of standard {\tt Singular}
procedure the polynomial is decomposed  as
\[f_{(l)}=\sum_{j=1}^7 f_jt_j,\]
and the auxiliary procedure {\tt ev} is called to compute $(w' \cdot t_j)(h_{s})$, where $w'$ and $s$ are
also calculated out of $L_{(l)}$ ({\tt ev} itself calls the auxiliary procedure {\tt word} for
the explicit calculation of the transformation on $\tlie^*$ defined by $w'$).
The outputs of the $l$-th step are the polynomial
\[f_{(l-1)}=\sum_{j=1}^7 (w'\cdot t_j)(h_{s})f_j\]
and the list $L_{(l-1)}$ obtained by erasing first integer in $L_{(l)}$.
In the last step we start with a degree one polynomial $f_{(1)}=\sum_{j=1}^7 a_jt_j$ and the auxiliary procedure {\tt ev}
is called to compute $f_{(1)}(z_0)$.

Analogous ideas are applied when the
contraction from horizontal to vertical cell does not happen in the last step.

\underline{Case 4:  $\PP_6$ (computed assisted calculations II).}

The polynomial

\newpage

%

\[\begin{split}
 f&= \tfrac{3}{8}t_1^4-\tfrac{3}{32}t_1^3t_2+\tfrac{1}{4}t_1^2t_2^2-\tfrac{3}{32}t_1^3t_3+\tfrac{37}{128}t_1t_2^2t_3-\tfrac{5}{32}t_2^3t_3
 +\tfrac{1}{4}t_1^2t_3^2+\tfrac{37}{128}t_1t_2t_3^2-\tfrac{11}{128}t_2^2t_3^2-                     \\
 &-\tfrac{5}{32}t_2t_3^3+\tfrac{17}{32}t_1^3t_4+\tfrac{7}{8}t_1t_2^2t_4+\tfrac{17}{128}t_1t_2t_3t_4+\tfrac{37}{128}t_2^2t_3t_4+
\tfrac{7}{8}t_1t_3^2t_4+\tfrac{37}{128}t_2t_3^2t_4+\tfrac{15}{128}t_1t_2t_4^2-                      \\
& -\tfrac{7}{32}t_2^2t_4^2+\tfrac{15}{128}t_1t_3t_4^2+\tfrac{15}{128}t_2t_3t_4^2-\tfrac{7}{32}t_3^2t_4^2+\tfrac{83}{128}t_1t_4^3+
 \tfrac{3}{128}t_2t_4^3+\tfrac{3}{128}t_3t_4^3-\tfrac{3}{32}t_4^4+  \tfrac{193}{512}t_1t_2^2t_5-                      \\
 & -\tfrac{17}{128}t_2^3t_5+  +\tfrac{85}{256}t_1t_2t_3t_5+\tfrac{97}{512}t_2^2t_3t_5+ \tfrac{193}{512}t_1t_3^2t_5+\tfrac{97}{512}t_2t_3^2t_5
 -\tfrac{17}{128}t_3^3t_5 +\tfrac{25}{256}t_1t_2t_4t_5+                             \\
&+\tfrac{157}{512}t_2^2t_4t_5+\tfrac{25}{256}t_1t_3t_4t_5+\tfrac{179}{512}t_2t_3t_4t_5+
 \tfrac{157}{512}t_3^2t_4t_5+ \tfrac{105}{512}t_1t_4^2t_5+\tfrac{69}{512}t_2t_4^2t_5+ \tfrac{69}{512}t_3t_4^2t_5+         \\
 &+ \tfrac{3}{64}t_4^3t_5+\tfrac{81}{256}t_1t_2t_5^2 -\tfrac{163}{512}t_2^2t_5^2+\tfrac{81}{256}t_1t_3t_5^2+\tfrac{111}{512}t_2t_3t_5^2-
 \tfrac{163}{512}t_3^2t_5^2+  \tfrac{231}{256}t_1t_4t_5^2+    \tfrac{171}{512}t_2t_4t_5^2                                        \\
&+\tfrac{171}{512}t_3t_4t_5^2-\tfrac{231}{512}t_4^2t_5^2+\tfrac{77}{512}t_1t_5^3-\tfrac{39}{512}t_2t_5^3-
\tfrac{39}{512}t_3t_5^3+  \tfrac{41}{512}t_4t_5^3-\tfrac{29}{128}t_5^4
 +\tfrac{17}{32}t_1^3t_6+\tfrac{7}{8}t_1t_2^2t_6+                                                                     \\
 &+\tfrac{17}{128}t_1t_2t_3t_6+\tfrac{37}{128}t_2t_3^2t_6+\tfrac{7}{8}t_1t_3^2t_6+\tfrac{37}{128}t_2^2t_3t_6-
 \tfrac{13}{32}t_1^2t_4t_6+\tfrac{45}{128}t_2t_3t_4t_6+\tfrac{103}{128}t_1t_4^2t_6+                                    \\
 & +\tfrac{43}{128}t_2t_4^2t_6+ \tfrac{43}{128}t_3t_4^2t_6-\tfrac{29}{128}t_4^3t_6+\tfrac{25}{256}t_1t_2t_5t_6+
\tfrac{157}{512}t_2^2t_5t_6+\tfrac{25}{256}t_1t_3t_5t_6+\tfrac{179}{512}t_2t_3t_5t_6+                      \\
& +\tfrac{157}{512}t_3^2t_5t_6-\tfrac{9}{256}t_1t_4t_5t_6+\tfrac{171}{512}t_2t_4t_5t_6+\tfrac{171}{512}t_3t_4t_5t_6+
\tfrac{181}{512}t_4^2t_5t_6+\tfrac{231}{256}t_1t_5^2t_6+\tfrac{171}{512}t_2t_5^2t_6+                               \\
&+ \tfrac{171}{512}t_3t_5^2t_6+\tfrac{23}{512}t_4t_5^2t_6+\tfrac{41}{512}t_5^3t_6+\tfrac{15}{128}t_1t_2t_6^2-
 \tfrac{7}{32}t_2^2t_6^2+\tfrac{15}{128}t_1t_3t_6^2+\tfrac{15}{128}t_2t_3t_6^2-\tfrac{7}{32}t_3^2t_6^2+            \\
 &+ \tfrac{103}{128}t_1t_4t_6^2+ \tfrac{43}{128}t_2t_4t_6^2+\tfrac{43}{128}t_3t_4t_6^2-\tfrac{1}{4}t_4^2t_6^2+
 \tfrac{105}{512}t_1t_5t_6^2+\tfrac{69}{512}t_2t_5t_6^2+ \tfrac{69}{512}t_3t_5t_6^2+\tfrac{181}{512}t_4t_5t_6^2-     \\
 &-\tfrac{231}{512}t_5^2t_6^2+\tfrac{83}{128}t_1t_6^3+\tfrac{3}{128}t_2t_6^3+\tfrac{3}{128}t_3t_6^3-\tfrac{29}{128}t_4t_6^3+
 \tfrac{3}{64}t_5t_6^3-\tfrac{3}{32}t_6^4-\tfrac{5}{3}t_1^3t_7-2t_1^2t_2t_7-                                       \\
&-\tfrac{831}{512}t_1t_2^2t_7-\tfrac{691}{384}t_2^3t_7+\tfrac{85}{256}t_1t_2t_3t_7+\tfrac{97}{512}t_2^2t_3t_7-\tfrac{1087}{512}t_1t_3^2t_7-
 \tfrac{1183}{512}t_2t_3^2t_7+ \tfrac{77}{384}t_3^3t_7+                                                           \\
 &+\tfrac{25}{256}t_1t_2t_4t_7+\tfrac{157}{512}t_2^2t_4t_7-\tfrac{103}{256}t_1t_3t_4t_7- \tfrac{77}{512}t_2t_3t_4t_7-
 \tfrac{99}{512}t_3^2t_4t_7-\tfrac{1175}{512}t_1t_4^2t_7-\tfrac{1211}{512}t_2t_4^2t_7-                              \\
 &-\tfrac{187}{512}t_3t_4^2t_7+\tfrac{73}{192}t_4^3t_7-2t_1^2t_5t_7+\tfrac{51}{64}t_1t_2t_5t_7-\tfrac{459}{256}t_2^2t_5t_7+
 \tfrac{19}{64}t_1t_3t_5t_7+\tfrac{115}{256}t_2t_3t_5t_7-                                                                 \\
 & - \tfrac{587}{256}t_3^2t_5t_7+ \tfrac{1}{16}t_1t_4t_5t_7+ \tfrac{85}{256}t_2t_4t_5t_7-\tfrac{43}{256}t_3t_4t_5t_7- \tfrac{601}{256}t_4^2t_5t_7-
 \tfrac{817}{512}t_1t_5^2t_7-\tfrac{113}{64}t_2t_5^2t_7+                                                                  \\
& +\tfrac{15}{64}t_3t_5^2t_7+\tfrac{45}{128}t_4t_5^2t_7-\tfrac{2641}{1536}t_5^3t_7-2t_1^2t_6t_7+\tfrac{153}{256}t_1t_2t_6t_7-
\tfrac{867}{512}t_2^2t_6t_7+\tfrac{25}{256}t_1t_3t_6t_7+                                                                    \\
&-\tfrac{179}{512}t_2t_3t_6t_7-\tfrac{1123}{512}t_3^2t_6t_7-\tfrac{9}{256}t_1t_4t_6t_7+\tfrac{171}{512}t_2t_4t_6t_7-
\tfrac{85}{512}t_3t_4t_6t_7-\tfrac{1099}{512}t_4^2t_6t_7+    \tfrac{9}{16}t_1t_5t_6t_7  +                                \\
&+\tfrac{213}{256}t_2t_5t_6t_7+\tfrac{85}{256}t_3t_5t_6t_7+\tfrac{81}{256}t_4t_5t_6t_7-
\tfrac{211}{128}t_5^2t_6t_7-\tfrac{919}{512}t_1t_6^2t_7-\tfrac{955}{512}t_2t_6^2t_7 +\tfrac{69}{512}t_3t_6^2t_7+         \\
&+\tfrac{181}{512}t_4t_6^2t_7-\tfrac{473}{256}t_5t_6^2t_7-\tfrac{311}{192}t_6^3t_7+\tfrac{465}{256}t_1t_2t_7^2-\tfrac{163}{512}t_2^2t_7^2+
 \tfrac{337}{256}t_1t_3t_7^2+     \tfrac{623}{512}t_2t_3t_7^2-\tfrac{419}{512}t_3^2t_7^2                                    \\
& +\tfrac{487}{256}t_1t_4t_7^2+\tfrac{683}{512}t_2t_4t_7^2+\tfrac{427}{512}t_3t_4t_7^2-
 \tfrac{487}{512}t_4^2t_7^2+\tfrac{975}{512}t_1t_5t_7^2+\tfrac{111}{64}t_2t_5t_7^2+\tfrac{79}{64}t_3t_5t_7^2+\tfrac{173}{128}t_4t_5t_7^2-\\
&- \tfrac{141}{256}t_5^2t_7^2+ \tfrac{615}{256}t_1t_6t_7^2+\tfrac{939}{512}t_2t_6t_7^2+\tfrac{683}{512}t_3t_6t_7^2+
\tfrac{535}{512}t_4t_6t_7^2+\tfrac{237}{128}t_5t_6t_7^2-\tfrac{231}{512}t_6^2t_7^2-\tfrac{947}{512}t_1t_7^3 -                               \\
&-\tfrac{1063}{512}t_2t_7^3-\tfrac{39}{512}t_3t_7^3+\tfrac{41}{512}t_4t_7^3-
 \tfrac{1051}{512}t_5t_7^3-\tfrac{983}{512}t_6t_7^3-\tfrac{343}{384}t_7^4.
 \end{split}
 \]
is $W_6$-invariant modulo the ideal $I^+$ generated by homogeneous $W$-invariant polynomials of strictly positive degree.
If we change from $t_1,\cdots,t_7$ variables to weight variables, then $f$ corresponds to a polynomial
with non-integral coefficients. However, $f$ represents an integral class in the cohomology of $\OOO=E_7/\PP_6$. Furthermore,

\begin{equation}\label{eq:integralw6}
 \int_{[bX_{\sigma_4\sigma_5s_6}]}\kappa(f)=\frac{3d}{2},
\end{equation}
and this proves Case 4 of the Theorem.

To find $f$ and to check (\ref{eq:integralw6}) we used the \emph{Singular} procedures {\tt GChevE7}
and {\tt ZcohE7}, respectively.
The procedure  {\tt ZcohE7} --- located in the file {\tt ZcohE7.txt} --- computes for a given degree a basis of homogeneous polynomials
invariant under the action of the residual Weyl group $W_6$, modulo the ideal $I^+$ generated by  $W$-invariant polynomials
vanishing at the origin. We were led to use it because unlike the case of the other minimal coadjoint orbits,
our calculations of the integrals $\int_{[bX_w]}\kappa(g)$ for any word $w$
satisfying the properties given just before Lemma \ref{lem:tits},
 and any polynomial $g$ in weight variables $W_6$-invariant and with integral coefficients, produced just integer numbers.

To use the procedures, enter at the Linux (or {\tt Cygwin}, if working with Windows) prompt line:
$${\tt singular\,\,ZcohE7.txt\,\,GChevE7.txt}$$
The input of {\tt ZcohE7} is an integer $\delta\in \{1,\cdots,5\}$ and the output
is a basis of degree $\delta$ homogeneous polynomials in the variables {\tt a,b,c,d,e,f,g} --- which correspond to the variables
$t_1,\cdots t_7$ --- invariant under the action of the residual Weyl group $W_6$ modulo the ideal $I^+$. Typing
$${\tt >\,\ ZcohE7(4);}$$
gives as output two polynomials and
our $f$ is the second one in the list.

To reproduce our computations type:
\begin{align*}
&{\tt >\,\,poly\,\,p=ZcohE7(4)[2];}\\
&{\tt >\,\,list\,\,l=(4,5,6);}\\
&{\tt >\,\,GChevE7(p,l);}
\end{align*}
The output is $\tfrac{3}{2}$, which proves (\ref{eq:integralw6})
(recall that both sides of the equality (\ref{eq:integralw6}) are
linear on $d$ and $z_0$ corresponds to a generator of $\Lambda_w^{\vee}/\Lambda_r^{\vee}$).

In short, the procedure  {\tt ZcohE7} first implements a result of Nakagawa \cite[Lemma 5.1]{N1} which
shows that the integral cohomology ring of the regular orbit of $\mathfrak{e}_7^*$
is the image of the quotient weighted polynomial ring
\begin{equation}\label{eq:zcohe7}
 \ZZ[t_1,\dots,t_7,\varpi_1,\varpi_3,\varpi_4,\varpi_{5},\varpi_9]/I\to \QQ[t_1,\dots,t_7]/I^+,
 \end{equation}
where $\varpi_i$ are degree $i$ variables, $I$ is a given ideal of relations, and the morphism is uniquely determined
by the generators of $I$. Next, it computes
the subring of $W_6$-invariant elements of the LHS of (\ref{eq:zcohe7}); to do that, one does not work in
the quotient ring, but rather, in a $\ZZ$-module complementary
to the ideal $I$. Finally, the resulting $W_6$-invariant polynomials are transferred to the RHS of
(\ref{eq:zcohe7}).

As {\tt ZcohE7.txt} contains more than one thousand lines of code, we now give a more precise
description of how it is structured.

{\tt ZcohE7.txt} is divided in the following six blocks:
\begin{enumerate}
 \item The Smith normal form for a matrix with integer coefficients is implemented in the
 procedures {\tt smithunitcol} and {\tt smithcopri}. The Smith normal
 form is based on elimination by locating units (more generally, coprimes) in rows, operations
 that are basic {\it Singular}  procedures.
 \item A $\ZZ$-module $C$ of weighted homogeneous polynomials of a given degree complementary
 to the ideal of relations in the LHS of (\ref{eq:zcohe7})
 is computed by the procedure {\tt compzmod}.
 For that we  supply the description of the ideal of relations $I$ and describe weighted polynomials
 by their integer valued vector of coefficients; this is done using {\it Singular} procedures to handle weighted ideals
 and working on a base ring with integer coefficients.
 Then, we apply the  Smith normal form procedures {\tt smithunitcol} and {\tt smithcopri} in block (i) to the
 vectors of coefficients to obtain the complementary $\ZZ$-module $C$.
 \item The image of $C$ in $\QQ[t_1,\cdots,t_7]/I^+$ is calculated by the procedure  {\tt qimbed}.
 In this block we work with rational coefficients. It
 is worth noting that in {\tt ZcohE7.txt} we work with four different base rings.
 \item The action of the residual Weyl group $W_6$ on the LHS of (\ref{eq:zcohe7})
 is computed and stored in the maps $s(1),\cdots,s(5),s6$.
 By looking at the ideal of relations it is possible to extend the action
 of each simple reflection to the new weighted variables inductively on the degree.
 The complicated extension is the one of the simple reflection $s_7$ (denoted by $s6$ in the code).
By calling the procedure  {\tt compzmod} in block (ii) we arrange that $s6(\varpi_i)\subset C$.
\item A basis of $W_6$-invariant subspace of $C\otimes \QQ$ of a given degree is computed by the procedure {\tt icoh}.
The procedure {\tt compzmod} in block (ii) is used to push
the maps associated to simple reflections computed in the previous block to maps on $C$.
The $W_6$-invariant subspace is computed as a common kernel by subtracting the identity map from the simple reflections.
 \item The image in the RHS of (\ref{eq:zcohe7}) of a basis of quasi-homogeneous $W_6$-invariant polynomials of $C$
 is computed by the main procedure {\tt ZcohE7}. Denominators in the $W_6$-invariant subspace of $C\otimes \QQ$ computed in the previous block
 are cleared. The procedure  {\tt qembed} in block (iii) is called to map these polynomials
 to polynomials in the original ring $\QQ[t_1,\cdots,t_7]$.

\end{enumerate}

\end{document}